\documentclass[11pt, draft]{amsart}
\usepackage{amssymb, amstext, amscd, amsmath, amssymb}
\usepackage{mathtools, xypic, paralist, color, dsfont, rotating}
\usepackage{mathabx}
\usepackage{verbatim}
\usepackage{enumerate,enumitem}
\usepackage{float}
\usepackage{setspace}

\usepackage{tikz}

\floatstyle{boxed} 
\restylefloat{figure}

\numberwithin{equation}{section}
\usepackage[a4paper]{geometry}
\usepackage{quoting}
\quotingsetup{vskip=.1in}
\quotingsetup{leftmargin=.17in}
\quotingsetup{rightmargin=.17in}
\let\OLDthebibliography\thebibliography
\renewcommand\thebibliography[1]{
  \OLDthebibliography{#1}
  \setlength{\parskip}{0pt}
  \setlength{\itemsep}{2pt plus 0.5ex}
}

\makeatletter
\def\@cite#1#2{{\m@th\upshape\bfseries%
[{#1\if@tempswa{\m@th\upshape\mdseries, #2}\fi}]}}
\makeatother
\theoremstyle{plain}
\newtheorem{theorem}{Theorem}[section]
\newtheorem{corollary}[theorem]{Corollary}
\newtheorem{proposition}[theorem]{Proposition}
\newtheorem{lemma}[theorem]{Lemma}
\theoremstyle{definition}
\newtheorem{definition}[theorem]{Definition}
\newtheorem{example}[theorem]{Example}

\newtheorem{remark}[theorem]{Remark}

\newtheorem*{acknow}{Acknowledgements}

\theoremstyle{remark}
\mathtoolsset{centercolon}
  \newcommand{\A}{{\mathcal{A}}}
  \newcommand{\B}{{\mathcal{B}}}
  \newcommand{\C}{{\mathcal{C}}}

  \newcommand{\F}{{\mathcal{F}}}
  \newcommand{\G}{{\mathcal{G}}}
\renewcommand{\H}{{\mathcal{H}}}
  \newcommand{\I}{{\mathcal{I}}}
  
  \newcommand{\K}{{\mathcal{K}}}
\renewcommand{\L}{{\mathcal{L}}}
  
  \newcommand{\N}{{\mathcal{N}}}
\renewcommand{\O}{{\mathcal{O}}}
\renewcommand{\P}{{\mathcal{P}}}

\renewcommand{\S}{{\mathcal{S}}}
  \newcommand{\T}{{\mathcal{T}}}

  \newcommand{\X}{{\mathcal{X}}}
  \newcommand{\Y}{{\mathcal{Y}}}
  
\newcommand{\eps}{\varepsilon}
\def\al{\alpha}
\def\be{\beta}

\def\ga{\gamma}
\def\De{\Delta}
\def\de{\delta}

\def\la{\lambda}

\def\om{\omega}

\newcommand\vphi{\varphi}
\newcommand{\bC}{\mathbb{C}}

\newcommand{\bF}{\mathbb{F}}
\newcommand{\bK}{\mathbb{K}}
\newcommand{\bN}{\mathbb{N}}

\newcommand{\bZ}{\mathbb{Z}}
\newcommand{\bR}{\mathbb{R}}
\newcommand{\foral}{\text{ for all }}
\newcommand{\qand}{\quad\text{and}\quad}

\newcommand{\qiff}{\quad\text{if and only if}\quad}
\newcommand{\qfor}{\quad\text{for}\quad}
\newcommand{\qforal}{\quad\text{for all}\ }

\newcommand{\ca}{\mathrm{C}^*}

\newcommand{\cenv}{\mathrm{C}^*_{\textup{env}}}

\newcommand{\ol}{\overline}
\newcommand{\wt}{\widetilde}
\newcommand{\wh}{\widehat}

\newcommand{\Aut}{\operatorname{Aut}}
\newcommand{\alg}{\operatorname{alg}}

\newcommand{\env}{\operatorname{\textup{env}}}

\newcommand{\id}{{\operatorname{id}}}

\newcommand{\mt}{\emptyset}

\newcommand{\scv}{\operatorname{sc}}
\newcommand{\spn}{\operatorname{span}}

\newcommand{\sumoplus}{\operatornamewithlimits{\sum \strut^\oplus}}

\newcommand{\sca}[1]{\left\langle#1\right\rangle} %\sca{a,b} =<a,b>
\newcommand{\nor}[1]{\left\Vert #1\right\Vert} %\nor{x}=||x||
\newcommand{\quo}[2]{{\raisebox{.1em}{$#1$}\left/ \, \raisebox{-.1em}{$#2$}\right.}} % quotient of objects
\begin{document}
%%%%%%%%%%%%%%%%%%%%%%%%%%%%%%%%

%%%%%%%%%%%%%%%%%%%%%%%%%%%%%%%%
\title[C*-algebras of product systems and controlled maps]{Co-universality and controlled maps on product systems over right LCM-semigroups}

%should we add group-embeddable in the title?

\author[E.T.A. Kakariadis]{Evgenios T.A. Kakariadis}
\address{School of Mathematics, Statistics and Physics\\ Newcastle University\\ Newcastle upon Tyne\\ NE1 7RU\\ UK}
\email{evgenios.kakariadis@newcastle.ac.uk}

\author[E.G. Katsoulis]{Elias~G.~Katsoulis}
\address{Department of Mathematics\\ East Carolina University\\ Greenville\\ NC 27858\\USA}
\email{katsoulise@ecu.edu}

\author[M. Laca]{Marcelo Laca}
\address{Department of Mathematics and Statistics\\ University of Victoria\\ Victoria\\ BC\\ Canada}
\email{laca@uvic.ca}

\author[X. Li]{Xin Li}
\address{School of Mathematics and Statistics\\ University of Glasgow\\ University Place\\ Glasgow\\ G12 8QQ\\ UK}
\email{xin.li@glasgow.ac.uk}



%%%%%%%%%%%%%%%%%%%%%%%%%%%%%%%%
\begin{abstract}
We study the structure of C*-algebras associated with compactly aligned product systems over group embeddable right LCM-semigroups.
Towards this end we employ controlled maps and a controlled elimination method that associates the original cores to those of the controlling pair, and we combine with applications of the C*-envelope theory for cosystems of nonselfadjoint operator algebras recently produced. 
We derive several applications of these methods that generalize results on single C*-correspondences.

First we show that if the controlling group is exact then the co-universal C*-algebra of the product system coincides with the quotient of the Fock C*-algebra by the ideal of strong covariance relations.
We show that if the controlling group is amenable then the product system is amenable.
In particular if the controlling group is abelian then the co-universal C*-algebra is the C*-envelope of the tensor algebra.

Secondly we give necessary and sufficient conditions for the Fock C*-algebra to be nuclear and exact.
When the controlling group is amenable we completely characterize nuclearity and exactness of any equivariant injective Nica-covariant representation of the product system. 

Thirdly we consider controlled maps that enjoy a saturation property. 
In this case we induce a compactly aligned product system over the controlling pair that shares the same Fock representation, and preserves injectivity.
By using co-universality, we show that they share the same reduced covariance algebras.
If in addition the controlling pair is a total order then the fixed point algebra of the controlling group induces a super-product system that has the same reduced covariance algebra and is moreover reversible.
\end{abstract}

\maketitle

%\tableofcontents

%%%%%%%%%%%%%%%%%%%%%%%%%%%%%%%%
\section{Introduction} \label{S:sgrp alg} \addtocontents{toc}{\protect\setcounter{tocdepth}{1}}
%%%%%%%%%%%%%%%%%%%%%%%%%%%%%%%%

%%%%%%%%%%%%%%%%%%%%%%%%%%%%%%%%
\subsection{Framework}
%%%%%%%%%%%%%%%%%%%%%%%%%%%%%%%%

In the present project we study further the effect of nonselfadjoint operator algebras and boundary theory of group coactions on the C*-algebras theory recently initiated in \cite{DKKLL20}.
We work in the class of algebras of a compactly aligned product system $X$ over a right LCM-semigroup $P$ in a group $G$ with coefficients in a C*-algebra $A$ (for brevity we will say that such a pair $(G,P)$ is a \emph{weak right LCM-inclusion}).
Continuous product systems of Hilbert spaces were coined by Arveson \cite{Arv89} for $\bR^+$, and their discrete counterparts were studied by Dinh \cite{Din91}. 
Motivated by Pimsner's seminal work \cite{Pim97}, Fowler \cite{Fow99} studied product systems of correspondences over quasi-lattices.
Since then discrete product systems have been studied by many authors (far too many to list here) and constitute an active area of research in their own right.
Recently there has been a growing interest in passing from quasi-lattices to right LCM-semigroups.
Kwa\'{s}niewski and Larsen \cite{KL19b} studied the Toeplitz-Nica-Pimsner C*-algebra $\N\T(X)$ for right LCM-semigroups proving Toeplitz-Cuntz-Krieger-type uniqueness theorems.
Here we turn our focus to equivariant quotients with an eye towards Cuntz-type covariant realizations.

One of the main questions in this direction has been to identify the appropriate quotient of $\N\T(X)$ so that faithful representations of $A$ lift to faithful representations of the quotient.
This cannot be expected to hold unconditionally.
The next best hope is thus to locate the quotient of $\N\T(X)$ so that faithful representations of $A$ lift to faithful representations of its fixed point algebra.
Sehnem \cite{Seh18} has provided a full answer by introducing the strongly covariant representations.
This generalizes the study of Cuntz-Nica-Pimsner relations, initiated by Sims and Yeend \cite{SY11}, and later continued by Carlsen-Larsen-Sims-Vittadello \cite{CLSV11}.
A second aim of \cite{CLSV11} was to use these relations and provide a co-universal object by passing to an appropriate reduced quotient.
This was achieved under extra conditions on the product system (such as injectivity or directness of the quasi-lattice).

Co-universality and boundary representations arise naturally in the context of nonselfadjoint operator algebras and their C*-envelope in the sense of Arveson.
With Dor-On, in \cite{DKKLL20} we introduced a coaction variant of the C*-envelope and used it to fully answer the problem of Carlsen-Larsen-Sims-Vittadello \cite{CLSV11} without any assumptions on the product system $X$.
Even more, the results of \cite{DKKLL20} apply to weak right LCM-inclusions $(G,P)$ rather than just quasi-lattices; more specifically, the C*-envelope $\cenv(\T_\la(X)^+, G, \ol{\de}_G)$ of the Fock tensor algebra $\T_\la(X)^+$ with its normal coaction is co-universal for equivariant injective Nica-covariant representations of $X$.
Seeing Sehnem's covariance algebra $A \times_{X} P$ as the universal C*-algebra of an induced Fell bundle we further showed that $\cenv(\T_\la(X)^+, G, \ol{\de})$ coincides with the reduced C*-algebra of this Fell bundle, here denoted by $A \times_{X, \la} P$.

The algebraic structure of $\cenv(\T_\la(X)^+, G, \ol{\de}_G)$ was studied in \cite{DKKLL20}.
Pivotal in this endeavour was the remark that the strong covariance relations of Sehnem are actually filtered through the Fock representation.
Following \cite{Seh18} we will denote by $A \times_{X} P$ the universal C*-algebra with respect to the strongly covariant representations of $X$.
We further consider the induced quotient $q_{\scv}(\T_\la(X))$ of $\T_\la(X)$ by the strong covariance relations.
In \cite{DKKLL20} it is shown that the canonical map
\begin{equation}\label{eq:maps}
q_{\scv}(\T_\la(X)) \longrightarrow \cenv(\T_\la(X)^+, G, \ol{\de}) \simeq A \times_{X, \la} P
\end{equation}
is faithful if and only if the normal coaction of $\T_\la(X)$ descends to a normal coaction on $q_{\scv}(\T_\la(X))$, e.g., when $G$ is exact.

The motivation for the present work is two-fold.
On one hand we wish to explore further general settings that entail normality of the coaction of $q_{\scv}(\T_\la(X))$ and thus identify the algebraic structure of the co-universal object.
Our main theorem here is that this happens when $(G,P)$ is controlled by another weak right LCM-inclusion $(\G,\P)$ with $\G$ is exact.
When $\G$ is abelian we can further induce dual actions on the C*-algebras.
This has the remarkable consequence that the canonical $*$-epimorphism
\begin{equation}
\cenv(\T_\la(X)^+, G, \ol{\de}) \longrightarrow \cenv(\T_\la(X)^+)
\end{equation}
is faithful. 
On the other hand we wish to use the co-universal property in such a context and apply it in the identification of C*-algebras.
The quotient by the strong covariance relations is used as a model in several constructions and this line of reasoning allows to show functoriality without checking a long list of C*-properties.
This is quite pleasing in particular because reduced C*-algebras do not enjoy a priori universal properties.
In fact we follow the reverse route of using the identification of reduced objects and then lift them to $*$-isomorphisms of the universal ones.

%%%%%%%%%%%%%%%%%%%%%%%%%%%%%%%%
\subsection{Main results}
%%%%%%%%%%%%%%%%%%%%%%%%%%%%%%%%

Controlled maps $\vartheta \colon (G,P) \to (\G,\P)$ between quasi lattice ordered groups were introduced by Laca-Raeburn \cite{LR96} with the purpose of extending the range of application of the faithfulness and uniqueness theorems for Toeplitz algebras of quasi lattice ordered groups.  
The key idea is that $(G,P)$ is amenable in the sense of Nica \cite{Nic92} provided that $\G$ is an amenable group.
A similar notion of controlled map was formulated simultaneously and independently by  
Crisp to prove that some Artin monoids inject in their groups \cite{Cri99}.  
The combination of these two sets of ideas led to the amenability and nonamenability results for Artin monoids in \cite{CL02}. 
Similar results can be derived for the Fock algebra $\T_\la(X)$ of a product system over $P$, as it has a $\P$-core that can be expressed as a direct sum of matrix algebras (see for example the proof of Theorem \ref{T:ntx un}).
As a consequence one obtains for example that compactly aligned product systems over the free semigroup $\bF_+^n$ are amenable, although the group $\bF^n$ is not, the reason being that the pair $(\bF^n, \bF_+^n)$ is controlled by its abelianization or by its length map on $(\bZ, \bZ_+)$.

However this type of argument is no longer valid for equivariant quotients as these relations live in the diagonal of the $P$-core (and thus in the $\P$-core).
An elimination method was recently developed in \cite{Kak19} when $(\G,\P) = (\bZ^n, \bZ_+^n)$ with the purpose of studying nuclearity and exactness properties.
By building further on these techniques, in Subsection \ref{Ss:ce} we give a controlled elimination method for passing from the $\P$-cores to the $P$-cores of injective Nica-covariant representations.
Essentially the method asserts that any relation in a $\P$-core must live at the diagonal and thus in a $P$-core.
We then use this to lift all properties from the realm of the $P$-fixed point algebras to the $\P$-fixed point algebras.
For example this applies to the fixed-point-algebra property of Sehnem's algebra \cite{Seh18} (Corollary \ref{C:fpa con}).
In particular exactness of $\G$ impacts on the maps appearing in (\ref{eq:maps}).

\vspace{8pt}

\noindent
{\bf Theorem A.} [Theorem \ref{T:cenv red}].
\emph{Let $\vartheta \colon (G,P) \to (\G, \P)$ be a controlled map between weak right LCM-inclusions and let $X$ be a compactly aligned product system over $P$ with coefficients in $A$. 
Let the canonical $*$-epimorphisms
\begin{equation}
q_{\scv}(\T_\la(X)) \longrightarrow A \times_{X, \la} P \simeq \cenv(\T_\la(X)^+, G, \ol{\de}_G) \longrightarrow \cenv(\T_\la(X)^+).
\end{equation}
If $\G$ is exact then the left map is faithful.
If in addition $\G$ is abelian then the right map is also faithful.
}

\vspace{8pt}

Theorem A implies that the coaction on $q_{\scv}(\T_\la(X))$ is normal when $\G$ is exact.
As pointed out in \cite{DKKLL20} this implies that the reduced Hao-Ng problem over discrete group actions has a positive answer (Remark \ref{R:Hao-Ng}).
A similar method applies whenever the C*-envelope functor is stable under crossed products, e.g., for dynamics over abelian locally compact groups or when the tensor algebra is hyperrigid \cite{Kat20, KR16}, and we leave this to the interested reader.
A further consequence of Theorem A is that amenability of $\G$ implies amenability of the product system and thus universality of the reduced constructions (Theorem \ref{T:ntx un}).
The case of abelian $\G$ directly generalizes the results of \cite{DK18a}.
There is further potential for Takai duality results even when $(G,P)$ does not admit a dual.
A further consequence of Theorem A provides a generalization of the Extension Theorem of \cite{KR16}, which recognizes a Fock tensor algebra by the presence of a coaction (Corollary \ref{C:extension con}).

Another application of the controlled elimination method concerns nuclearity/exactness results.
It has been observed by Katsura \cite{Kat04} that nuclearity of a Cuntz-Pimsner algebra is equivalent to the coefficient algebra being nuclearly embedded in the fixed point algebra.
Kakariadis \cite{Kak19}  produced similar results for $\bZ_+^n$.
In Theorem \ref{T:nuclear} we first give an equivalent characterization for nuclearity of $\T_\la(X)$ for right LCM-semigroups.
Although our original goal was to exploit $A \times_{X} P$, we tackle any equivariant quotient of $\N\T(X)$ that is injective on $A$.

\vspace{8pt}

\noindent
{\bf Theorem B.} [Theorem \ref{T:exact}, Theorem \ref{T:nuclear 2}].
\emph{Let $\vartheta \colon (G,P) \to (\G,\P)$ be a controlled map between weak right LCM-inclusions with $\G$ amenable and let $X$ be a compactly aligned product system over $P$ with coefficients in $A$.
Let $(\pi,t)$ be an equivariant injective Nica-covariant representation of $X$.
Then:
\begin{enumerate}
\item $A$ is exact if and only if $\ca(\pi,t)$ is exact.
\item $A \hookrightarrow \ca(\pi,t)$ is nuclear if and only if $\ca(\pi,t)$ is nuclear.
\end{enumerate}}

\vspace{8pt}

We emphasize that the controlled elimination process occurs at the level of representations.
One might be intrigued to introduce a product system $Y$ over $\P$ that would share the same algebras with $X$ over $P$.
However it is not clear that such a procedure gives a compactly aligned product system.
For this reason we introduce the notion of saturation for controlled maps, which preserves inclusions of ideals in the semigroups.
Under this condition we do get a super-product system on the \emph{same} coefficient algebra that does the job.

\vspace{8pt}

\noindent
{\bf Theorem C.} [Theorem \ref{T:con ta}].
\emph{Let $\vartheta \colon (G, P) \to (\G, \P)$ be a saturated controlled map between weak right LCM-inclusions.
Let $X$ be a (resp.\ injective) compactly aligned product system over $P$ with coefficients in $A$ and let
\begin{equation*} \label{eq:Y}
Y_h: = \sumoplus_{p \in \vartheta^{-1}(h)} X_p \qfor h \in \P.
\end{equation*}
Then the collection $Y = \{Y_h\}_{h \in \P}$ is a (resp.\ injective) compactly aligned product system over $\P$ with coefficients in $A$ such that $\T_\la(X)^+ \simeq \T_\la(Y)^+$ with
\[
\T_\la(X) \simeq \T_\la(Y) 
\qand
A \times_{X, \la} P \simeq  A \times_{Y, \la} \P,
\]
by $*$-isomorphisms that preserve the inclusions $X_p \mapsto Y_{\vartheta(p)}$ for all $p \in P$.
These $*$-isomor\-phisms further lift to $*$-isomorphisms
\[
\N\T(X) \simeq \N\T(Y)
\qand
A \times_{X} P \simeq A \times_{Y} \P,
\]
 that preserve the inclusions $X_p \hookrightarrow Y_{\vartheta(p)}$ for all $p \in P$.}

\vspace{8pt}

Our method here is to show that the $*$-isomorphism $\T_\la(X) \simeq \T_\la(Y)$ is canonical on the tensor algebras and then apply the C*-envelope machinery to induce the $*$-isomorphism $A \times_{X, \la} P \simeq  A \times_{Y, \la} \P$.
The saturation property can be induced by free products of abelian total orders, and is preserved by semi-direct products.
As a notable application of this method we deduce that Sehnem's covariance algebra of a product system over $\bF_n^+$ is nothing more than the Cuntz-Pimsner algebra of a single C*-correspondence, in a similar way that the Nica-Cuntz-Pimsner algebra of $\bF_n^+$ coincides with $\O_n$ (Corollary \ref{C:freesem}).

We then take a closer look at total orders.
To further motivate these results, recall that the Cuntz algebra $\O_n$ may be viewed as the Cuntz-Pimsner algebra of a Hilbert bimodule over the $n^{\infty}$-hyperfinite C*-algebra. 
In spite of the coefficient algebra of the latter being much larger, Hilbert bimodules are better behaved than other types of C*-correspondences and they allow for a rich theory, including versions of Takai duality. 
Here we will show that the situation with $\O_n$ generalizes to product systems that are controlled by exact total orders.
Towards this end we consider \emph{reversible} product systems for which the image of every fiber in $A \times_{X, \la} P$ is a Hilbert bimodule.
We then show that reversible product systems produce all possible covariance algebras for weak right LCM-inclusions that are controlled by total orders in a saturated way.
The construction relies on using the fixed point algebra and generalizes results of Pimsner \cite{Pim97}, Abadie, Eilers and Exel \cite{AEE98}, Schweizer \cite{Sch01}, Kakariadis and Katsoulis \cite{KK12}, and Meyer and Sehnem \cite{MS19}.
However our proof uses the C*-envelope machinery and thus avoids categorical arguments.

\vspace{8pt}

\noindent
{\bf Theorem D.} [Theorem \ref{T:dil rev}].
\emph{Let $\vartheta \colon (G,P) \to (\G,\P)$ be a saturated controlled map between weak right LCM-inclusions and suppose that $(\G,\P)$ is a total order.
Let $X$ be a (resp.\ injective) product system over $P$ with coefficients in $A$.
Then there exists a (resp.\ injective) reversible product system $Z$ over $\P$ with coefficients in a C*-algebra $B$ such that
\begin{equation}
A \subseteq B \qand X_p \subseteq Z_{\vartheta(p)} \foral p \in P,
\end{equation}
that satisfies
\begin{equation}
A \times_{X} P \simeq B \times_{Z} \P
\qand
A \times_{X, \la} P \simeq B \times_{Z, \la} \P,
\end{equation}
by $*$-isomorphisms that preserve the inclusions $X_p \hookrightarrow Z_{\vartheta(p)}$ for all $p \in P$.
}

\vspace{8pt}

Semigroup C*-algebras have been an important source of inspiration for this study.
Our results have a direct application to C*-algebras of right LCM-semigroups where $X_p = \bC$ for every $p \in P$.
In this case the Nica-Toeplitz C*-algebra is denoted by $\ca_s(P)$ for the Nica-covariant representations of $P$ and Theorem A (and in particular Theorem \ref{T:ntx un}) is a direct generalization of \cite[Theorem 4.7]{CL07}.
Faithfulness of the maps of Theorem A has been further investigated in \cite{KKLL21} for (not-necessarily right LCM) semigroups that embed in exact groups.
Theorem B asserts that every quotient of $\ca_s(P)$ is nuclear and aligns with \cite[Corollary 8.3]{Li13} for quasi-lattices.
Under the saturation property, Theorem C asserts that the operator algebras of $P$ coincide with those of a product system $Y$ over $\P$ with $Y_h = \bC^{|\vartheta^{-1}(h)|}$ for $h \in \P$.
This follows a recurring idea of obtaining realizations of the same C*-algebra in different classes.
It has been shown in \cite{Li17} that $\bC \times_{\bC, \la} P$ can be realized as the partial crossed product of the smallest $\G$-invariant subspace of the fixed point algebra of $\ca_s(\P)$ by $G$.
Theorem D provides a similar (augmented) realization when $\vartheta$ is saturated and $(\G,\P)$ is a total order.

Let us close with a remark on controlled maps.
It has been known that controlled maps cannot handle HNN extensions of quasi-lattices as the height map does not have a trivial kernel on the semigroup.
In order to resolve this, recently an Huef, Nucinkis, Sehnem and Yang \cite{HNSY19} introduced a more general definition of controlled maps for weak quasi-lattices that allows infinite descending chains and thus produces direct limits of matrix algebras.
The controlled elimination arguments we provide here should be compatible with this general definition, as they refer to ideals of representations, which are compatible with direct limits.

%%%%%%%%%%%%%%%%%%%%%%%%%%%%%%%%
\subsection{Structure of the paper}
%%%%%%%%%%%%%%%%%%%%%%%%%%%%%%%%

In Section \ref{S:coa} we review the boundary theory and the theory of the cosystems from \cite{DKKLL20}.
In Sections \ref{S:oaps} and \ref{S:caps lcm} we review the main elements of the product systems theory, and we see how they are enriched under the presence of a controlled map.
We have included more details from \cite{DKKLL20} in order to set the ground for the next sections, and also prove additional results that are not covered in \cite{DKKLL20}.
In Section \ref{S:ce} we present the controlled elimination method.
Section \ref{S:app} contains the applications to Sehnem's covariance algebra, the structure of the co-universal C*-algebra, amenable product systems, nuclearity/exactness, and the reduced Hao-Ng problem.
In Section \ref{S:scm} we give the product system re-parametrizations under the saturation property with applications  to reversible product systems.

%%%%%%%%%%%%%%%%%%%%%%%%%%%%%%%%
\begin{acknow}
%Part of the research was carried out at the Banff International Research Station during the Focused Research Group week on ``Noncommutative boundaries for tensor algebras'' (20frg248).
Evgenios Kakariadis was partially supported by EPSRC (Grant No.\ EP/T02576X/1) and LMS (Grant No.\ 41908).
Marcelo Laca was partially supported by NSERC Discovery Grant RGPIN-2017-04052.
Xin Li has received funding from the European Research Council (ERC) under the European Union's Horizon 2020 research and innovation programme (grant agreement No. 817597).
\end{acknow}

%%%%%%%%%%%%%%%%%%%%%%%%%%%%%%%%
\section{Operator algebras and their coactions} \label{S:coa}
%%%%%%%%%%%%%%%%%%%%%%%%%%%%%%%%

%%%%%%%%%%%%%%%%%%%%%%%%%%%%%%%%
%\subsection{Notation}
%%%%%%%%%%%%%%%%%%%%%%%%%%%%%%%%

%%%%%%%%%%%%%%%%%%%%%%%%%%%%%%%%
\subsection{Operator algebras}
%%%%%%%%%%%%%%%%%%%%%%%%%%%%%%%%

The reader may refer to \cite{BL04, Pau02} for the general theory of nonselfadjoint operator algebras and dilations of their representations.

Let $\A$ be an operator algebra, which in this paper means a subalgebra of $\B(H)$ for a Hilbert space $H$.
We say that $(C, \iota)$ is a \emph{C*-cover} of $\A$ if $\iota \colon \A \to C$ is a completely isometric representation with $C = \ca(\iota(\A))$.
The \emph{C*-envelope} $\cenv(\A)$ of $\A$ is a C*-cover $(\cenv(\A), \iota)$ with the following co-universal property:
if $(C', \iota')$ is a C*-cover of $\A$ then there exists a (necessarily unique) $*$-epimorphism $\Phi \colon C' \to \cenv(\A)$ such that $\Phi(\iota'(a)) = \iota(a)$ for all $a \in \A$.
Arveson defined the C*-envelope in \cite{Arv69} and computed it for a variety of operator algebras, predicting its existence in general.
Ten years later Hamana \cite{Ham79} confirmed Arveson's prediction by proving the existence of injective envelopes for the unital case.
The C*-envelope is the C*-algebra generated in the injective envelope of $\A$ once this is endowed with the Choi-Effros C*-structure. 

Dritschel-McCullough \cite{DM05} provided an alternative proof based on maximal dilations for the unital case.
A \emph{dilation} of a representation $\phi \colon \A \to \B(H)$ is a representation $\phi' \colon \A \to \B(H')$ such that $H \subseteq H'$ and $\phi(a) = P_H \phi'(a) |_H$ for all $a \in \A$.
A completely contractive map $\phi \colon \A \to \B(H)$ is called \emph{maximal} if every dilation $\phi' \colon \A \to \B(H')$ is trivial, i.e., $P_H \phi'(a) =\phi(a) = \phi'(a) |_H$ for all $a \in \A$.
It follows that the C*-envelope is the C*-algebra generated by a maximal completely isometric representation.

It does not hold in general that if $\pi \colon \cenv(\A) \to \B(H)$ is a $*$-representation then it is the unique ccp extension of $\pi|_{\A}$.
The algebra $\A$ is called \emph{hyperrigid} if this is the case for any representation $\pi$ of $\cenv(\A)$.
An operator algebra $\A$ is said to be \textit{Dirichlet} if
\[
\cenv(\A) = \overline{\A+\A^*}.
\] 
Equivalently, $\A$ is Dirichlet if there exists a C*-cover $(C,\iota)$ of $\A$ such that $C = \ol{\iota(\A) + \iota(\A)^*}$, in which case $C = \cenv(\A)$.
It follows that Dirichlet algebras are automatically hyperrigid.

%%%%%%%%%%%%%%%%%%%%%%%%%%%%%%%%
\subsection{Co-actions on operator algebras}
%%%%%%%%%%%%%%%%%%%%%%%%%%%%%%%%

If $\X$ and $\Y$ are subspaces of some $\B(H)$ then we write $[\X \Y]:= \ol{\spn}\{x y \mid x \in \X, y \in \Y\}$.
All groups and semigroups we consider are discrete and unital.
We denote the spatial tensor product by $\otimes$.

For a discrete group $G$ we write $u_g$ for the unitary generator associated with $g \in G$ in the full group C*-algebra $\ca(G)$.
We write $\la_g$ for the generators of the left regular representation $\ca_\la(G)$.
We write $\la \colon \ca(G) \to \ca_\la(G)$ for the canonical $*$-epimorphism.
Recall that $\ca(G)$ admits a faithful $*$-homomorphism
\[
\De \colon \ca(G) \longrightarrow \ca(G) \otimes \ca(G) 
\text{ such that }
\De(u_g) = u_g \otimes u_g,
\]
given by the universal property of $\ca(G)$, and with left inverse given by $\id \otimes \chi$ for the character $\chi$ of $\ca(G)$.
We will require some preliminaries from \cite{DKKLL20} on coactions on operator algebras.

%%%%%%%%%%%%%%%%%%%%%%%%%%%%%%%%%%%%%%%%%%%
\begin{definition}\label{D:cis coa} \cite[Definition 3.1]{DKKLL20}
Let $\A$ be an operator algebra.
A \emph{coaction of $G$ on $\A$} is a completely isometric representation $\de \colon \A \to \A \otimes \ca(G)$ such that the linear span of the induced subspaces
\[
\A_g := \{a \in \A \mid \de(a) = a \otimes u_g\}
\] 
is norm-dense in $\A$, in which case $\de$ satisfies the coaction identity
\[
(\de \otimes \id_{\ca(G)}) \de = (\id_{\A} \otimes \De) \de.
\]
If, in addition, the map $(\id \otimes \la) \de$ is injective then the coaction $\de$ is called \emph{normal}. 

If $\A$ is an operator algebra and $\de \colon \A \to \A \otimes \ca(G)$ is a coaction on $\A$, then we will refer to the triple $(\A, G, \de)$ as a \emph{cosystem}. 
A map $\phi \colon \A \to \A'$ between two cosystems $(\A, G, \de)$ and $(\A', G, \de')$ is said to be \emph{$G$-equivariant}, or simply equivariant, if $\de' \phi=(\phi\otimes \id)\de$.
\end{definition}

If $(\A, G, \de)$ is a cosystem then $\A_r \cdot \A_s \subseteq \A_{r s}$ for all $r, s \in G$, since $\de$ is a homomorphism.

%%%%%%%%%%%%%%%%%%%%%%%%%%%%%%%%
\begin{remark}\label{R:nd cis} \cite{DKKLL20}
Suppose that $(\A, G, \de)$ is a cosystem and that $\de$ extends to a $*$-homomorphism $\de \colon \ca(\A) \to \ca(\A) \otimes \ca(G)$ that satisfies the coaction identity
\[
(\de \otimes \id) \de(c) = (\id \otimes \De) \de(c) \foral c \in \ca(\A).
\]
Then $\de$ is automatically non-degenerate on $\ca(\A)$ in the sense that
\[
\big[ \de(\ca(\A)) \ca(\A) \otimes \ca(G) \big] = \ca(\A) \otimes \ca(G).
\]
Moreover Definition \ref{D:cis coa} covers that of full coactions of Quigg \cite{Qui96} when $\A$ is a C*-algebra.
In this case $\de$ is a faithful $*$-homomorphism and we have that
\[
(\A_g)^* = \{a^* \in \A \mid \de(a^*) = a^* \otimes u_{g^{-1}} \} = \A_{g^{-1}}.
\]
\end{remark}

Due to the Fell absorption principle, the existence of a ``reduced'' coaction implies that of a normal coaction.

%%%%%%%%%%%%%%%%%%%%%%%%%%%%%%%%
\begin{proposition}\label{P:Fell ind} \cite[Proposition 3.4]{DKKLL20}
Let $\A$ be an operator algebra.
Suppose there is a group $G$ that induces a grading on $\A$, i.e., there are subspaces $\{\A_g\}_{g \in G}$ such that $\sum_{g \in G} \A_g$ is norm-dense in $\A$, and a completely isometric homomorphism
\[
\de_\la \colon \A \longrightarrow \A \otimes \ca_\la(G),
\]
such that
\[
\de_\la(a_g) = a_g \otimes \la_g \foral a_g \in \A_g, \foral g \in G.
\]
Then $\A$ admits a normal coaction $\de$ of $G$ such that $\de_\la = (\id \otimes \la) \de$.
\end{proposition}

%%%%%%%%%%%%%%%%%%%%%%%%%%%%%%%%
\begin{example}
The reduced group C*-algebra $\ca_\la(G)$ admits a faithful $*$-homomophism
\[
\De_\la \colon \ca_\la(G) \longrightarrow \ca_\la(G) \otimes \ca_\la(G) 
\text{ such that }
\De_\la(\la_g) = \la_g \otimes \la_g.
\]
Thus $\ca_\la(G)$ admits a normal coaction $\de$ of $G$ such that $\De_\la = (\id \otimes \la) \de$.
\end{example}

%%%%%%%%%%%%%%%%%%%%%%%%%%%%%%%%%%%%%%%%%%%
\begin{definition}\label{D:co-action} \cite[Definition 3.6]{DKKLL20}
Let $(\A, G, \de)$ be a cosystem.
A triple $(C, \iota, \de_C)$ is called a \emph{C*-cover} for $(\A, G, \de)$ if $(C, \iota)$ is a C*-cover of $\A$ and $\de_C \colon C \to C \otimes \ca(G)$ is a coaction on $C$ such that the following diagram
\[
\xymatrix{
\A \ar[rr]^{\iota} \ar[d]^{\de} & & C \ar[d]^{\de_C} \\
\A \otimes \ca(G) \ar[rr]^{\iota \otimes \id} & & C \otimes \ca(G)
}
\]
commutes.
When the coaction is understood we will say that $C$ is a C*-cover for $\A$ over $G$.
\end{definition}

%%%%%%%%%%%%%%%%%%%%%%%%%%%%%%%%%%%%%%%%%%%
\begin{definition} \cite[Definition 3.7]{DKKLL20}
Let $(\A, G, \de)$ be a cosystem.
The \emph{C*-envelope} of $(\A, G, \de)$ is a C*-cover $(\cenv(\A, G, \de), \iota, \de_{\env})$ such that: for every C*-cover $(C', \iota', \de')$ of $(\A, G, \de)$ there exists a $*$-epimorphism $\Phi \colon C' \to \cenv(\A, G, \de)$ that fixes $\A$ and intertwines the coactions, i.e., the diagram
\[
\xymatrix{
\iota'(\A) \ar[rrr]^{\de'} \ar[d]^{\Phi} & & & C' \otimes \ca(G) \ar[d]^{\Phi \otimes \id} \\
\iota(\A) \ar[rrr]^{\de_{\env}} & & & \cenv(\A, G, \de) \otimes \ca(G)
}
\]
is commutative on $\A$, and thus is commutative on $C'$.
\end{definition}

The existence of the C*-envelope of a cosystem was proved in \cite{DKKLL20} by a direct computation  that uses the C*-envelope of the ambient operator algebra.
In order to state the result explicitly we need to make some preliminary remarks and establish the notation.
Suppose $(\A, G, \de)$ is a cosystem, let $i \colon \A \to \cenv(\A)$ be the C*-envelope of $\A$, and recall that the spatial tensor product of completely isometric maps is completely isometric.
Then the representation of $\A$  obtained via the  composition
\[
\xymatrix@C=2cm{
\A \ar[r]^{\de} & \A \otimes \ca(G) \ar[r]^{i \otimes \id} \ar[r] & \cenv(\A) \otimes \ca(G)
}
\]
is completely isometric, and the C*-algebra 
\[
\ca((i \otimes \id) \de(\A)) := \ca(i(a_g) \otimes u_g \mid g \in G)
\]
becomes a C*-cover of $\A$.
This C*-cover is special because it admits a coaction $\id \otimes \De$, such that the triple
\[
(\ca(i(a_g) \otimes u_g \mid g \in G), (i \otimes \id) \de, \id \otimes \De)
\]
is a C*-cover for $(\A, G, \de)$. 
The following theorem summarizes fundamental results about existence and representations of C*-envelopes for cosystems.

%%%%%%%%%%%%%%%%%%%%%%%%%%%%%%%%%%%%%%%%%%%
\begin{theorem} \label{T:fpa 1-1} \cite[Theorem 3.8, Corollary 3.9 and Corollary 3.10]{DKKLL20}
Let $(\A, G, \de)$ be a cosystem and let $i \colon \A \to \cenv(\A)$ be the inclusion map.
Then
\[
(\cenv(\A, G, \de), \iota, \de_{\env}) \simeq (\ca(i(a_g) \otimes u_g \mid g \in G), (i \otimes \id )\de, \id \otimes \De).
\]
If in addition $\de$ is normal on $\A$ then $\de_{\env}$ is normal on $\cenv(\A, G, \de)$.

Moreover if $\Phi \colon \cenv(\A, G, \de) \to B$ is a $*$-homomorphism that is completely isometric on $\A$ then it is faithful on the fixed point algebra of $\cenv(\A, G, \de)$.
\end{theorem}

%%%%%%%%%%%%%%%%%%%%%%%%%%%%%%%%
\begin{remark}\label{R:cenv ab}
A co-action of an abelian group $G$ is equivalent to point-norm continuous actions $\{\be_\ga\}_{\ga \in \wh{G}}$ of the dual group $\wh{G}$.
Since every $\be_\ga$ is a completely isometric automorphism it extends to the C*-envelope.
Hence the C*-envelope of a cosystem coincides with the usual C*-envelope of the ambient operator algebra when $G$ is abelian.
Equivalently, every coaction of an abelian group on an operator algebra lifts to a coaction on its C*-envelope.
As pointed out in \cite{DKKLL20}, it is unknown if this is the case for general amenable groups.
\end{remark}

Group homomorphisms implement coactions.
Note that the following proposition for $\G = \{e_\G\}$ says nothing more than that every C*-cover of a cosystem is a C*-cover of the ambient operator algebra.

%%%%%%%%%%%%%%%%%%%%%%%%%%%%%%%%
\begin{proposition}\label{P:coa ind}
Let $(\A, G, \de_G)$ be a (resp. normal) cosystem and let $\vartheta \colon G \to \G$ be a group homomorphism.
Then $\G$ induces a (resp. normal) coaction $\de_\G$ on $\A$.
Thus every C*-cover of $\A$ over $G$ is also a C*-cover of $\A$ over $\G$.
\end{proposition}

\begin{proof}
By the universal property of $\ca(G)$ we have a $*$-homomorphism
\[
\wt{\vartheta} \colon \ca(G) \longrightarrow \ca(\G) ; u_g \mapsto u_{\vartheta(g)}.
\]
We then have the canonical completely contractive homomorphism
\[
\xymatrix{
\de_\G \colon \A \ar[rr]^{\de_G \phantom{ooooo}} & & \A \otimes \ca(G) \ar[rr]^{\id \otimes \wt{\vartheta}} & & \A \otimes \ca(\G)
}
\]
which has $\id \otimes \chi$ as a completely contractive left inverse.
By definition we have that
\[
\A_h := \{a \in \A \mid \de_\G(a) = a \otimes u_h\} \supseteq \{a \in \A_g \mid \vartheta(g) = h \}
\]
and thus
\[
\A = \ol{\sum_{g \in G} \A_g} \subseteq \ol{\sum_{h \in \G} \A_h} \subseteq \A.
\]
Hence $(\id \otimes \wt{\vartheta}) \de_G$ defines a coaction of $\G$ on $\A$.

Next suppose that $\de_G$ is normal and let $\de_{G, \la} = (\id \otimes \la) \de_G$.
Let $\de_\G$ be the coaction induced by $\de_G$.
By Fell's absorption principle we have that the map $\la_g \mapsto \la_g \otimes \la_{\vartheta(g)}$ gives a faithful $*$-homomorphism of $\ca_\la(G)$ and thus we get the induced completely isometric representation
\[
\xymatrix{
\A \ar[rr]^{\de_{G, \la} \phantom{oooooooo} } \ar@{.>}[d]_{\de_{\G, \la}} & & \ol{\alg}\{ a_g \otimes \la_g \mid g \in G \} \ar[d] \\
\ol{\alg}\{ a_g \otimes \la_{\vartheta(g)} \mid g \in \G \} & & \ol{\alg}\{ a_g \otimes \la_g \otimes \la_{\vartheta(g)} \mid g \in G\} \ar[ll]_{\de_{G, \la}^{-1} \otimes \id}
}
\]
which induces a faithful $*$-homomorphism $\de_{\G, \la}$.
It follows that $\de_{\G, \la} = (\id \otimes \la) \de_\G$ and thus $\de_\G$ is a normal coaction of $\G$ on $\A$.
\end{proof}

%Let us see an application on crossed products by discrete groups $\fG$.
%Suppose that $\A \subseteq \B(H)$ admits a $\fG$-action by completely isometric automorphisms.
%Then one can define the reduced crossed product $\A \rtimes_{\al, \la} \fG$ of $\A$ by $\fG$ as the norm-closed subalgebra of $\B(H) \otimes \B(\ell^2(\fG))$ generated by $\{\pi(a) \fU_{\fg} \mid a \in \A, \fg \in \fG\}$ such that
%\[
%\pi(a) (\xi \otimes e_{\fh}) = \al_{\fh}(a)\xi \otimes e_{\fh}
%\foral a \in \A
%\qand
%\fU_\fg (\xi \otimes e_{\fh}) = \xi \otimes e_{\fg \cdot \fh}
%\foral
%\fg \in \fG.
%\]
%At the same time the $\fG$-action extends to an action on $\cenv(\A)$ and one can form the reduced C*-crossed product.
%Katsoulis \cite[Theorem 2.5]{Kat17} has shown that
%\[
%\cenv(\A \rtimes_{\al, \la} \fG) \simeq \cenv(\A) \rtimes_{\al, \la} \fG.
%\]
%In the presence of a canonical coaction this passes down to the C*-envelopes of the cosystems.
%
%%%%%%%%%%%%%%%%%%%%%%%%%%%%%%%%%%
%\begin{proposition}\label{P:cp env} \cite{DKKLL20}
%Let $(\A, G, \de)$ be a (resp. normal) coaction pair by a group $G$.
%Let $\fG$ be a group acting on $\A$ by completely isometric automorphisms so that
%\[
%\de \al_{\fg} = (\al_{\fg} \otimes \id) \de
%\foral
%\fg \in \fG.
%\]
%Then $G$ induces a (resp. normal) coaction $\de \otimes \id$ on $\A \rtimes_{\al, \la} \fG$ and
%\[
%\cenv(\A \rtimes_{\al, \la} \fG, G, \de) \simeq \cenv(\A, G, \de \otimes \id) \rtimes_{\al, \la} \fG.
%\]
%\end{proposition}

Let us close this section with some remarks on topological gradings from \cite{Exe97, Exe17}.
Recall that a \emph{topological grading} $\{\B_g\}_{g \in G}$ of a C*-algebra $\B$ consists of linearly independent subspaces that span a dense subspace of $\B$ and are compatible with the group $G$, i.e., $\B_g^* = \B_{g^{-1}}$ and $\B_g \cdot \B_h \subseteq \B_{g h}$.
By \cite[Theorem 3.3]{Exe97} the linear independence condition can be substituted by the existence of a conditional expectation on $\B_e$.
The maximal C*-algebra $\ca(\B)$ of $\B$ is defined as universal with respect to the representations of $\B$.
The reduced C*-algebra $\ca_\la(\B)$ of $\B$ is defined by the left regular representation of $\B$ on $\ell^2(\B)$.

%%%%%%%%%%%%%%%%%%%%%%%%%%%%%%%%
\begin{definition}
Let $\B = \{\B_g\}_{g \in G}$ be a topological grading over a group $G$ in a C*-algebra $\ca(\B)$ that it generates, with completely contractive Fourier maps $E_g \colon \ca(\B) \to \B_g$, i.e., 
\[
E_g(b) = \de_{g,h} b \foral b \in \B_h \text{ and } g, h \in G.
\]
An ideal $\I \lhd \ca(\B)$ is called \emph{induced} if $\I = \sca{\I \cap \B_e}$.
An ideal $\I \lhd \ca(\B)$ is called \emph{Fourier} if $E_g(f) \subseteq \I$ for every $f \in \I$.
\end{definition}

%%%%%%%%%%%%%%%%%%%%%%%%%%%%%%%%
\begin{remark}\label{R:Exel 1}
It follows that an ideal $\I \lhd \ca(\B)$ is Fourier if and only if $E_e(f^*f) \in \I$ for all $f \in \I$.
Every induced ideal is a Fourier ideal.
The converse holds if $G$ is exact and $E_e$ is a faithful conditional expectation.
These can be found at \cite[Proposition 23.9]{Exe17}.
\end{remark}

A topological grading defines a \emph{Fell bundle} and once a representation of a Fell bundle is established the two notions are the same.
In a loose sense a Fell bundle $\B$ over a discrete group $G$ is a collection of Banach spaces $\{\B_g\}_{g \in G}$, often called \emph{the fibers of $\B$}, that satisfy canonical algebraic properties and the C*-norm properties; see \cite[Definition 16.1]{Exe17}.
So we will alternate freely between Fell bundles and topologically graded C*-algebras.
Spectral subspaces of coactions on C*-algebras are an important source of topological gradings.

%%%%%%%%%%%%%%%%%%%%%%%%%%%%%%%%
\begin{definition}
Let $\de$ be a coaction of $G$ on a C*-algebra $C$ and let $\I \lhd C$ be an ideal of $C$.
We say that the quotient map is \emph{$G$-equivariant}, or that the quotient $C/\I$ is \emph{$G$-equivariant} if $\de$ descends to a coaction of $G$ on $C/\I$.
\end{definition}

%%%%%%%%%%%%%%%%%%%%%%%%%%%%%%%%
\begin{remark}\label{R:Exel 2}
If $\de \colon C \to C \otimes \ca(G)$ is a coaction and $\I \lhd C$ is an induced ideal then $\de$ descends to a faithful coaction of $G$ on $C/\I$, see for example \cite[Proposition A.1]{CLSV11}.
The same holds for the normal actions when $G$ is exact, see for example \cite[Proposition A.5]{CLSV11}. 
\end{remark}

%%%%%%%%%%%%%%%%%%%%%%%%%%%%%%%%
\section{Operator algebras of product systems} \label{S:oaps}
%%%%%%%%%%%%%%%%%%%%%%%%%%%%%%%%

%%%%%%%%%%%%%%%%%%%%%%%%%%%%%%%%
\subsection{C*-correspondences}
%%%%%%%%%%%%%%%%%%%%%%%%%%%%%%%%

A \emph{C*-correspondence} $X$ over $A$ is a right Hilbert module over $A$ with a left action given by a $*$-homomorphism $\vphi_X \colon A \to \L X$.
We write $\L X$ and $\K X$ for the adjointable operators and the compact operators of $X$, respectively.
For two C*-corresponden\-ces $X, Y$ over the same $A$ we write $X \otimes_A Y$ for the balanced tensor product over $A$.
We say that $X$ is unitarily equivalent to $Y$ (symb. $X \simeq Y$) if there is a surjective adjointable operator $U \in \L(X,Y)$ such that $\sca{U \xi, U \eta} = \sca{\xi, \eta}$ and $U (a \xi b) = a U(\xi) b$ for all $\xi, \eta \in X$ and $a,b \in A$.
A C*-correspondence is called \emph{injective} if the left action is injective.

A \emph{representation} $(\pi,t)$ of a C*-correspondence is a left module map that preserves the inner product.
Then $(\pi,t)$ is automatically a bimodule map.
Moreover there exists a $*$-homomorphism $\psi$ on $\K X$ such that $\psi(\theta_{\xi, \eta}) = t(\xi) t(\eta)^*$ for all $\theta_{\xi, \eta} \in \K X$. When $\pi$ is injective, then both $t$ and $\psi$ are isometric.
A representation $(\pi,t)$ is called covariant if it satisfies $\pi(a) = \psi(\vphi_X(a))$ for all $a$ in Katsura's ideal $J_X := \ker\vphi_X^\perp \bigcap \vphi_X^{-1}(\K X)$.

%%%%%%%%%%%%%%%%%%%%%%%%%%%%%%%%
\subsection{Toeplitz algebras}
%%%%%%%%%%%%%%%%%%%%%%%%%%%%%%%%

Let $P$ be a unital subsemigroup of a group $G$.
We will write $P^*$ for the set of elements in $P$ that are invertible in $P$.
A \emph{product system $X$ over $P$} is a family $\{X_p \mid p \in P\}$ of C*-correspondences over the same C*-algebra $A$ such that:
\begin{enumerate}
\item $X_e = A$.
\item There are multiplication rules $X_p \otimes_A X_q \simeq_{u_{p,q}} X_{pq}$ for every $p, q \in P \setminus \{e\}$.
\item There are multiplication rules $A \otimes_A X_p \simeq_{u_{e, p}} [A \cdot X_p]$ and $X_p \otimes_A A \simeq_{u_{p, e}} [X_p \cdot A] = X_p$ for all $p \in P$.
\item The multiplication rules are associative in the sense that
\[
u_{pq, r} (u_{p,q} \otimes \id_{X_r}) = u_{p, qr} (\id_{X_p} \otimes u_{q,r}) \foral p,q,r \in P.
\]
\end{enumerate}
We say that $X$ is \emph{injective} if every $X_p$ is injective.
If $x \in P^*$ then the multiplication rules impose that
\[
X_x \otimes_A X_{x^{-1}} \simeq A \simeq X_{x^{-1}} \otimes_A X_x.
\]
In particular every such $X_x$ is non-degenerate since
\[
A \otimes_A X_x \simeq X_x \otimes_A X_{x^{-1}} \otimes_A X_x \simeq X_{x} \otimes_A A = X_x.
\]
Throughout this work we will be assuming that all left actions are non-degenerate.
We do this in order to be able to use freely the results from \cite{DKKLL20, Seh18}.
Nevertheless it is possible that this assumption can be removed.

Henceforth we will suppress the use of symbols for the multiplication rules.
Thus we write $\xi_p \xi_q$ for the image of $\xi_p \otimes \xi_q$ under $u_{p,q}$, and so
\[
\vphi_{pq}(a)(\xi_p \xi_q) = (\vphi_p(a) \xi_p) \xi_q \foral a \in A \text{ and } \xi_p \in X_p, \xi_q \in X_q.
\]
The product system structure gives maps
\[
i_{p}^{pq} \colon \L X_{p} \longrightarrow \L X_{pq}
\; \textup{ such that } \;
i_{p}^{pq}(S) (\xi_{p} \xi_{q})
=
(S \xi_{p}) \xi_{q}.
\]
If $x \in P^*$ then $i_{r}^{rx} \colon \L X_r \to \L X_{rx}$ is a $*$-isomorphism with inverse $i_{rx}^{rxx^{-1}} \colon \L X_{rx} \to \L X_{r}$.

%%%%%%%%%%%%%%%%%%%%%%%%%%%%%%%%
\begin{definition}
Let $P$ be a unital subsemigroup of a group $G$ and $X$ be a product system over $P$ with coefficients in $A$.
A \emph{Toeplitz representation $(\pi,t)$ of $X$} consists of a family of representations $(\pi, t_p)$ of $X_p$ over $A$ such that
\[
t_p(\xi_p) t_q(\xi_q) = t_{pq}(\xi_p \xi_q) \foral \xi_p \in X_p, \xi_q \in X_q.
\]
The \emph{Toeplitz algebra $\T(X)$ of $X$} is the universal C*-algebra generated by $A$ and $X$ with respect to the representations of $X$.
The \emph{Toeplitz tensor algebra $\T(X)^+$ of $X$} is the subalgebra of $\T(X)$ generated by $A$ and $X$.
\end{definition}

If $(\pi, t)$ is a Toeplitz representation then we write $\psi_p$ for the induced representation on $\K X_p$.
We obtain a bimodule triple $(\psi_r, \psi_{r,s}, \psi_s)$ on the bimodule $(\K X_r, \K(X_s, X_r), \K X_s)$ so that $\psi_{r,s}(\theta_{\xi_r, \xi_s}) = t_r(\xi_r) t_s(\xi_s)^*$.
We will often interpret $\pi$ as $t_e$ or $\psi_e$ to simplify our notation henceforth.

%%%%%%%%%%%%%%%%%%%%%%%%%%%%%%%%
\begin{proposition}\label{P:star inv LCM} \cite[Proposition 2.4]{DKKLL20}
Let $X$ be a product system over $P$ with coefficients in $A$.
Let $(\pi,t)$ be a Toeplitz representation of $X$.
If $x \in P^*$ then 
\[
t_{x}(X_x)^* = t_{x^{-1}}(X_{x^{-1}}).
\]
If $w \in P$ and $x \in P^*$ then
\[
i_w^{wx}(k_w) \in \K X_{wx}
\text{ and }
\psi_{wx}(i_{w}^{wx}(k_w)) = \psi_w(k_w) \foral k_w \in \K X_w.
\]
\end{proposition}

Suppose that $\T(X)$ is faithfully represented by $(\wt{\pi}, \wt{t})$.
By the universal property of $\T(X)$ there is a canonical $*$-homomorphism
\[
\wt{\de} \colon \T(X) \longrightarrow \T(X) \otimes \ca(G) ; \wt{t}(\xi_p) \mapsto \wt{t}(\xi_p) \otimes u_p.
\]
Sehnem \cite[Lemma 2.2]{Seh18} has shown that $\wt{\de}$ is a non-degenerate and faithful coaction of $\T(X)$ when $X$ is non-degenerate, with each spectral space $\T(X)_g$, with $g \in G$, be given by the products
\[
\wt{t}_{p_1}(\xi_{p_1}) \wt{t}_{p_2}(\xi_{p_2})^* \cdots \wt{t}_{p_n}(\xi_{p_n})^* \qfor p_1 p_2^{-1} \cdots p_n^{-1} = g.
\]
We will do a little bit more for semigroup homomorphisms.

%%%%%%%%%%%%%%%%%%%%%%%%%%%%%%%%
\begin{definition}
Let $P$ (resp. $\P$) be a unital subsemigroup of a group $G$ (resp. $\G$).
If $\vartheta \colon G \to \G$ is a group homomorphism such that $\vartheta(P) \subseteq \P$, we write $\vartheta \colon (G, P) \to (\G,\P)$ and say that $\vartheta$ is a \emph{semigroup preserving homomorphism}.
\end{definition}

%%%%%%%%%%%%%%%%%%%%%%%%%%%%%%%%
\begin{proposition}\label{P:t coaction}
Let $P$ be a unital subsemigroup of a group $G$ and $X$ be a product system over $P$ with coefficients in $A$.
Let $\vartheta \colon (G,P) \to (\G,\P)$ be a semigroup preserving homomorphism and suppose that $(\wt{\pi}, \wt{t})$ is a faithful representation of $\T(X)$.
Then there is a coaction of $\G$ on $\T(X)$ such that
\[
\wt{\de} \colon \T(X) \longrightarrow \T(X) \otimes \ca(\G) ; \wt{t}(\xi_p) \mapsto \wt{t}(\xi_p) \otimes u_{\vartheta(p)}.
\]
Moreover each spectral space $\T(X)_h$ with $h \in \G$ is given by the products of the form
\[
\wt{t}_{p_1}(\xi_{p_1}) \wt{t}_{p_2}(\xi_{p_2})^* \cdots \wt{t}_{p_n}(\xi_{p_n})^* \qfor \vartheta(p_1) \vartheta(p_2)^{-1} \cdots \vartheta(p_n)^{-1} = h,
\]
where we impose that $\wt{t}_{p_i}(\xi_{p_i}) = I$ when $p_i = e_P$ and $h \neq e_\P$.
\end{proposition}

\begin{proof}
The universal property induces a $*$-homomorphism $\wt{\de} \colon \T(X) \to \T(X) \otimes \ca(G)$.
Moreover $\wt{\de}$ is injective with left inverse given by $\id \otimes \chi$.
By construction the fibers $[\T(X)]_g$ contain the generators of $\T(X)$.
By Remark \ref{R:nd cis} and the definition of $\T(X)^+$, this gives the coaction of $G$.
Proposition \ref{P:coa ind} provides the coaction of $\G$.
\end{proof}

%%%%%%%%%%%%%%%%%%%%%%%%%%%%%%%%
\begin{remark}
The Fock space representation of Fowler \cite{Fow02} ensures that $A$, and thus $X$, embeds isometrically in $\T(X)$.
In short, let $\F(X) = \sum^\oplus_{q \in P} X_q$ and for $a \in A$ and $\xi_p \in X_p$ define $(\ol{\pi}, \ol{t}_p)$ by
\[
\ol{\pi}(a) \xi_q = \vphi_q(a) \xi_q
\qand
\ol{t}_p(\xi_p) \xi_q = \xi_p \xi_q
\qforal
\xi_q \in X_q.
\]
Then every $(\ol{\pi}, \ol{t}_p)$ defines a representation of $X_p$ and hence it induces a representation of $\T(X)$.
By taking the compression at the $(e, e)$-entry we see that $\ol{\pi}$, and thus $\ol{t}_p$, is injective.
\end{remark}

%%%%%%%%%%%%%%%%%%%%%%%%%%%%%%%%
\begin{definition}
Let $P$ be a unital subsemigroup of a group $G$ and $X$ be a product system over $P$ with coefficients in $A$.
The \emph{Fock algebra} $\T_\la(X)$ is the C*-algebra generated by the Fock representation $(\ol{\pi}, \ol{t})$.
The \emph{Fock tensor algebra $\T_\la(X)^+$ of $X$} is the subalgebra of $\T_\la(X)$ generated by $A$ and $X$.
\end{definition}

It is shown in \cite[Proposition 4.1]{DKKLL20} that the Fock algebra admits an analogous normal coaction.
Proposition \ref{P:coa ind} yields the next proposition.

%%%%%%%%%%%%%%%%%%%%%%%%%%%%%%%%
\begin{proposition}\label{P:f coaction}
Let $P$ be a unital subsemigroup of a group $G$ and $X$ be a product system over $P$ with coefficients in $A$, and let $\T_\la(X) = \ca(\ol{\pi}, \ol{t})$ be its associated Fock algebra.
If $\vartheta \colon (G,P) \to (\G,\P)$ is a semigroup preserving homomorphism then there is a normal coaction of $\G$ on $\T_\la(X)$ such that
\[
\ol{\de}_\G \colon \T_\la(X) \longrightarrow \T_{\la}(X) \otimes \ca(\G) ; \ol{t}(\xi_p) \mapsto \ol{t}(\xi_p) \otimes u_{\vartheta(p)}.
\]
Moreover for each $h \in \G$ the spectral space $\T_\la(X)_h$ is the closed linear span of the products of the form
\[
\ol{t}_{p_1}(\xi_{p_1}) \ol{t}_{p_2}(\xi_{p_2})^* \cdots \ol{t}_{p_{n-1}}(\xi_{p_{n-1}}) \ol{t}_{p_n}(\xi_{p_n})^* \qfor \vartheta(p_1) \vartheta(p_2)^{-1} \cdots \vartheta(p_n)^{-1} = h,
\]
where we impose that $\ol{t}_{p_i}(\xi_{p_i}) = I$ when $p_i = e_P$ and $h \neq e_\P$.
\end{proposition}

In turn the coaction of $\G$ induces a faithful conditional expectation of the following form.

%%%%%%%%%%%%%%%%%%%%%%%%%%%%%%%%
\begin{proposition}\label{P:Fock ce}
Let $P$ be a unital subsemigroup of a group $G$ and $X$ be a product system over $P$ with coefficients in $A$.
Let $\vartheta \colon (G,P) \to (\G,\P)$ be a semigroup preserving homomorphism.
Then $\T_\la(X)$ admits a faithful conditional expectation $\ol{E}_\P$ such that
\[
\ol{E}_\P(\ol{\psi}_{r,s}(k_{r,s})) = \de_{\vartheta(r), \vartheta(s)} \ol{\psi}_{r,s}(k_{r,s}) \foral k_{r,s} \in \K(X_s,X_r).
\]
\end{proposition}

\begin{proof}
Let $\ol{\de}_{\G} \colon \T_\la(X) \to \T_\la(X) \otimes \ca(\G)$ be the normal coaction and let $\om_{e,e}$ be the faithful conditional expectation on $\ca_\la(\G)$.
Then $\T_\la(X)$ admits the faithful conditional expectation
\[
\ol{E}_\P := (\id \otimes \om_{e,e}) (\id \otimes \la) \ol{\de}_G.
\]
On the other hand for $h \in \P$ let $Y_h := \sum^\oplus_{\vartheta(p) = h} X_h$ and let the projections $Q_{h} \colon \F(X) \to Y_h$.
We will show that
\[
\ol{E}_\P(\cdot) = \sum_{h \in \P} Q_h \cdot Q_h.
\]
It suffices to apply on the spanning elements of the form
\[
f
:=
\ol{t}_{p_1}(\xi_{p_1}) \ol{t}_{p_2}(\xi_{p_2})^* \cdots \ol{t}_{p_{n-1}}(\xi_{p_{n-1}}) \ol{t}_{p_n}(\xi_{p_n})^*
\]
where we impose that $\ol{t}_{p_i}(\xi_{p_i}) = I$ when $p_i = e_P$.
For $p \in P$ we directly compute
\[
\ol{E}_\P(f)
=
\begin{cases}
f \xi_p & \text{if } \vartheta(p_1^{-1}p_2 \cdots p_{n-1}^{-1} p_n) = e_{\G}, \\
0 & \text{ otherwise}.
\end{cases}
\]
If $f \xi_p \neq 0$ then it is in some $X_r$ with $r = p_1^{-1} p_2 \cdots p_{n-1}^{-1} p_n p$ which gives $\vartheta(r) = \vartheta(p)$.
On the other hand we have that
\[
\left( \sum_{h \in \P} Q_h f Q_h \right) \xi_p
=
\begin{cases}
f \xi_p & \text{if } \vartheta(p_1^{-1}p_2 \cdots p_{n-1}^{-1} p_n p) = \vartheta(p), \\
0 & \text{ otherwise}.
\end{cases}
\]
We have that $\vartheta(p_1^{-1}p_2 \cdots p_{n-1}^{-1} p_n p) = \vartheta(p)$ if and only if $\vartheta(p_1^{-1}p_2 \cdots p_{n-1}^{-1} p_n) = e_{\G}$ and so
\[
\ol{E}_\P(f) = \sum_{h \in \P} Q_h f Q_{h}.
\]

For the second part let $r, s \in P$ and $\xi_p \in Y_h$ so that $\vartheta(p) = h$.
Then we directly compute
\begin{align*}
\ol{E}_\P( \ol{\psi}_{r,s}(k_{r,s}) ) \xi_p
& =
Q_{h} \ol{\psi}_{r,s}(k_{r,s}) \xi_p
=
\begin{cases}
\ol{\psi}_{r,s}(k_{r,s}) \xi_p & \text{if } p = s s', \vartheta(p) = \vartheta(r s'),\\
0 & \text{ otherwise},
\end{cases} \\
& =
\de_{\vartheta(r), \vartheta(s)} \ol{\psi}_{r,s}(k_{r,s}) \xi_p,
\end{align*}
where we used that $\vartheta$ is a group homomorphism and so $\vartheta(s) \vartheta(s') = \vartheta(p) = \vartheta(r) \vartheta(s')$.
As $p \in P$ is arbitrary the proof is complete.
\end{proof}

%%%%%%%%%%%%%%%%%%%%%%%%%%%%%%%%
\subsection{Covariance algebras and Cuntz-Nica-Pimsner algebras}
%%%%%%%%%%%%%%%%%%%%%%%%%%%%%%%%

Let us review Sehnem's strong covariance relations from \cite{Seh18}.
We will be using a description presented in \cite{DKKLL20}.
Let $P$ be a unital subsemigroup of a group $G$.
For a finite set $F \subseteq G$ let
\[
K_F := \bigcap_{g \in F} gP.
\]
For $r \in P$ and $g \in F$ define the ideal of $A$ given by
\[
I_{r^{-1} K_{\{r,g\}}} :=
\begin{cases}
\bigcap\limits_{t \in K_{\{r,g\}}} \ker \vphi_{r^{-1}t} & \text{if } K_{\{r,g\}} \neq \mt \text{ and } r \notin K_{\{r,g\}},\\
A & \text{otherwise}.
\end{cases}
\]
Then let
\[
I_{r^{-1} (r \vee F)} := \bigcap_{g \in F} I_{r^{-1} K_{\{r,g\}}},
\]
and let the C*-correspondences
\[
X_F := \oplus_{r \in P} X_r I_{r^{-1} (r \vee F)}
\qand
X_F^+ := \oplus_{g \in G} X_{gF}.
\]
For every $p \in P$ define the representation $(\pi_F, t_{F, p})$ to $X_F^+$ given by
\[
t_{F, p}(\xi_p) (\eta_r) = u_{p,r}(\xi_p \otimes \eta_r) \in X_{pr} I_{(pr)^{-1}(pr \vee pF)}, \foral \eta_r \in X_r I_{r^{-1} (r \vee F)}.
\]
It is well-defined as $I_{r^{-1}(r \vee F)} = I_{(pr)^{-1}(pr \vee pF)}$ for all $r \in P$, and $I_{r^{-1}(r \vee F)} = I_{(s^{-1}r)^{-1}(s^{-1}r \vee s^{-1}F)}$ for all $r \in sP$.
This provides a representation $(\pi_F, t_F)$ of $X$ on $\L(X_F^+)$ that integrates to a representation
\[
\Phi_F \colon \T(X) \longrightarrow \L(X_F^+).
\]
Now let the projections $Q_{g,F} \colon X_F^+ \to X_{gF}$ and define 
\[
\nor{f}_F := \nor{Q_{e,F} \Phi_F(f) Q_{e,F}} \foral f \in \left[ \T(X) \right]_e.
\]
In particular we have that 
\[
t_{F,p}(\xi_p) Q_{g, F} = Q_{pg, F} t_{F,p}(\xi_p)
\qand
t_{F,p}(\xi_p)^* Q_{g,F} = Q_{p^{-1}g, F} t_{F,p}(\xi_p)^*.
\]
and so $Q_{e,F}$ is reducing for the fixed point algebra $[\T(X)]_e$ under $\Phi_F$.

%%%%%%%%%%%%%%%%%%%%%%%%%%%%%%%%
\begin{definition}\cite[Definition 3.2]{Seh18}
A Toeplitz representation is called \emph{strongly covariant} if it vanishes on the ideal $\I_e \lhd \left[ \T(X) \right]_e$ given by
\[
\I_e := \{f \in \left[ \T(X) \right]_e \mid \lim_F \nor{f}_F = 0\}
\]
where the limit is taken with respect to the partial order induced by inclusion on finite sets of $P$.
The universal C*-algebra with respect to the strongly covariant representations of $X$ is denoted by $A \times_{X} P$.
\end{definition}

That is $A \times_{X} P$ is the quotient $\T(X)/\I_\infty$ for the ideal $\I_\infty \lhd \T(X)$ of \emph{strong covariance relations} generated by $\I_e$.
One of the important points of Sehnem's theory is that $A \hookrightarrow A \times_{X} P$ faithfully.
As a quotient by an induced ideal of $\T(X)$, the C*-algebra $A \times_{X} P$ inherits the coaction of $G$.
The following is the main theorem of \cite{Seh18}.

%%%%%%%%%%%%%%%%%%%%%%%%%%%%%%%%
\begin{theorem}\cite[Theorem 3.10]{Seh18}
Let $P$ be a unital subsemigroup of a group $G$ and $X$ be a product system over $P$ with coefficients in $A$.
Then a $*$-homomorphism of $A \times_{X} P$ is faithful on $A$ if and only if it is faithful on the fixed point algebra $\left[ A \times_{X} P \right]_e$.
\end{theorem}

Due to the grading $A \times_{X} P$ is the maximal C*-algebra of a Fell bundle over $G$.
We consider two reduced versions.

%%%%%%%%%%%%%%%%%%%%%%%%%%%%%%%%
\begin{definition}
Let $P$ be a unital subsemigroup of a group $G$ and $X$ be a product system over $P$ with coefficients in $A$.
We write $A \times_{X, \la} P$ for the reduced C*-algebra of the Fell bundle in $A \times_{X} P$.
If $q \colon \T(X) \to \T_\la(X)$ is the canonical $*$-epimorphism, then we write $q_{\scv}(\T_\la(X))$ for the quotient of $\T_\la(X)$ by the ideal $q(\I_\infty)$.
\end{definition}

%%%%%%%%%%%%%%%%%%%%%%%%%%%%%%%%
\begin{remark} \label{R:notation change}
The notation $\S\C X$ is used in \cite{DKKLL20} to denote the $G$-Fell bundle inside $A \times_{X}^G P$.
Therefore we have two ways of writing the related C*-algebras in the sense that
\[
A \times_{X} P = \ca(\S\C X) \qand A \times_{X, \la} P = \ca_\la(\S\C X).
\]
Sehnem shows in \cite[Lemma 3.9]{Seh18} that the strong covariance relations do not depend on the group embedding in the following sense.
Suppose that $P$ admits two group embeddings $i_G \colon P \to G$ and $i_H \colon P \to H$ and write $\ca_{\max}(\S\C_G X) = \ca(\pi^G, t^G)$ and $\ca_{\max}(\S\C_H X) = \ca(\pi^H, t^H)$.
Then there exists a $*$-isomoprihsm
\[
\ca_{\max}(\S\C_G X) \longrightarrow \ca_{\max}(\S\C_H X); t_{i_G(p)}^G(\xi_{i_G(p)}) \mapsto t_{i_H(p)}^H(\xi_{i_H(p)}).
\]

The $*$-isomorphism between $\ca_{\max}(\S\C_G X)$ and $\ca_{\max}(\S\C_H X)$ descends to a $*$-isomorphism that fixes $X$ at the reduced level, as well, and thus $A \times_{X, \la} P$ does not depend on the group embedding either.
Indeed suppose that $G$ is the enveloping group of $P$ and thus there exists a group homomoprhism $\ga \colon G \to H$ that is injective on $P$.
We then have that there is a $*$-homomorphism between the maximal C*-algebras induced by the $G$-Fell bundle and the $H$-Fell bundle on Sehnem's covariance algebra.
Sehnem's result \cite[Lemma 3.9]{Seh18} is that this $*$-homomorphism is faithful.
By Fell bundle theory we then get a canonical $*$-epimorphism
\[
\ca_{\la}(\S\C_G X) \longrightarrow \ca_{\la}(\S\C_H X)
\]
that fixes $X$.
Hence by construction it intertwines the normal faithful conditional expectations.
Their fixed point algebras are $*$-isomorphic to the fixed point algebras in the maximal C*-algebras and these are $*$-isomorphic by \cite[Lemma 3.9]{Seh18}.
Thus the $*$-epimorphism on the reduced models is faithful.
\end{remark}

We see that the representations $\Phi_F$ used to define the strong covariance relations are sub-representations of $\ol{\de}_{G, \la} \colon \T_\la(X) \to \T_\la(X) \otimes \ca_\la(G)$ for $\ol{\de}_{G, \la} = (\id \otimes \la) \ol{\de}_{G}$ where $\ol{\de}_G$ is the normal coaction on the Fock representation.
Indeed we can identify
\[
X_F^+ = \oplus_{g \in G} \oplus_{r \in P} X_r I_{r^{-1}(r \vee gF)}
\]
with a submodule of $\F X \otimes \ell^2(G)$ through the isometry given by
\[
X_{r} I_{r^{-1}(r \vee gF)} \ni \eta_r \mapsto \eta_r \otimes \de_g \in X_r \otimes \ell^2(G).
\]
Recall here that $\F X \otimes \ell^2(G)$ is the exterior tensor product of two modules (seeing $\ell^2(G)$ as a module over $\bC$), and there is a faithful $*$-homomorphism
\[
\T_\la(X) \otimes \ca_\la(G) \subseteq \L(\F X) \otimes \B(\ell^2(G)) \hookrightarrow \L(\F X \otimes \ell^2(G)).
\]
We then see that
\[
t_{F,p}(\xi_p) = (\ol{t}_p(\xi_p) \otimes \la_p)|_{X_F^+} = \ol{\de}_{G, \la}(\ol{t}_p(\xi_p))|_{X_F^+}
\foral
p \in P,
\]
and likewise for their adjoints.
Thus $X_F^+$ is reducing under $\ol{\de}_{G, \la}(\T_\la(X))$.
Recall also that $X_F$ is reducing for $[\T(X)]_e$ as the range of the projection $Q_{e,F}$ and so we obtain the representation
\[
\bigoplus\limits_{\textup{fin } F \subseteq G} \Phi_F(\cdot)|_{X_F} \colon [\T(X)]_e \longrightarrow [\T_\la(X)]_e \longrightarrow \prod\limits_{\text{fin } F \subseteq G} \L(X_F).
\]
In particular, by definition we have for an $f \in \T(X)$ that
\[
f \in \I_e \qiff \bigoplus\limits_{\textup{fin } F \subseteq G} \Phi_F(f)|_{X_F} \in c_0(\L(X_F) \mid \textup{fin } F \subseteq G).
\]
By definition we then get that the following diagram 
\[
\xymatrix{
%line 1
[\T(X)]_e \ar[d] \ar[rr] & & [\T_\la(X)]_e \ar[rr] \ar[d]^{q_{\scv}} & & 
\prod\limits_{\text{fin } F \subseteq G} \L(X_F) \ar[d] \\
%line 2
[A \times_{X} P]_e \ar[rr] & & [q_{\scv}(\T_\la(X))]_e \ar[rr] & & 
\quo{\prod\limits_{\text{fin } F \subseteq G} \L(X_F)}{c_0(\L(X_F) \mid \textup{fin } F \subseteq G)}
}
\]
is commutative.
Consequently the $e$-graded $*$-algebraic relations in $\T_\la(X)$ and $A \times_{X} P$ induce relations in $q_{\scv}(\T_\la(X))$.
In particular, since by \cite[Proposition 3.5]{Seh18} $A$ is represented faithfully in the bottom right corner of the above diagram, we obtain the following corollary.

%%%%%%%%%%%%%%%%%%%%%%%%%%%%%%%%
\begin{corollary} \label{C:A in red} \cite[Corollary 5.5]{DKKLL20}
Let $P$ be a unital subsemigroup of a group $G$ and $X$ be a product system over $P$ with coefficients in $A$.
Then $A \hookrightarrow q_{\scv}(\T_\la(X))$.
Moreover a $*$-homomorphism of $q_{\scv}(\T_\la(X))$ is faithful on $A$ if and only if it is faithful on $\left[ q_{\scv}(\T_\la(X)) \right]_e$.
Likewise for the reduced C*-algebra $A \times_{X, \la} P$.
\end{corollary}

%%%%%%%%%%%%%%%%%%%%%%%%%%%%%%%%
\section{Compactly aligned product systems over weak right LCM-inclusions} \label{S:caps lcm}
%%%%%%%%%%%%%%%%%%%%%%%%%%%%%%%%

%%%%%%%%%%%%%%%%%%%%%%%%%%%%%%%%
\subsection{Weak right LCM-inclusions}
%%%%%%%%%%%%%%%%%%%%%%%%%%%%%%%%

A semigroup $P$ is said to be a \emph{right LCM-semigroup} if it is left-cancellative and satisfies \emph{Clifford's condition} \cite{Law12, Nor14}: 
\begin{center}
for every $p, q \in P$ with $p P \cap q P \neq \mt$ there exists a $w \in P$ such that $p P \cap q P = w P$.
\end{center}
In other words, if $p, q \in P$ have a right common multiple then they have a right Least Common Multiple.
As we always see a semigroup $P$ inside a group $G$, it follows that $P$ is by default cancellative, and we will refer to $(G,P)$ simply as a \emph{weak right LCM-inclusion}.
We use the adjective ``weak'' here to emphasize that we do not assume that the Least Common Multiple property holds for \emph{all} elements in $G$.

It is clear that $w$ is a right Least Common Multiple for $p, q$ if and only if $w x$ is a right LCM of $p, q$ for every $x \in P^*$.
A \emph{weak quasi-lattice} $(G,P)$ is a weak right LCM-inclusion with $P \cap P^{-1} = \{e\}$, i.e., when least common multiples are unique (whenever they exist).

%%%%%%%%%%%%%%%%%%%%%%%%%%%%%%%%
\begin{definition}
Let $(G,P)$ be a right weak LCM-inclusion. 
A finite set $F$ is said to be \emph{$\vee$-closed} if for any $p,q \in F$ with $p P \cap q P \neq \mt$ there exists a unique $w \in F$ such that $p P \cap q P = w P$.
\end{definition}

Equivalently, a finite $F \subseteq P$ is $\vee$-closed if and only if the familiar relation
\[
p \leq q \Leftrightarrow q^{-1}p \in  P
\]
defines a partial order on $F$. 
In particular, if $F$ is $\vee$-closed, then $p P \neq q P$ for any $p,q\in F$ with $p\neq q$. 
Furthermore, any $\vee$-closed set admits maximal and minimal elements.
Our terminology here regarding $\vee$-closed sets extends the familiar one from the case where $(G,P)$ is a weak quasi-lattice order.  
There is an alternative way for describing $\vee$-closed sets in the context of weak right LCM-inclusions.
Given a finite subset $F \subseteq P$ we write
\[
\I(F) := \{p P \mid p \in F\}
\]
for the set of principal ideals defined by $F$.
It then follows that $F$ is $\vee$-closed if and only if $\I(F)$ is closed under intersections and the partial order defined on $\I(F)$ by set  theoretic inclusion lifts to a partial order on $F$.

Let  $F\subseteq P$  be a finite  set  so that $\I(F)$ is closed under intersections.
From such  a set $F$ we can produce a $\vee$-closed subset $F^{\vee}$ such that $\I(F) = \I(F^\vee)$ by choosing a minimal set of distinct representatives for the principal ideals. This process does not  produce a unique $F^{\vee}$ in general.

%%%%%%%%%%%%%%%%%%%%%%%%%%%%%%%%
\subsection{Nica-covariant representations}
%%%%%%%%%%%%%%%%%%%%%%%%%%%%%%%%

Following Fowler's work \cite{Fow02}, Brownlowe, Larsen and Stammeier \cite{BLS18}, and Kwa\'{s}niewski and Larsen \cite{KL19b} considered product systems of right LCM-semigroups.

%%%%%%%%%%%%%%%%%%%%%%%%%%%%%%%%
\begin{definition}
A product system $X$ over a weak right LCM-semigroup $P$ with coefficients in $A$ is called \emph{compactly aligned} if for $p, q \in P$ with $p P \cap q P = w P$ we have that
\[
i_{p}^{w}(S) i_{q}^{w}(T) \in \K X_{w} \text{ whenever } S \in \K X_{p}, T \in \K X_{q}.
\]
\end{definition}

A note is in order for clarifying that this is independent of the choice of $w$.
Recall that if $w'$ is a right LCM of $p, q$ then $w' = wx$ for some $x \in P^*$.
Since $\L X_{w} \simeq \L X_{wx}$ we have that $i_{p}^{w}(S) i_{q}^{w}(T) \in \K X_w$ if and only if $i_{p}^{wx}(S) i_{q}^{wx}(T) = i^{wx}_w (i_{p}^{w}(S) i_{q}^{w}(T)) \in \K X_{wx}$ for all $x \in P^*$.

%%%%%%%%%%%%%%%%%%%%%%%%%%%%%%%%%
%\begin{details}
%Suppose that $wP = w'P$.
%Then there are $x, y \in P$ such that $wx= w'$ and $w=w'y$.
%Thus $wxy = w'y = w$ and so $xy = e$.
%Likewise $yx = e$ and so $x \in P^*$.
%\end{details}

%%%%%%%%%%%%%%%%%%%%%%%%%%%%%%%%
\begin{definition}
Let $X$ be a compactly aligned product system over a right LCM-semigroup $P$ with coefficients in $A$.
A \emph{Nica-covariant representation $(\pi, t)$} is a Toeplitz representation of $A$ that in addition satisfies the \emph{Nica-covariance} condition: for $S \in \K X_{p}$ and $T \in \K X_{q}$ we have that
\[
\psi_{p}(S) \psi_{q}(T) = 
\begin{cases}
\psi_{w} (i_{p}^{w}(S) i_{q}^{w}(T)) & \text{ if } p P \cap q P = w P, \\
0 & \text{ otherwise}.
\end{cases}
\]
The \emph{Toeplitz-Nica-Pimsner algebra $\N\T(X)$ of $X$} is the universal C*-algebra generated by $A$ and $X$ with respect to the representations of $X$.
The \emph{Toeplitz-Nica-Pimsner tensor algebra $\N\T(X)^+$ of $X$} is the subalgebra of $\T(X)$ generated by $A$ and $X$.
\end{definition}

%%%%%%%%%%%%%%%%%%%%%%%%%%%%%%%%
\begin{remark}
As noted in \cite{DKKLL20}, the definition of Nica-covariance requires that the right hand side is independent of the choice of the least common multiple, i.e., if $pP \cap qP = wP$ and $x \in P^*$ then
\[
\psi_{w} (i_{p}^{w}(S) i_{q}^{w}(T))
=
\psi_{wx} (i_{p}^{wx}(S) i_{q}^{wx}(T))
\foral
S \in \K X_p, T \in \K X_q.
\]
This is verified in \cite[Proposition 2.4]{DKKLL20} (see Proposition \ref{P:star inv LCM} herein) and completes the definition of Nica-covariance in \cite{KL19b}.
\end{remark}

%%%%%%%%%%%%%%%%%%%%%%%%%%%%%%%%
\begin{remark}
By definition $\N\T(X)$ is a quotient of $\T(X)$ by an ideal generated by a subspace of $\left[ \T(X) \right]_e$.
Even though $\N\T(X) = \T(X)$ when $P = \bZ_+$, this is not the case even when $P = \bZ_+^n$.
Dor-On and Katsoulis provide a counterexample to this effect in \cite[Example 5.2]{DK18a}.
The same example further shows that $\T(X)^+$ is not completely isometric to $\N\T(X)^+$.
\end{remark}

Under the assumption of compact alignment, one can check that the Fock representation is automatically Nica-covariant.
Thus $\N\T(X)$ is non-trivial.
As $\N\T(X)$ is a quotient of $\T(X)$ by an induced ideal, by \cite[Proposition A.1]{CLSV11} the non-degenerate and faithful coaction of $\T(X)$ descends canonically to one on $\N\T(X)$.
Alternatively one may use the arguments of the proof of Proposition \ref{P:t coaction} for the Toeplitz-Nica-Pimsner tensor algebra to deduce the following.

%%%%%%%%%%%%%%%%%%%%%%%%%%%%%%%%
\begin{proposition}\label{P:nt coa}
Let $(G,P)$ be a weak right LCM-inclusion and $X$ be a compactly aligned product system over $P$ with coefficients in $A$.
Suppose that $(\wh{\pi}, \wh{t})$ is a faithful representation of $\N\T(X)$.
Then the canonical $*$-homomorphism 
\[
\wh{\de} \colon \N\T(X) \longrightarrow \N\T(X) \otimes \ca(G) ; \wh{t}(\xi_p) \mapsto \wh{t}(\xi_p) \otimes u_p
\]
defines a coaction of $G$ on $\N\T(X)$.
\end{proposition}

We have refrained from describing the spectral spaces for the coaction on $\N\T(X)$ because of the following additional property of Nica-covariant representations.
Let $(\pi,t)$ be a Nica-covariant representation of $X$.
We compute
\[
t_p(X_p)^* t_p(X_p) \cdot t_p(\xi_p)^* t_q(\xi_q) \cdot t_q(X_q)^* t_q(X_q)
\subseteq
\big[ t_p(X_p)^* \psi_p(\K X_p) \psi_q(\K X_q) t_q(X_q) \big].
\]
Next take a limit by c.a.i.'s in $\big[ t_p(X_p)^* t_p(X_p) \big]$ and in $\big[ t_q(X_q)^* t_q(X_q) \big]$, and derive that
\[
t_p(\xi_p)^* t_q(\xi_q) \in \big[ t_{p'}(X_{p'}) t_{q'}(X_{q'})^* \big] \qfor w P = pP \cap q P, p' = p^{-1} w, q' = q^{-1} w,
\]
and
\[
t_p(\xi_p)^* t_q(\xi_q) = 0 \qfor pP \cap qP = \mt.
\]
Hence the C*-algebra $\ca(\pi,t)$ generated by $\pi(A)$ and $t_p(X_p)$ admits a \emph{Wick ordering} in the sense that
\[
\ca(\pi,t) = \ol{\spn}\{t_p(\xi_p) t_q(\xi_q)^* \mid \xi_p \in X_p, \xi_q \in X_q \textup{ and } p,q \in P\}.
\]
In particular if $\N\T(X) = \ca(\wh{\pi}, \wh{t})$ then the spectral spaces that only matter are of the form
\[
\N\T(X)_{p q^{-1}} = \ol{\spn}\{ \wh{t}_p(\xi_p) \wh{t}_q(\xi_q)^* \mid \xi_p \in X_p, \xi_q \in X_q \},
\]
that is, only for $g \in G$ of the form $g = p q^{-1}$ for some $p, q \in P$.

The following proposition gives a direct criterion to check compact alignment.

%%%%%%%%%%%%%%%%%%%%%%%%%%%%%%%%
\begin{proposition}\label{P:com al}
Let $(G,P)$ be a weak right LCM-inclusion and let $X = \{X_p\}_{p \in P}$ be a product system over $X_e = A$. 
Let $(\pi, t)$ be an injective representation $X$.
Then $X$ is compactly aligned, if and only if for all $p,q \in P$ we have that
\[
t_p(X_p)^* t_q(X_q) \subseteq [t_{p^{-1} w}(X_{p^{-1}w}) t_{q^{-1}w}(X_{q^{-1}w})^*] \qfor wP = pP \cap qP,
\]
if and only if for all $p,q \in P$ we have that
\[
t_p(X_p) t_p(X_p)^* t_q(X_q) t_q(X_q)^* \subseteq [t_w(X_w) t_w(X_w)^*] = \psi_w(\K X_w) \qfor wp = pP \cap qP,
\]
with the understanding that the left hand sides are the zero space when $p$ and $q$ have no right common multiple.
\end{proposition}

\begin{proof}
The first equivalence follows in the same way with \cite[Proposition 3.2]{Kat20} and it is omitted.
By using that $X_p X_p^* X_p$ is dense in $X_p$ for every $p \in P$, we get the second equivalence.
\end{proof}

Let us now pass to the analysis of the cores of a Nica-covariant representation $(\pi,t)$ of $X$.
For a finite $F \subseteq P$ that is $\vee$-closed we write
\[
B_F := \spn\{ \psi_{p}(k_{p}) \mid k_p \in \K X_p, p \in F\}.
\]
Since $F$ is $\vee$-closed, Nica-covariance implies that $B_F$ is a $*$-subalgebra of $\ca(\pi,t)$.
In \cite[Proposition 2.10]{DKKLL20} we show that every $B_F$ is actually a C*-subalgebra.
Moreover for such an $F$ we write
\[
B_{F \cdot P} := \ol{\spn} \{ \psi_{q}(k_{q}) \mid k_q \in \K X_q, q \in F \cdot P\}.
\]
Likewise this is also a (closed) $*$-subalgebra.
Finally we write
\[
B_{P \setminus \{e\}} 
:= \ol{\spn} \{ \psi_p(k_p) \mid k_p \in \K X_p, e \neq p \in P \}
\qand
B_{P} 
:= \pi(A) + B_{P \setminus \{e\}}.
\]
We see that $B_{P \setminus \{e\}}$ is an ideal in $B_{P}$ and thus the sum $\pi(A) + B_{P \setminus \{e\}}$ is indeed closed.
We refer to these sets as the \emph{cores} of the representation $(\pi,t)$.
In \cite[Proposition 2.11]{DKKLL20} we showed that we can exhaust the cores by using finite $\vee$-closed sets, in the sense that 
\[
B_{P} 
=
\ol{\bigcup \{ B_F \mid F \subseteq P \text{ finite and $\vee$-closed} \}}.
\]

We denote by $\ol{B}_F$ the cores of $\T_\la(X) = \ca(\ol{\pi}, \ol{t})$. 
Recall that $\T_\la(X)$ admits the faithful conditional expectation
\[
\ol{E}_P \colon \T_\la(X) \longrightarrow \ol{B}_{P} ; \ol{t}_p(\xi_p) \ol{t}_q(\xi_q)^* = \de_{p,q} \ol{t}_p(\xi_p) \ol{t}_q(\xi_q)
\]
given by the sum of compressions to the $(r,r)$-entries in $\L(\F X)$ (see Proposition \ref{P:Fock ce}).

The Toeplitz-Nica-Pimsner algebra models the Fock algebra in this context.
A compactly aligned product system $X$ over $P$ with coefficients in $A$ is called \emph{amenable} if the Fock representation is faithful on $\N\T(X)$.
Let us give some equivalent condition for this to happen.

%%%%%%%%%%%%%%%%%%%%%%%%%%%%%%%%
\begin{theorem}\label{T:chara}
Let $(G,P)$ be a weak right LCM-inclusion and $X$ be a compactly aligned product system over $P$ with coefficients in $A$.
The following are equivalent:
\begin{enumerate}
\item The coaction of $G$ on $\N\T(X)$ is normal.
\item The conditional expectation on $\N\T(X)$ is faithful.
\item The Fock representation is faithful on $\N\T(X)$.
\item The representation $\N\T(X) \to \ca(\pi,t) \otimes \ca_\la(P) ; \wt{t}_p(\xi_p) \mapsto t_p(\xi_p) \otimes V_p$ is faithful for any injective Nica-covariant pair $(\pi,t)$.
\end{enumerate}
\end{theorem}

\begin{proof}
By the universal property there exists a canonical $*$-representation
\[
\N\T(X) \longrightarrow \N\T(X) \otimes \ca_\la(G)
\]
that intertwines the conditional expectations.
Thus items (i) and (ii) are equivalent.
For the same reason items (ii) and (iii) are equivalent.

Assuming item (iii) we have to show that the representation $\N\T(X) \to \ca(\pi,t) \otimes \ca_\la(P)$ is faithful on the fixed point algebra.
It suffices to show injectivity on the $F$-boxes for arbitrary $\vee$-closed $F \subseteq P$.
To this end suppose that
\[
\sum_{p \in F} \psi_p(k_p) = 0
\]
for some $k_p \in \K X_p$ and let $p_0$ be minimal so that $k_{p_0} \neq 0$.
Injectivity of $\pi$ then induces that $\psi_{p_0}(k_{p_0}) \neq 0$ as well.
However, if $Q_{p_0} \colon \ell^2(P) \to \bC e_{p_0}$ is the canonical projection, minimality of $p_0$ yields
\[
\psi_{p_0}(k_{p_0}) = I \otimes Q_{p_0} \left( \sum_{p \in F} \psi_p(k_p) \right) I \otimes Q_{p_0} = 0,
\]
which is a contradiction.
This shows that item (iii) implies item (iv).

Since the $*$-representation $\N\T(X) \to \ca(\pi,t) \otimes \ca_\la(P)$ intertwines the conditional expectations, we finally have that item (iv) implies item (i), and the proof is complete.
\end{proof}

On the other hand strongly covariant representations are Nica-covariant (which is expected as Nica-covariance is an $e$-graded relation in $[\T_\la(X)]_e$).
It is proven by Sehnem in \cite[Proposition 4.2]{Seh18} for quasi-lattices, but the same proof passes to right LCM-semigroups as well \cite[Proposition 5.4]{DKKLL20}.
Hence $A \times_{X} P$ is a quotient of $\N\T(X)$.

%%%%%%%%%%%%%%%%%%%%%%%%%%%%%%%%
\begin{proposition} \label{P:sc lcm} \cite[Proposition 5.4]{DKKLL20} \cite[Proposition 4.2]{Seh18}
Let $X$ be a compactly aligned product system over a right LCM-semigroup $P$ with coefficients in $A$.
Let $\psi_{F, p} \colon \K X_p \to \L(X_F^+)$ be the induced representations from $(\pi_F, t_{F, p})$. 
A representation $(\pi,t)$ of $X$ is strongly covariant if and only if it is Nica-covariant and it satisfies
\[
\sum_{p \in F} \psi_{F, p}(k_p) |_{X_F} = 0 \Longrightarrow \sum_{p \in F} \psi_p(k_p) = 0
\]
for any finite $F \subseteq P$ and $k_p \in \K X_p$.
\end{proposition}

Carlsen-Larsen-Sims-Vittadello \cite{CLSV11} explored the idea of finding the co-universal C*-algebra with respect to injective equivariant Nica-covariant representations of $X$.
By using the C*-envelope machinery we can prove that this object always exists, thus completing the co-universal aspect of their programme at the more general context of right weak LCM-inclusions.

%%%%%%%%%%%%%%%%%%%%%%%%%%%%%%%%
\begin{definition}
Let $(G,P)$ be a weak right LCM-inclusion and $X$ be a compactly aligned product system over $P$ with coefficients in $A$.
We say that a representation $(\pi,t)$ of $X$ is \emph{co-universal for $\N\T(X)$} if
\begin{enumerate}
\item $\pi$ is faithful.
\item $\ca(\pi,t)$ is an equivariant quotient of $\N\T(X)$.
\item $(\pi,t)$ factors through any other equivariant quotient of $\N\T(X)$ that is injective on $A$.
\end{enumerate}
\end{definition}

Of course the C*-algebras of co-universal representations are automatically $*$-isomorphic by an equivariant homomorphism.
In \cite{DKKLL20} we proved that the equivariant representation
\[
\N\T(X) \longrightarrow \cenv(\T_\la(X)^+, G, \ol{\de}_G)
\]
that is given by the diagram
\[
\xymatrix{
\N\T(X) \ar@{.>}[rr]  \ar[dr]& & \cenv(\T_\la(X)^+, G, \ol{\de}_G) \\
& \T_\la(X) \ar[ur] &
}
\]
is co-universal.
Let us review the main arguments and see what more we can obtain.

%%%%%%%%%%%%%%%%%%%%%%%%%%%%%%%%
\begin{proposition}\label{P:P coa B} \cite[Proposition 4.4]{DKKLL20}
Let $(G,P)$ be a weak right LCM-inclusion and $X$ be a compactly aligned product system over $P$ with coefficients in $A$.
Let $\Phi \colon \T_\la(X) \to B$ be a $*$-representation such that $\Phi|_{\ol{\pi}(A)}$ is faithful.
Then there exists a faithful $*$-homomorphism
\[
\T_\la(X) \longrightarrow B \otimes \ca_\la(P) ; \ol{t}_p(\xi_p) \mapsto \Phi \ol{t}_p(\xi_p) \otimes V_p.
\]
\end{proposition}

As a consequence the injective equivariant representations on product systems generate C*-covers for the cosystem $(\T_\la(X)^+, G, \ol{\de}_G)$.

%%%%%%%%%%%%%%%%%%%%%%%%%%%%%%%%
\begin{proposition}\label{P:P cover} \cite[Proposition 4.5]{DKKLL20}
Let $(G,P)$ be a weak right LCM-inclusion and $X$ be a compactly aligned product system over $P$ with coefficients in $A$.
Let $\Phi \colon \T_\la(X) \to B$ be an equivariant $*$-epimorphism such that $\Phi|_{\ol{\pi}(A)}$ is faithful.
Then $B$ is a C*-cover for the cosystem $(\T_\la(X)^+, G, \ol{\de}_G)$.
\end{proposition}

Another consequence of Proposition \ref{P:P coa B} provides a generalization of the Extension Theorem of \cite{KR16}. 
It essentially allows us to recognize a Fock tensor algebra by the presence of a coaction.

%%%%%%%%%%%%%%%%%%%%%%%%%%%%%%%%
\begin{theorem}[Extension Theorem] \label{T:extension}
Let $(G,P)$ be a weak right LCM-inclusion and $X$ be a compactly aligned product system over $P$ with coefficients in $A$.
Let ${\Phi} \colon \T_\la(X) \to B$ be a representation of $X$ and set
\[
\A := \ol{\alg}\{ \Phi \ol{\pi}(A), \Phi\ol{t}_p(X_p) \mid p \in P \}.
\]
Then the following are equivalent:
\begin{enumerate}
\item ${\Phi}|_{\T_\la(X)^+}$ is completely isometric.
\item There exists a completely contractive map 
\[
\A \longrightarrow B \otimes \ca(G) ; \Phi \ol{t}_p(\xi_p) \mapsto \Phi \ol{t}_p(\xi_p) \otimes u_p.
\]
\item There exists a completely contractive map 
\[
\A \longrightarrow B \otimes \ca_\la(G) ; \Phi \ol{t}_p(\xi_p) \mapsto \Phi \ol{t}_p(\xi_p) \otimes \la_p.
\]
\item There exists a completely contractive map 
\[
\A \longrightarrow B \otimes \ca_\la(P) ; \Phi \ol{t}_p(\xi_p) \mapsto \Phi \ol{t}_p(\xi_p) \otimes V_p.
\]
\end{enumerate}
\end{theorem}

\begin{proof}
In Figure \ref{F:diag} we have a diagram of completely contractive representations induced by Proposition \ref{P:P coa B} and Proposition \ref{P:P cover}.
If any of the items holds then it makes the representation of $\T_\la(X)^+$ to $\A$ completely isometric and the proof is complete.
\end{proof}

%%%%%%%%%%%%%%%%%%%%%%%%%%%%%%%%
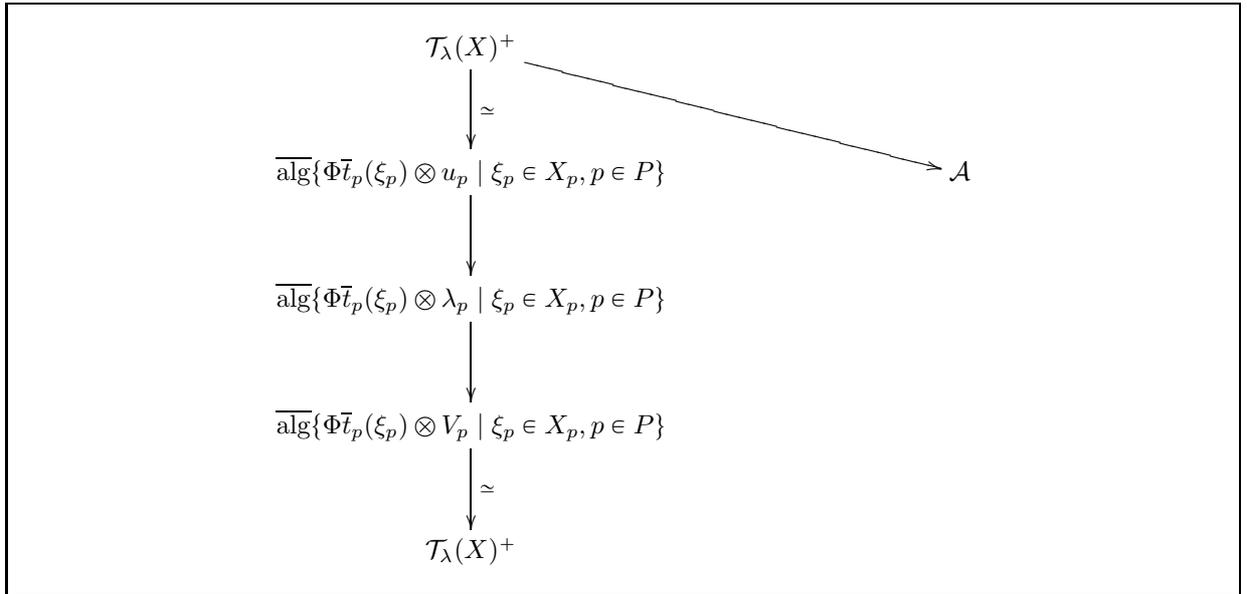
\begin{figure}[h]
{\small
\[
\xymatrix@R=30pt@C=100pt{
\T_\la(X)^+ \ar[d]^{\simeq} \ar[rd] & \\
\ol{\alg} \{ \Phi \ol{t}_p(\xi_p) \otimes u_p \mid \xi_p \in X_p, p \in P\} \ar[d] & \A \\
\ol{\alg} \{ \Phi \ol{t}_p(\xi_p) \otimes \la_p \mid \xi_p \in X_p, p \in P\} \ar[d] & \\ 
\ol{\alg} \{ \Phi \ol{t}_p(\xi_p) \otimes V_p \mid \xi_p \in X_p, p \in P\} \ar[d]^{\simeq} & \\
\T_\la(X)^+ &
}
\]
}
\caption{Diagram of completely positive maps fixing the nonselfadjoint part}
\label{F:diag}
\end{figure}

We now come to the last part of \cite{DKKLL20} that connects reduced C*-algebras with the C*-envelope.
By Corollary \ref{C:A in red} and Proposition \ref{P:P cover} we get a canonical $*$-epimorphism
\[
q_{\scv}(\T_\la(X)) \longrightarrow A \times_{X, \la} P \simeq \cenv(\T_\la(X)^+, G, \ol{\de}_G).
\]
The reader is also reminded here of the notation used here and in \cite{DKKLL20} as explained in Remark \ref{R:notation change}.
The same remark asserts that the C*-envelope of the cosystem is independent of the group embedding \emph{in this setting}.

%%%%%%%%%%%%%%%%%%%%%%%%%%%%%%%%
\begin{theorem} \label{T:co-univ} \cite[Theorem 4.9, Theorem 5.3 and Corollary 5.6]{DKKLL20}
Let $(G,P)$ be a weak right LCM-inclusion and $X$ be a compactly aligned product system over $P$ with coefficients in $A$.
Then the equivariant $*$-epimorphism
\[
\N\T(X) \longrightarrow \cenv(\T_\la(X)^+, G, \ol{\de}_G)
\]
is co-universal.
Moreover we have an equivariant $*$-isomorphism
\[
\cenv(\T_\la(X)^+, G, \ol{\de}_G) \simeq A \times_{X, \la} P.
\]
The equivariant $*$-epimorphism
\[
q_{\scv}(\T_\la(X)) \longrightarrow A \times_{X, \la} P \simeq \cenv(\T_\la(X)^+, G, \ol{\de}_G)
\]
is faithful if and only if the coaction of $G$ on $q_{\scv}(\T_\la(X))$ is normal.
\end{theorem}

%%%%%%%%%%%%%%%%%%%%%%%%%%%%%%%%
\section{Controlled maps} \label{S:ce}
%%%%%%%%%%%%%%%%%%%%%%%%%%%%%%%%

Let $\vartheta \colon (G,P) \to (\G, \P)$ be a semigroup preserving homomorphism between weak right LCM-inclusions and let $X$ be a compactly aligned product system over $P$ with coefficients in $A$. 
By Proposition \ref{P:coa ind} the Toeplitz algebra admits a $\G$-grading that contains the $G$-grading, and the same is true for the fixed point algebras.
Of course this may be useless; for example the $\vartheta$-fixed point algebra for the map for $\vartheta \colon G \to \{e\}$ is the entire C*-algebra.
Nevertheless more can be obtained for weak right LCM-inclusions as long as we impose axioms that control the map.
The following extends the controlled maps on quasi-lattice ordered groups from \cite{LR96}, see also \cite{Fow02} and \cite{CL07}, to the context of weak right LCM-inclusions.

%%%%%%%%%%%%%%%%%%%%%%%%%%%%%%%%
\begin{definition}\label{D:control}
A \emph{controlled map} $\vartheta \colon (G,P) \to (\G,\P)$ between weak right LCM-inclusions is a semigroup preserving homomorphism such that:
\begin{enumerate}
\item[(A1)] If $p P \cap q P \neq \mt$ then $\vartheta(p) \P \cap \vartheta(q) \P = \vartheta(p P \cap q P) \P$.
\item[(A2)] If $p P \cap q P \neq \mt$ and $\vartheta(p) = \vartheta(q)$ then $p=q$.
\end{enumerate}
\end{definition}

It is worth pointing out that in the case where $P = G$ then there is only one right ideal (generated by the identity).
Therefore a controlled map in this case is simply an injective group homomorphism due to (A2).

%%%%%%%%%%%%%%%%%%%%%%%%%%%%%%%%
\begin{remark}
It is clear that (A1) is equivalent to having $\vartheta(p) \P \cap \vartheta(q) \P = \vartheta(w) \P$ whenever $p P \cap q P = w P$.
Moreover, because of (A2) we have that $\vartheta^{-1}(e_\G) \cap P = \{e_G\}$.
Indeed as $\vartheta$ is a group homomorphism we have that $\vartheta(e_G) = e_\G$.
Now if $\vartheta(p) = e_\G$ for some $p \in P$, then since $p P \cap e_G P = p P \neq \mt$ we get by (A1) that $p = e_G$.
This extra generality is crucial when we wish to consider, the generalized length function given by abelianization on the free monoid $\mathbb F_n^+$ \cite{LR96}, and, more generally,  on Artin monoids of rectangular type \cite{CL02}.
\end{remark}

%%%%%%%%%%%%%%%%%%%%%%%%%%%%%%%%
\begin{remark}
A similar type of maps appear in \cite{Cri99} and \cite[Section 3]{BLS18}.
However the maps therein satisfy the stronger requirement that $\vartheta(p) \P \cap \vartheta(q) \P = \vartheta(p P \cap q P) \P$ for all $p,q \in P$.
This means that $p,q \in P$ have a right LCM if and only if so do $\vartheta(p), \vartheta(q) \in \P$.
In our Definition \ref{D:control}, the condition (A1) allows the possibility that $\vartheta(p), \vartheta(q)$ have a right LCM in $\P$ even when $p P \cap q P = \mt$.
\end{remark}

We will investigate the impact of the existence of a controlled map on Nica-covariant representations.
Henceforth fix a controlled map $\vartheta \colon (G,P) \to (\G,\P)$ between two weak right LCM-inclusions.
Suppose that $(\pi,t)$ is a Nica-covariant representation of a compactly aligned product system $X$ over $P$ with coefficients in $A$.
If $p,q \in P$ with $\vartheta(p) = \vartheta(q)$ then by (A2) either $p = q$ or $p P \cap q P = \mt$; thus Nica-covariance yields the orthogonality
\[
t_{p}(\xi_{p})^* t_{q}(\xi_{q})
=
\de_{p,q} \pi(\sca{\xi_{p}, \xi_{q}}).
\]
Hence the C*-algebra
\[
B_{\vartheta^{-1}(h)} := 
\ol{\spn} \{ \psi_{p,q}(k_{p,q}) \mid k_{p,q} \in \K(X_q, X_p), \vartheta(p) = h = \vartheta(q) \}
\]
is a matrix C*-algebra.
For a $\vee$-closed $\F \subseteq \P$ we define
\[
B_{\vartheta^{-1}(\F)} := \ol{\spn}\{ B_{\vartheta^{-1}(h)} \mid h \in \F \}.
\]
By conditions (A1) and (A2) of Definition \ref{D:control} we get that $\vartheta^{-1}(\F)$ is also $\vee$-closed (and thus the above space is a C*-algebra).
Therefore every $B_{\vartheta^{-1}(\F)}$ is the inductive limit of the matrix C*-subalgebras
\[
\spn\{ \psi_{p,q}(k_{p,q}) \mid k_{p,q} \in \K(X_q, X_p), p,q \in F, \vartheta(p) = \vartheta(q) \}
\textup{ for finite $\vee$-closed } F \subseteq \vartheta^{-1}(\F).
\]
Taking the closure of the union we obtain the $\vartheta$-fixed point algebra
\[
B_{\vartheta^{-1}(\P)} := \ol{\spn}\{ \psi_{p,q}(k_{p,q}) \mid k_{p,q} \in \K(X_q, X_p), \vartheta(p) = \vartheta(q) \}.
\]
It follows that
\[
B_{P} = \ol{\spn}\{ \psi_p(\K X_p) \mid p \in P \}
\subseteq 
B_{\vartheta^{-1}(\P)}.
\]
It is clear that the faithful conditional expectation $\ol{E}_\P$ on $\ca(\ol{\pi}, \ol{t}) = \T_\la(X)$ of Proposition \ref{P:Fock ce} is onto $\ol{B}_{\vartheta^{-1}(\P)}$.
We already commented on the effect of semigroup preserving homomorphisms $\vartheta \colon (G,P) \to (\G,\P)$ on $\T(X)$ and $\T_\la(X)$.
We give some basic facts about the effect of controlled maps on $\T_\la(X)$.

%%%%%%%%%%%%%%%%%%%%%%%%%%%%%%%%
\begin{proposition}\label{P:coa B 2}
Let $\vartheta \colon (G,P) \to (\G, \P)$ be a controlled map between weak right LCM-inclusions and let $X$ be a compactly aligned product system over $P$ with coefficients in $A$. 
Let $\Phi \colon \T_\la(X) \to B$ be a $*$-representation such that $\Phi|_{\ol{\pi}(A)}$ is faithful.
Then there exists a faithful $*$-homomorphism
\[
\T_\la(X) \longrightarrow B \otimes \ca_\la(\P) ; \ol{t}_p(\xi_p) \mapsto \Phi \ol{t}_p(\xi_p) \otimes V_{\vartheta(p)}.
\]
\end{proposition}

\begin{proof}
The proof follows the same lines with Proposition \ref{P:P coa B} with the observation that $\ol{B}_{\vartheta^{-1}(h)}$ for $h \in \G$ is a matrix algebra.
\end{proof}

As an immediate consequence we have the following corollary which extends Theorem \ref{T:extension} to the controlled setting.

%%%%%%%%%%%%%%%%%%%%%%%%%%%%%%%%
\begin{corollary} \label{C:ext thm}
Let $\vartheta \colon (G,P) \to (\G, \P)$ be a controlled map between weak right LCM-inclusions and let $X$ be a compactly aligned product system over $P$ with coefficients in $A$.
Let ${\Phi} \colon \T_\la(X) \to B$ be a $*$-representation and set
\[
\A := \ol{\alg}\{ \Phi \ol{\pi}(A), \Phi\ol{t}_p(X_p) \mid p \in P \}.
\]
Then the following are equivalent:
\begin{enumerate}
\item ${\Phi}|_{\T_\la(X)^+}$ is completely isometric.
\item There exists a completely contractive map 
\[
\A \longrightarrow B \otimes \ca(\G) ; \Phi \ol{t}_p(\xi_p) \mapsto \Phi \ol{t}_p(\xi_p) \otimes u_{\vartheta(p)}.
\]
\item There exists a completely contractive map 
\[
\A \longrightarrow B \otimes \ca_\la(\G) ; \Phi \ol{t}_p(\xi_p) \mapsto \Phi \ol{t}_p(\xi_p) \otimes \la_{\vartheta(p)}.
\]
\item There exists a completely contractive map 
\[
\A \longrightarrow B \otimes \ca_\la(\P) ; \Phi \ol{t}_p(\xi_p) \mapsto \Phi \ol{t}_p(\xi_p) \otimes V_{\vartheta(p)}.
\]
\end{enumerate}
\end{corollary}

\begin{proof}
The proof follows as in Theorem \ref{T:extension}, modulo Proposition \ref{P:f coaction} and Proposition \ref{P:coa B 2}.
\end{proof}

%%%%%%%%%%%%%%%%%%%%%%%%%%%%%%%%
\subsection{Controlled elimination} \label{Ss:ce}
%%%%%%%%%%%%%%%%%%%%%%%%%%%%%%%%

We will require the following lemma for solving polynomial equations in the $\vartheta$-fixed point algebra.

%%%%%%%%%%%%%%%%%%%%%%%%%%%%%%%%
\begin{lemma}\label{L:off diagonal}
Let $\vartheta \colon (G,P) \to (\G, \P)$ be a controlled map between weak right LCM-inclusions and let $X$ be a compactly aligned product system over $P$ with coefficients in $A$. 
Let $(\pi,t)$ be an injective Nica-covariant representation of $X$.

\noindent
(i) Let $p, q$ be distinct in $\vartheta^{-1}(h)$.
For $r,s \in \vartheta^{-1}(h)$ with $(r, s) \neq (p, q)$ we get
\[
t_p(X_p)^* \psi_{r,s}(k_{r,s}) t_q(X_q) = (0) \foral k_{r,s} \in \K(X_s, X_r).
\]
(ii) Let $\F \subseteq \P$ be $\vee$-closed and $F \subseteq \vartheta^{-1}(\F)$ be finite and $\vee$-closed.
Let $(r, s) \in F \times F$ with $\vartheta(r) = \vartheta(s)$ and $k_{r,s} \in \K(X_s, X_r)$ such that
\[
\sum_{r,s \in F, \vartheta(r) = \vartheta(s)} \psi_{r,s}(k_{r,s}) = 0,
\]
and suppose that $h \neq e_\G$ is minimal in $\F$ so that $k_{p,q} \neq 0$ for distinct $p, q \in \vartheta^{-1}(h)$. 
Then there exists a $\vee$-closed $\F' \subseteq \P$ and a finite $\vee$-closed $F' \subseteq \vartheta^{-1}(\F')$ with $e_\G \notin \F'$ and $|\F'| \leq |\F| - 1$ such that
\[
t_p(X_p)^* \psi_{p,q}(k_{p,q}) t_q(X_q) \subseteq B_{F'}.
\]
\end{lemma}

\begin{proof}
(i) First we note that condition (A2) of Definition \ref{D:control} yields $p P \cap q P = \mt$.
By Nica-covariance we have that $t_p(X_p)^* \psi_{r,s}(k_{r,s}) t_q(X_q) = (0)$ unless:
\begin{equation}\label{eq:unions}
\exists w, z, v \in P \, \textup{ such that } \, p P \cap r P = w P, q P \cap s P = z P \text{ and } r^{-1} w P \cap s^{-1} z P = v P.
\end{equation}
If $(r,s) \neq (p,q)$ and $\vartheta(r) = h = \vartheta(s)$, then condition (A2) of Definition \ref{D:control} implies that $p P \cap r P = \mt$ or $q P \cap s P = \mt$ in which case
\[
t_p(X_p)^* \psi_{r,s}(k_{r,s}) t_q(X_q) = (0).
\]

\noindent
(ii) Minimality of $h$ in $\F$ forces minimality of $p, q$ in $F$.
If (\ref{eq:unions}) holds, then Nica-covariance yields
\[
t_p(\xi_p)^* \psi_{r,s}(k_{r,s}) t_q(\xi_q) \in \psi_{p^{-1} r v, q^{-1} s v}(\K(X_{q^{-1} s v}, X_{p^{-1} r v})),
\]
otherwise the product is zero.
If $r=s$ and $v$ exists then there are $p', q', x, x', y, y' \in P$ such that
\[
pp' = rx, qq' = ry \text{ and } xx' = yy'.
\]
But then 
\[
pp'x' = rxx' = ryy' = qq'y,
\]
giving the contradiction that $p P \cap q P \neq \mt$.
Hence in this case the product is zero.
We will show that the product is zero also when $\vartheta(p^{-1} r v) = e_\G = \vartheta(q^{-1} s v)$ for $r \neq s$ unless $(r,s) = (p,q)$.
If $\vartheta(p^{-1} r v) = e_\G$ then condition (A2) of Definition \ref{D:control} yields $p \in rP$.
Likewise $q \in sP$.
Minimality of $p, q$ in $F$ forces that either $(p,q) = (r,s)$ or that $k_{r,s} = 0$.
Set 
\[
\F' := \{ h^{-1} g \mid g \in \F, g > h\}
\qand
F' := \{u^{-1} v \mid u,v \in F, \vartheta(u) = h, u > v\} \subseteq \vartheta^{-1}(\F').
\]
We see that $\F'$ is $\vee$-closed with $|\F'| \leq |\F| - 1$ and so  $F'$ is $\vee$-closed with 
\[
|F'| \leq |F \setminus \{p,q\}| = |F| - 2.
\]
Moreover we see that $p^{-1} rv, q^{-1} s v \in F'$ whenever $v$ exists.
Hence for every $\xi_p \in X_p$ and $\xi_q \in X_q$ there are suitable $k'_{r',s'}$ with non-trivial $r', s' \in F'$ so that
\begin{align*}
0 
& = 
\sum_{r,s} t_p(\xi_p)^* \psi_{r,s}(k_{r,s})  t_q(\xi_q)
 = 
t_p(\xi_p)^* \psi_{p,q}(k_{p,q}) t_q(\xi_q) + \sum_{r',s'} \psi_{r',s'}(k_{r',s'}'),
\end{align*}
and the proof is complete.
\end{proof}

In the next proposition we show that we can eliminate elements of the form $\psi_{r,s}(k_{r,s})$ for $r \neq s$ with $\vartheta(r) = \vartheta(s)$, from a polynomial equation in the $\vartheta$-fixed point algebra.
Such arguments for the left-regular representation appear in \cite[Proposition 2.10]{Din91} and \cite[Lemma 4.1]{LR96} for semigroups over quasi-lattices, i.e., when $X_p = \bC$ for every $p \in P$ and $(G,P)$ is a quasi-lattice.
Here we need to move in three directions: (a) beyond one-dimensional fibers, (b) beyond quasi-lattices, and (c) beyond just the left regular representation.
A step towards this direction is done in \cite{Kak19} for quasi-lattices that are controlled by $(\bZ^n, \bZ^n_+)$, and here we expand further on this approach.

%%%%%%%%%%%%%%%%%%%%%%%%%%%%%%%%
\begin{proposition}\label{P:injective}
Let $\vartheta \colon (G,P) \to (\G, \P)$ be a controlled map between weak right LCM-inclusions and let $X$ be a compactly aligned product system over $P$ with coefficients in $A$. 
Let $(\pi,t)$ and $(\pi',t')$ be injective Nica-covariant representations such that there exists a canonical $*$-epimorphism
\[
\Phi \colon \ca(\pi', t') \longrightarrow \ca(\pi,t) 
\text{ with }
\Phi(\pi'(a)) = \pi(a), \Phi(t_p'(\xi_p)) = t_p(\xi_p).
\]
Then $\Phi$ is injective on $B_{P}'$ if and only if it is injective on $B_{\vartheta^{-1}(\P)}'$.
\end{proposition}

\begin{proof}
As $B'_{P} \subseteq B'_{\vartheta^{-1}(\P)}$ we need to show just one direction.
To this end suppose that $\Phi$ is injective on the C*-subalgebras of the form
\[
B_F' = \ol{\spn}\{ \psi_p'(\K X_p) \mid p \in F\},
\]
for every finite $\vee$-closed $F \subseteq P$.
We will show that $\Phi$ is injective on every
\[
B_{\vartheta^{-1}(\F)}' = \ol{\spn}\{ \psi_{r,s}'(k_{r,s}) \mid \vartheta(r) = \vartheta(s) \in \F\},
\]
for all $\vee$-closed $\F \subseteq \P$.
Our strategy is to show the implication
\[
\sum_{r,s \in F, \vartheta(r) = \vartheta(s) \in \F} \psi_{r,s}'(k_{r,s}) \in \ker \Phi
\, \Longrightarrow \,
k_{r,s} = 0 \text{ whenever } r \neq s,
\]
for every finite $\vee$-closed $F \subseteq \vartheta^{-1}(\F)$.
Then injectivity of $\Phi$ in the smaller cores yields
\[
\sum_{r \in F} \psi_{r}(k_{r})
=
\sum_{r,s \in F, \vartheta(r) = \vartheta(s) \in \F} \psi_{r,s}(k_{r,s})
=
0
\, \Longrightarrow \,
\sum_{r \in F} \psi_{r}'(k_{r})
=
0,
\]
and so
\[
\sum_{r,s \in F, \vartheta(r) = \vartheta(s) \in \F} \psi_{r,s}'(k_{r,s})
=
\sum_{r \in F} \psi_{r}'(k_{r})
=
0.
\]
Since $F$ is arbitrary this proves injectivity of $\Phi$ on $B_{\vartheta^{-1}(\F)}'$.
We proceed by induction on the size of $\F$.

\smallskip

\noindent
{\bf Case 1.} Assume that $\F = \{h\}$ and let $F$ be a finite $\vee$-closed subset of $\vartheta^{-1}(\F)$.
Suppose that
\[
\sum_{r,s \in F, \vartheta(r) = \vartheta(s) = h} \psi_{r,s}(k_{r,s}) = 0,
\]
and fix $p,q \in \vartheta^{-1}(h)$.
Then condition (A2) of Definition \ref{D:control} implies that
\[
\psi_p(\K X_p) \psi_{p,q}(k_{p,q}) \psi_q(\K X_q)
=
\psi_p(\K X_p) \left( \sum_{r,s \in F, \vartheta(r) = \vartheta(s) = h} \psi_{r,s}(k_{r,s}) \right) \psi_q(\K X_q)
=
(0).
\]
Using an approximate identity on both sides gives that $\psi_{p,q}(k_{p,q}) = 0$ and injectivity of $\psi$ implies that $k_{p,q} = 0$.
As $(p,q)$ was arbitrary we have that $k_{r,s} = 0$ for all $r,s \in \vartheta^{-1}(h)$ and so
\[
\sum_{r,s \in F, \vartheta(r) = \vartheta(s) = h} \psi_{r,s}'(k_{r,s}) = 0.
\]
Hence $\Phi$ is injective on $B'_{\vartheta^{-1}(\F)}$ whenever $|\F| = 1$.

\smallskip

\noindent
{\bf Case 2.} Assume that $\F = \{ e_\G, h\}$ and let $F$ be a finite $\vee$-closed subset of $\vartheta^{-1}(\F)$.
Suppose that
\[
\sum_{r,s \in F, \vartheta(r) = \vartheta(s)} \psi_{r,s}(k_{r,s}) = 0.
\]
By condition (A1) of Definition \ref{D:control} we have that if $p \neq q$ with $\vartheta(p) = \vartheta(q) \in \F$ then $p, q \in \vartheta^{-1}(h)$.
As before and by using item (i) of Lemma \ref{L:off diagonal} on $p,q$ we get that
\[
\psi_p(\K X_p) \psi_{p,q}(k_{p,q}) \psi_q(\K X_q)
=
\psi_p(\K X_p) \left(\sum_{r,s \in F, \vartheta(r) = \vartheta(s)} \psi_{r,s}(k_{r,s}) \right) \psi_q(\K X_q)
=
(0).
\]
Using an approximate identity eventually gives that $k_{p,q} = 0$ whenever $p \neq q$.
Hence $k_{r,s} = 0$ whenever $r \neq s$ in $F$ and injectivity of $\Phi$ on $B'_F$ gives that 
\[
\sum_{r,s \in F, \vartheta(r) = \vartheta(s)} \psi_{r,s}'(k_{r,s}) = 0.
\]
Hence $\Phi$ is injective on $B'_{\vartheta^{-1}(\F)}$ whenever $\F = \{e,h\}$.

\smallskip

\noindent
{\bf Case 3.} Assume that $\F = \{h_1, h_2\}$ and let $F$ be a finite $\vee$-closed subset of $\vartheta^{-1}(\F)$.
Suppose that
\[
\sum_{r,s \in F, \vartheta(r) = \vartheta(s) \in \{h_1, h_2\}} \psi_{r,s}'(k_{r,s}) \in \ker \Phi.
\]
Without loss of generality assume that it is written with the understanding that for every $\psi'_{r,s}(k_{r,s})$ we have that either $\psi'_{r,s}(k_{r,s}) = 0$ or that 
\[
\psi'_{r,s}(k_{r,s}) \notin B'_{\vartheta^{-1}(\vartheta(r) \P)}.
\]
Choose $h \in \F$ be minimal such that $\psi'_{p,q}(k_{p,q}) \neq 0$ for distinct $p,q \in \vartheta^{-1}(h)$.
Hence $k_{p,q} \neq 0$ and so 
\[
0 \neq \psi'_{p,q}(k_{p,q}) \notin B'_{\vartheta^{-1}(h \P)}.
\]
By using Lemma \ref{L:off diagonal} item (ii) we have that
\[
t_p(X_p)^* \psi_{p,q}(k_{p,q}) t_q(X_q) \subseteq B_{\vartheta^{-1}(\F')}
\qfor
|\F'| \leq 1,
\]
with $e_\G \notin \F'$.
By using injectivity of Case 2 we then derive that
\[
\psi'_p(\K X_p) \psi_{p,q}'(k_{p,q}) \psi'_q(\K X_q) \subseteq B'_{\vartheta^{-1}(h \P)}.
\]
By using approximate identities on both sides we get the contradiction
\[
\psi'_{p,q}(k_{p,q}) \in B'_{\vartheta^{-1}(h \P)}.
\]
Hence $\Phi$ is injective on $B'_{\vartheta^{-1}(\F)}$ whenever $|\F| \leq 2$.

\smallskip

\noindent 
{\bf Case 4.} Let $\F \subseteq \P$ be $\vee$-closed with $|\F| = n+1$ and assume that $\Phi$ is injective on $B'_{\vartheta^{-1}(\F')}$ for all $\F' \subseteq \P$ with $|\F'| \leq n$.
We will show that it is injective on $B'_{\vartheta^{-1}(\F)}$.
To this end let $F$ be a finite $\vee$-closed subset of $\vartheta^{-1}(\F)$ and suppose that
\[
\sum_{r,s \in F, \vartheta(r) = \vartheta(s) \in \F} \psi_{r,s}'(k_{r,s}) \in \ker \Phi,
\]
with the understanding that for every $\psi'_{r,s}(k_{r,s})$ we have that either $\psi'_{r,s}(k_{r,s}) = 0$ or that
\[
\psi'_{r,s}(k_{r,s}) \notin B'_{\vartheta^{-1}( \vartheta(r) \P)}.
\]
Choose $h \in \F$ be minimal such that $\psi'_{p,q}(k_{p,q}) \neq 0$ for distinct $p,q \in \vartheta^{-1}(h)$.
Hence $k_{p,q} \neq 0$ and so 
\[
0 \neq \psi'_{p,q}(k_{p,q}) \notin B'_{\vartheta^{-1}(h \P)}.
\]
By using Lemma \ref{L:off diagonal} item (ii) we then have that
\[
t_p(X_p)^* \psi_{p,q}(k_{p,q}) t_q(X_q) \subseteq B_{\vartheta^{-1}(\F')}
\qfor
|\F'| \leq |\F| - 1 = n.
\]
By using the induction hypothesis we then derive that
\[
\psi_p'(\K X_p) \psi_{p,q}'(k_{p,q}) \psi_q'(\K X_q) \subseteq B_{\vartheta^{-1}(\F')} \subseteq B'_{\vartheta^{-1}(h \P)}.
\]
By using approximate identities on both sides we have the contradiction
\[
\psi'_{p,q}(k_{p,q}) \in B'_{\vartheta^{-1}(h \P)}.
\]
This concludes the proof of the proposition.
\end{proof}

Combining with \cite[Theorem 3.10]{Seh18} we get the following corollary.

%%%%%%%%%%%%%%%%%%%%%%%%%%%%%%%%
\begin{corollary}\label{C:fpa con}
Let $\vartheta \colon (G,P) \to (\G, \P)$ be a controlled map between weak right LCM-inclusions and let $X$ be a compactly aligned product system over $P$ with coefficients in $A$. 
Then the following are equivalent for a strongly covariant representation $(\pi,t)$ of $A \times_{X} P$:
\begin{enumerate}
\item The $*$-representation $\pi$ is faithful on $A$.
\item The induced $*$-representation is faithful on the fixed point algebra $B_{P}$ of $A \times_{X} P$.
\item The induced $*$-representation is faithful on the $\vartheta$-fixed point algebra $B_{\vartheta^{-1}(\P)}$ of $A \times_{X} P$. 
\end{enumerate}
In particular this holds for the $*$-representations of $q_{\scv}(\T_\la(X))$ and $A \times_{X, \la} P$.
\end{corollary}

A second application of the controlled elimination allows to pass in-between the C*-envelopes induced by $G$ and $\G$.

%%%%%%%%%%%%%%%%%%%%%%%%%%%%%%%%
\begin{proposition}\label{P:same G}
Let $\vartheta \colon (G,P) \to (\G, \P)$ be a controlled map between weak right LCM-inclusions, and let $X$ be a compactly aligned product system over $P$ with coefficients in $A$. 
Let $\ol{\de}_\G$ be the induced coaction of $\G$ on $\T_\la(X)$ and $\T_\la(X)^+$.
Then $\cenv(\T_\la(X)^+, G, \ol{\de}_G)$ inherits a normal coaction of $\G$ and there exists a $\G$-equivariant $*$-isomorphism
\[
\cenv(\T_\la(X)^+, G, \ol{\de}_G) \simeq \cenv(\T_\la(X)^+, \G, \ol{\de}_\G)
\]
that fixes $\T_\la(X)^+$.
\end{proposition}

\begin{proof}
By Theorem \ref{T:fpa 1-1}, Proposition \ref{P:coa ind} and Proposition \ref{P:f coaction}, we get that $\cenv(\T_\la(X)^+, G, \ol{\de}_G)$ admits a normal coaction of $\G$ and therefore there exists a $\G$-equivariant $*$-epimorphism
\[
\Phi \colon \cenv(\T_\la(X)^+, G, \ol{\de}_G) \longrightarrow \cenv(\T_\la(X)^+, \G, \ol{\de}_\G)
\]
that fixes $\T_\la(X)^+$.
By construction $\Phi$ is $\G$-equiva\-riant, and so it intertwines the faithful conditional expectations induced by $\G$.
On the other hand, by Theorem \ref{T:fpa 1-1} the map $\Phi$ is faithful on the $G$-fixed point algebra of $\cenv(\T_\la(X)^+, G, \ol{\de}_G)$.
By Proposition \ref{P:injective} the map $\Phi$ is faithful on the $\G$-fixed point algebra of $\cenv(\T_\la(X)^+, G, \ol{\de}_G)$.
Consequently $\Phi$ is injective.
\end{proof}

%%%%%%%%%%%%%%%%%%%%%%%%%%%%%%%%
\section{Applications}\label{S:app}
%%%%%%%%%%%%%%%%%%%%%%%%%%%%%%%%

%%%%%%%%%%%%%%%%%%%%%%%%%%%%%%%%
\subsection{Co-universality of Sehnem's covariance algebra}
%%%%%%%%%%%%%%%%%%%%%%%%%%%%%%%%

We will consider weak right LCM-inclusions that are controlled by exact groups.
In this case we get normality of the coaction of $G$ on $q_{\scv}(\T_\la(X))$, and thus the latter coincides with $A \times_{X, \la} P$, and by \cite[Theorem 5.3]{DKKLL20} with $\cenv(\T_\la(X)^+, G, \ol{\de}_G)$.
This provides another algebraic description of $\cenv(\T_\la(X)^+, G, \ol{\de}_G)$ by the strong covariance relations in the Fock space representation.

%%%%%%%%%%%%%%%%%%%%%%%%%%%%%%%%
\begin{theorem}\label{T:cenv red}
Let $\vartheta \colon (G,P) \to (\G, \P)$ be a controlled map between weak right LCM-inclusions and let $X$ be a compactly aligned product system over $P$ with coefficients in $A$. 
Consider the canonical $*$-epimorphisms
\[
q_{\scv}(\T_\la(X)) \longrightarrow A \times_{X, \la} P \simeq \cenv(\T_\la(X)^+, G, \ol{\de}_G) \longrightarrow \cenv(\T_\la(X)^+).
\]
If $\G$ is exact then the left map is faithful.
If in addition $\G$ is abelian then the right map is also faithful.
\end{theorem}

\begin{proof}
First we show that the ideal of the strong covariance relations is $\G$-induced.
Let $\I_{\la}$ be the image of the strong covariance relations in $\T_\la(X)$ so that $\T_\la(X)/\I_\la = q_{\scv}(\T_\la(X))$.
Let us denote by $\ol{B}_F$ the cores of the Fock representation $(\ol{\pi}, \ol{t})$ and let $q_{\I_\la} \colon \T_\la(X) \to q_{\scv}(\T_\la(X))$ be the canonical $*$-epimorphism.
Proposition \ref{P:injective} implies that
\begin{align*}
\I_\la \cap \ol{B}_{\vartheta^{-1}(\P)}
& =
\ol{\bigcup_{\text{finite, $\vee$-closed }\F \subseteq \P} \ker q_{\I_\la} \cap \ol{B}_{\vartheta^{-1}(\F)}} \\
& =
\ol{\bigcup_{\text{finite, $\vee$-closed }F \subseteq P} \ker q_{\I_\la} \cap \ol{B}_{F}}
=
\I_\la \cap \ol{B}_{P}.
\end{align*}
Therefore we get that
\[
\I_{\la} 
=
\sca{ \I_{\la} \cap \ol{B}_{P} }
= 
\sca{ \I_{\la} \cap \ol{B}_{\vartheta{^-1}(\P)} },
\]
showing that $\I_\la$ is indeed $\G$-induced.

Consequently, by exactness of $\G$ we derive that the normal coaction of $\G$ on $\T_\la(X)$ descends to a normal coaction on the quotient $q_{\scv}(\T_\la(X))$.
Thus by Proposition \ref{P:P cover} we have that $q_{\scv}(\T_\la(X))$ is a C*-cover for $(\T_\la(X)^+, \G, \ol{\de}_\G)$.
Therefore there exists a $\G$-equivariant $*$-epimorphism
\[
\Phi \colon q_{\scv}(\T_\la(X)) \longrightarrow \cenv(\T_\la(X)^+, \G, \ol{\de}_\G)
\]
that fixes $\T_\la(X)^+$.
The $*$-epimorphism $\Phi$ intertwines the coactions (and thus the faithful conditional expectations implemented by normality and exactness of $\G$), and it is faithful on the $G$-fixed point algebra $[q_{\scv}(\T_\la(X))]_e$ by Corollary \ref{C:A in red}.
Hence we derive that $\Phi$ is faithful by Corollary \ref{C:fpa con}.
By Proposition \ref{P:same G} we conclude that
\[
q_{\scv}(\T_\la(X)) \simeq \cenv(\T_\la(X)^+, \G, \ol{\de}_\G) \simeq \cenv(\T_\la(X)^+, G, \ol{\de}_G).
\]

Now if in addition $\G$ is abelian then $\cenv(\T_\la(X)^+)$ inherits the coaction of $\G$ by the dual gauge action $\wh{\G}$.
Due to co-universality we thus derive
\[
\cenv(\T_\la(X)^+, G, \ol{\de}_G) \simeq \cenv(\T_\la(X)^+, \G, \ol{\de}_\G) \simeq \cenv(\T_\la(X)^+)
\]
and the proof is complete.
\end{proof}

Combining with Proposition \ref{P:coa ind} and Proposition \ref{P:P cover} we get the following corollary.

%%%%%%%%%%%%%%%%%%%%%%%%%%%%%%%%
\begin{corollary}\label{C:coun}
Let $\vartheta \colon (G,P) \to (\G, \P)$ be a controlled map between weak right LCM-inclusions such that $\G$ is exact, and let $X$ be a compactly aligned product system over $P$ with coefficients in $A$. 
Then $q_{\scv}(\T_\la(X))$ is co-universal with respect to $G$-equivariant and to $\G$-equivariant quotients of $\T_\la(X)$ that are faithful on $A$.
\end{corollary}

%%%%%%%%%%%%%%%%%%%%%%%%%%%%%%%%
\begin{remark}\label{R:Hao-Ng}
As an immediate consequence of Theorem \ref{T:cenv red} we get that the coaction of $G$ on $q_{\scv}(\T_\la(X))$ is normal.
Therefore one can use the results of \cite{DKKLL20} to derive that the reduced Hao-Ng Problem for discrete group actions on $A \times_{X, \la} P$ has a positive answer when $(G,P)$ is controlled by $(\G,\P)$ with $\G$ exact.
A similar method applies whenever the C*-envelope functor is stable under crossed products, e.g., for dynamics over abelian locally compact groups or when the tensor algebra is hyperrigid \cite{Kat20, KR16}, and we leave this to the interested reader.
\end{remark}

Next we consider \emph{amenably controlled} weak right LCM-inclusions, i.e., the range of the controlled map is inside an amenable group.
In this case the reduced C*-algebras become universal with respect to classes of representations.
%By universality $\T(X)$ admits a coaction of $\G$ which by \cite[Proposition A.1]{CLSV11} descends to $\N\T(X)$ and $A \times_{X} P$.
First we consider $\N\T(X)$.
(A variant of) the following has been obtained by Fowler \cite{Fow02} for non-degenerate product systems over quasi-lattices.
Here we extend it to the weak right LCM-inclusions framework with a different approach that does not require non-degeneracy of $X$.

%%%%%%%%%%%%%%%%%%%%%%%%%%%%%%%%
\begin{theorem}\label{T:ntx un}
Let $\vartheta \colon (G,P) \to (\G,\P)$ be a controlled map between weak right LCM-inclusions with $\G$ amenable and let $X$ be a compactly aligned product system over $P$ with coefficients in $A$.
Then the Fock representation is faithful on $\N\T(X)$.

Conversely, suppose that $(\pi,t)$ is an injective $\G$-equivariant Nica-covariant representation of $X$ and for every $\vee$-closed $\F \subseteq \P$ we have linear independence in the $\vartheta$-cores in the sense that
\[
B_{\vartheta^{-1}(\F)} = \sumoplus_{h \in \F} B_{\vartheta^{-1}(h)}.
\]
Then $(\pi,t)$ integrates to a faithful representation of $\N\T(X)$.

In particular a Nica-covariant pair $(\pi, t)$ defines a faithful representation of $\N\T(X)$ if and only if the associated representation is $\G$-equivariant and satisfies the condition:
\[
\sum_{p \in F} \psi_{p}(k_{p}) = 0 \Longrightarrow k_p = 0 \foral p \in F,
\]
for every $\vee$-closed $\F \subseteq \P$ and every finite $\vee$-closed $F \subseteq \vartheta^{-1}(\F)$.
\end{theorem}

\begin{proof}
Let $(\wh{\pi}, \wh{t})$ be a faithful representation of $\N\T(X)$ and let the canonical $*$-epimorphism
\[
\Phi \colon \N\T(X) = \ca(\wh{\pi}, \wh{t}) \longrightarrow \T_\la(X) = \ca(\ol{\pi},\ol{t}).
\] 
Let $\ol{E}_\G$ be the faithful conditional expectation induced by Proposition \ref{P:Fock ce} on $\T_\la(X)$.
Let $\wh{E}_\G$ be the faithful conditional expectation on $\N\T(X)$ induced by the amenable $\G$.
Since $\Phi \wh{E}_\G = \ol{E}_\G \Phi$ it suffices to show injectivity of $\Phi$ on $\wh{B}_{\F}$ for every $\vee$-closed $\F \subseteq \P$.
To this end fix a finite $\vee$-closed $F \subseteq \vartheta^{-1}(\F)$ and suppose that
\[
f := \sum \{ \wh{\psi}_{r_1, r_2}(k_{r_1, r_2}) \mid k_{r_1, r_2} \in \K(X_{r_2}, X_{r_1}), r_1, r_2 \in F, \vartheta(r_1) = \vartheta(r_2) \} \in \ker\Phi.
\]
Let $h$ be minimal in $\F$ such that $k_{q_1, q_2} \neq 0$ with $\vartheta(q_1) = \vartheta(q_2) = h$.
By using condition (A2) of Definition \ref{D:control} and the Fock space representation we have that
\[
k_{q_1, q_2} = Q_{q_1} \Phi(f) Q_{q_2} = 0
\]
for the projections $Q_p \colon \F X \to X_p$, which gives the required contradiction.
Thus the Fock representation is injective and also we have linear independence of the cores.
The converse follows with a similar proof.

For the last part it is clear that the condition with $F = \{e_G\}$ and $\F = \{e_\G\}$ implies that $\pi$ is injective.
Moreover the condition shows that the canonical $*$-epimorphism $\Phi$ is injective on the C*-subalgebras 
\[
\wh{B}_{F} = \spn\{ \wh{\psi}_{r}(k_r) \mid r \in F\}
\]
for every finite $\vee$-closed $F \subseteq P$, and so $\Phi$ is injective on $\wh{B}_{P}$. 
Thus by Proposition \ref{P:injective} the map $\Phi$ is injective on $\wh{B}_{\vartheta^{-1}(\P)}$ and hence on $\N\T(X)$.
\end{proof}

Next we consider the universal covariance algebra $A \times_{X} P$.

%%%%%%%%%%%%%%%%%%%%%%%%%%%%%%%%
\begin{theorem}\label{T:co-un seh}
Let $\vartheta \colon (G,P) \to (\G,\P)$ be a controlled map between weak right LCM-inclusions with $\G$ amenable and let $X$ be a compactly aligned product system over $P$ with coefficients in $A$.
Then a strongly covariant representation of $X$ integrates to a faithful representation of $A \times_{X} P$, if and only if it is injective and $G$-equivariant, if and only if it is injective and $\G$-equivariant.
\end{theorem}

\begin{proof}
By Theorem \ref{T:ntx un} we have that $A \times_{X} P$ coincides with $q_{\scv}(\T_\la(X))$ and $A \times_{X, \la} P$.
Thus the result follows from Corollary \ref{C:coun}.
\end{proof}

%%%%%%%%%%%%%%%%%%%%%%%%%%%%%%%%
\begin{remark}
When $(G,P)$ is amenably controlled then we have a wider selection for a coaction that implements the Extension Theorem.
%Moreover due to the universal property of $\T_\la(X)$ we can replace $B$ with $\ca(\pi,t)$ for any Nica-covariant representation of $X$.
Figure \ref{F:diag cp} depicts those.
We denote restrictions of $*$-homomorphisms by solid arrows, and we have used Proposition \ref{P:P coa B} for the upper and lower completely isometric maps.
Recall that if $\G$ is amenable then $\ca(\G) \simeq \ca_\la(\G)$ is nuclear, and by \cite{Li13} $\ca_\la(\P)$ is also nuclear.
\end{remark}

%%%%%%%%%%%%%%%%%%%%%%%%%%%%%%%%
\begin{corollary}\label{C:extension con}
Let $\vartheta \colon (G,P) \to (\G, \P)$ be a controlled map between weak right LCM-inclusions with $\G$ amenable.
Suppose that $A, X_p \subseteq B(\H)$ for $p \in P$ define a compactly aligned product system $X = \{X_p\}_{p \in P}$ and set
\[
\A := \ol{\alg}\{A, X_p \mid p \in P \}.
\]
Then the following are equivalent:
\begin{enumerate}
\item There is a completely isometric isomorphism 
\[
\A \longrightarrow \T_\la(X)^+; \xi_p \mapsto \ol{t}(\xi_p).
\]
\item There is a completely contractive map 
\[
\A \longrightarrow \T_\la(X)^+ \otimes \ca(G) ; \xi_p \mapsto \ol{t}(\xi_p) \otimes u_p.
\]
\item There is a completely contractive map 
\[
\A \longrightarrow \T_\la(X)^+ \otimes \ca_\la(G) ; \xi_p \mapsto \ol{t}(\xi_p) \otimes \la_p.
\]
\item There is a completely contractive map 
\[
\A \longrightarrow \T_\la(X)^+ \otimes \ca_\la(P) ; \xi_p \mapsto \ol{t}(\xi_p) \otimes V_p.
\]
\item There is a completely contractive map 
\[
\A \longrightarrow \T_\la(X)^+ \otimes \ca(\G) ; \xi_p \mapsto \ol{t}(\xi_p) \otimes u_{\vartheta(p)}.
\]
\item There is a completely contractive map 
\[
\A \longrightarrow \T_\la(X)^+ \otimes \ca_\la(\P) ; \xi_p \mapsto \ol{t}(\xi_p) \otimes V_{\vartheta(p)}.
\]
\end{enumerate}
\end{corollary}

\begin{proof}
The proof follows by the system of maps in Figure \ref{F:diag cp}, where the solid arrows denote the maps that arise from restrictions of $*$-homomorphisms from the appropriate C*-algebras to the required subalgebras.
\end{proof}

%%%%%%%%%%%%%%%%%%%%%%%%%%%%%%%%
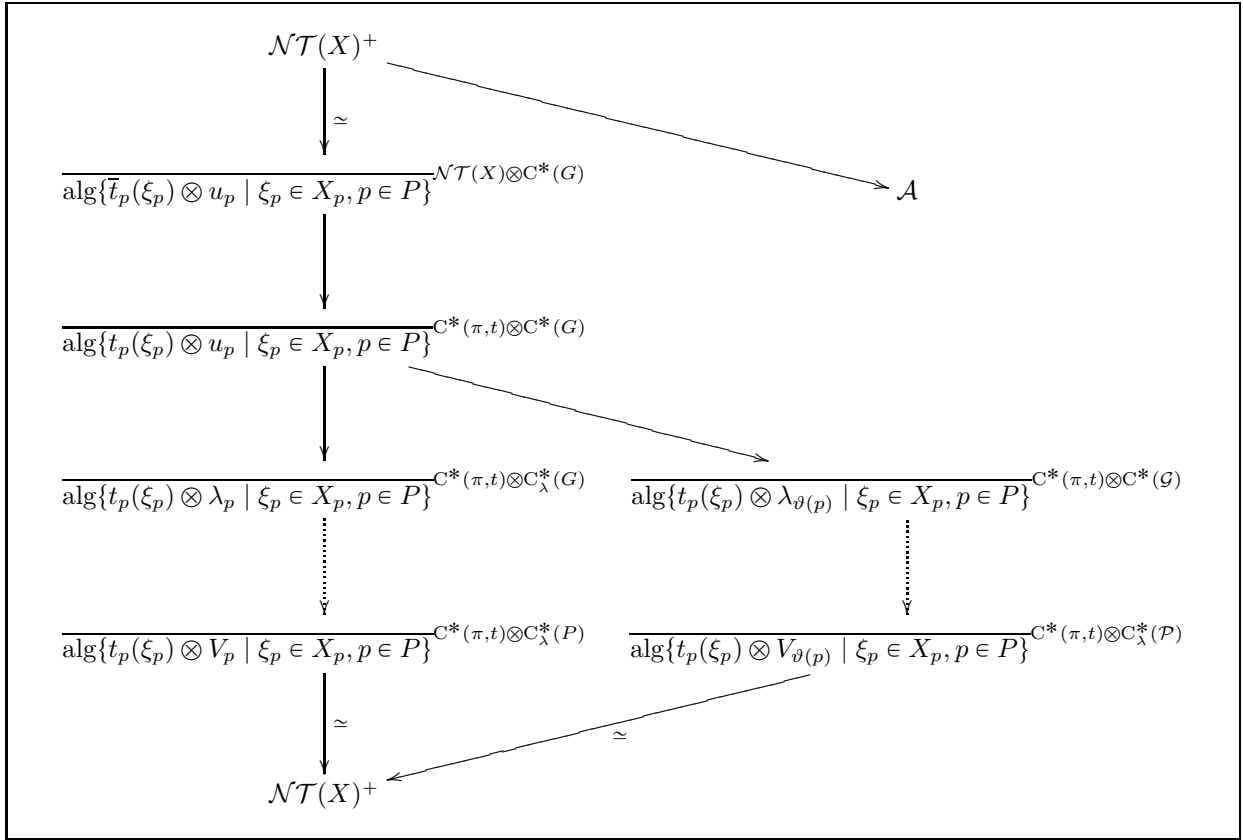
\begin{figure}[h]
{\small
\[
\xymatrix@R=35pt@C=10pt{
\N\T(X)^+ \ar[d]^{\simeq} \ar[rd] & \\
\ol{\alg\{ \ol{t}_p(\xi_p) \otimes u_p \mid \xi_p \in X_p, p \in P\}}^{\N\T(X) \otimes \ca(G)} \ar[d] &
\A \\
\ol{\alg\{ {t}_p(\xi_p) \otimes u_p \mid \xi_p \in X_p, p \in P\}}^{\ca(\pi,t) \otimes \ca(G)} \ar[d] \ar[dr] & \\
\ol{\alg\{ {t}_p(\xi_p) \otimes \la_p \mid \xi_p \in X_p, p \in P\}}^{\ca(\pi,t) \otimes \ca_\la(G)} \ar@{..>}[d] 
& 
\ol{\alg\{ {t}_p(\xi_p) \otimes \la_{\vartheta(p)} \mid \xi_p \in X_p, p \in P\}}^{\ca(\pi,t) \otimes \ca(\G)} \ar@{..>}[d] \\
\ol{\alg\{ {t}_p(\xi_p) \otimes V_p \mid \xi_p \in X_p, p \in P\}}^{\ca(\pi,t) \otimes \ca_\la(P)} \ar[d]^{\simeq} 
& 
\ol{\alg\{ {t}_p(\xi_p) \otimes V_{\vartheta(p)} \mid \xi_p \in X_p, p \in P\}}^{\ca(\pi,t) \otimes \ca_\la(\P)} \ar[dl]^{\simeq} \\
\N\T(X)^+
}
\]
}
\caption{Diagram of completely positive maps fixing the nonselfadjoint part}
\label{F:diag cp}
\end{figure}

%%%%%%%%%%%%%%%%%%%%%%%%%%%%%%%%
\subsection{Exactness and nuclearity}
%%%%%%%%%%%%%%%%%%%%%%%%%%%%%%%%

We will require some results about nuclearity which we record here for convenience.

%%%%%%%%%%%%%%%%%%%%%%%%%%%%%%%%
\begin{lemma}\label{L:Kat} \cite[Proposition B.8]{Kat04}
Let $(\pi, t)$ be a representation of a C*-correspondence $X$ over $A$ such that $\pi(A) \subseteq B$ and $t(X) \subseteq Y$ for a second C*-correspondence $X$ over $B$.
If $\pi \colon A \to B$ is nuclear then the induced map $\psi \colon \K X \to \K Y$ is nuclear.
\end{lemma}

%%%%%%%%%%%%%%%%%%%%%%%%%%%%%%%%
\begin{lemma} \label{L:nuc-quot} \cite[Proposition 3.1]{Kak19}
Let $A,A'$ be C*-algebras and let the ideals $I \lhd A$ and $I'\lhd A'$. 
Suppose we have the following commutative diagram of short exact sequences
\[
\xymatrix{
0 \ar[r] & I \ar[r] \ar[d]^{\varphi_0} & A \ar[r] \ar[d]^{\varphi} & A / I \ar[r] \ar[d]^{\widetilde{\varphi}} & 0\\
0 \ar[r] & I' \ar[r] & A' \ar[r] & A' / I' \ar[r] & 0
}
\]
where $\varphi \colon A \rightarrow A'$ is an injective $*$-homomorphism that satisfies $\varphi(I) \subseteq I'$, $\widetilde{\varphi} \colon A / I \rightarrow A' / I'$  is the induced map and $\varphi_0 :=\varphi|_{I}$. 
If $\varphi \colon A\rightarrow A'$ is nuclear, then $\varphi_0$ and $\widetilde{\varphi}$ are both nuclear.
\end{lemma}

%%%%%%%%%%%%%%%%%%%%%%%%%%%%%%%%
\begin{lemma} \label{L:nuc-ext} \cite[Proposition 3.3]{Kak19}
Let $A, A'$ be C*-algebras and let the ideals $I \lhd A$ and $I'\lhd A'$. 
Suppose we have the following commutative diagram of short exact sequences
\[
\xymatrix{
0 \ar[r] & I \ar[r] \ar[d]^{\varphi_0} & A \ar[r] \ar[d]^{\varphi} & A / I \ar[r] \ar[d]^{\widetilde{\varphi}} & 0\\
0 \ar[r] & I' \ar[r] & A' \ar[r] & A' / I' \ar[r] & 0
}
\]
where $\varphi \colon A \rightarrow A'$ is an injective $*$-homomorphism that satisfies $\varphi(I) \subseteq I'$, $\widetilde{\varphi} \colon A / I \rightarrow A' / I'$  is the induced map and $\varphi_0 :=\varphi|_{I}$. 
Suppose further that there exists a c.a.i.\ $(e_i)$ of $I'$ such that $\varphi(a)e_i \in \varphi_0(I)$ for all $a\in A$. 
If $\varphi_0$ and $\widetilde{\varphi}$ are nuclear, then so is $\varphi$.
\end{lemma}

First we provide a nuclearity/exactness result for $\T_\la(X)$.
%As $\ca_\la(P)$ is the Fock algebra for the trivial product system $X$ with $X_p = \bC$ for all $p \in P$, the following theorem extends the relevant nuclearity results of \cite{Li13} for amenable weak right LCM-inclusions.

%%%%%%%%%%%%%%%%%%%%%%%%%%%%%%%%
\begin{theorem}\label{T:nuclear}
Let $(G,P)$ be a weak right LCM-inclusion and $X$ be a compactly aligned product system over $P$ with coefficients in $A$.
Let $\ol{E}_P \colon \T_\la(X) \to \ol{B}_{P}$ be the faithful conditional expectation that arises by compressing to the diagonal.
Then the following are equivalent:
\begin{enumerate}
\item $A$ is nuclear (resp.\ exact) and $\ol{E}_P \otimes_{\max} \id_D$ is a faithful conditional expectation on $\T_\la(X) \otimes_{\max} D$ for all C*-algebras $D$.
\item $\T_\la(X)$ is nuclear (resp.\ exact).
\end{enumerate}
\end{theorem}

\begin{proof}
We will show nuclearity; exactness follows in the same way.
Notice that for any C*-algebra $D$ we have the following commutative diagram
\[
\xymatrix@R=35pt
{
\ca(\ol{\pi}, \ol{t}) \otimes_{\max} D \ar[rr] \ar[d]^{\ol{E}_P \otimes_{\max} \id} & & \ca(\ol{\pi}, \ol{t}) \otimes D \ar[d]^{\ol{E}_P \otimes \id} \\
\ol{B}_{P} \otimes_{\max} D \ar[rr] & & \ol{B}_{P} \otimes D
}
\]
and recall that $\ol{E}_P \otimes \id$ is faithful on $\ca(\ol{\pi}, \ol{t}) \otimes D$.

Suppose first that $\ca(\ol{\pi}, \ol{t})$ is nuclear.
Then trivially $\ol{E}_P \otimes_{\max} \id$ is faithful on $\ca(\ol{\pi}, \ol{t}) \otimes_{\max} D$.
Since $A$ is the corner of $\ca(\ol{\pi}, \ol{t})$ at the $(e,e)$-place we have that $A$ is nuclear, as the compression of a nuclear C*-algebra.

For the converse, the diagram above implies that it suffices to show that $\ol{B}_{P}$ is nuclear.
Equivalently it suffices to show that $\ol{B}_F$ is nuclear for every finite $\vee$-closed $F \subseteq P$.
To this end let $F= \{p_1, \dots, p_n\}$.
We choose the enumeration so that it covers the partial order in $F$ in the sense that if $p_m > p_{m'}$ then $m < m'$.
We will use induction on $n$.

For the first step we have that $\ol{\psi}_{p_1}(\K X_{p_1})$ is nuclear as $A$ is nuclear by \cite[Proposition B.7]{Kat04}.
For the inductive step suppose that $\ol{B}_{F_k}$ is nuclear for $F_k = \{p_1, \dots, p_k\}$ (which is $\vee$-closed by the choice of the enumeration).
We will show that so is $\ol{B}_{F_{k+1}}$ for $F_{k+1}=\{p_1, \dots, p_k, p_{k+1}\}$.
The enumeration shows that $p_{k+1}$ is minimal in $F_{k+1}$ and hence
\[
\ol{B}_{F_{k+1}} = \ol{B}_{F_k} \oplus \ol{\psi}_{p_{k+1}}(\K X_{p_{k+1}}).
\]
Indeed let $k_{p_i} \in \K X_{p_i}$ such that
\[
\sum_{i=1}^{k+1} \ol{\psi}_{p_i}(k_{p_i}) = 0.
\]
Due to minimality of $p_{k+1}$ in $F_{k+1}$ we have that
\[
k_{p_{k+1}} = Q_{p_{k+1}} \left( \sum_{i=1}^{k+1} \ol{\psi}_{p_i}(k_{p_i}) \right) Q_{p_{k+1}} = 0,
\]
for the projection $Q_{p_{k+1}} \colon \F X \to X_{p_{k+1}}$.
Minimality of $p_{k+1}$ also gives that $\ol{B}_{F_{k}}$ is an ideal in $\ol{B}_{F_{k+1}}$, and we thus derive the following short exact sequence
\[
\xymatrix{
0 \ar[r] & \ol{B}_{F_{k}} \ar[r] & \ol{B}_{F_{k+1}} \ar[r] & \ol{\psi}_{p_{k+1}}(\K X_{p_{k+1}}) \ar[r] & 0 .
}
\]
Since $\ol{B}_{F_k}$ is nuclear by the inductive hypothesis and $\ol{\psi}_{p_{k+1}}(\K X_{p_{k+1}})$ is nuclear by the base case we have that $\ol{B}_{F_{k+1}}$ is nuclear.
Inducing on $k$ gives that $\ol{B}_{F} = \ol{B}_{F_n}$ is nuclear.
\end{proof}

In the amenably controlled case, and by combining with Theorem \ref{T:ntx un}, we can deduce nuclearity/exactness of $\N\T(X)$ from nuclearity/exactness of $A$, and conversely.
The exactness equivalence passes to $A \times_{X} P$, however this fails for nuclearity even for $P = \bZ_+$ due to a counterexample of Ozawa in \cite{Kat04}.
In \cite{Kak19} it is shown that $A \times_{X} P$ is nuclear if and only if the embedding $A \hookrightarrow A \times_{X} P$ is nuclear when $(G,P)$ is a quasi-lattice controlled by $(\bZ^n, \bZ_+^n)$ that satisfies a minimality condition.
In fact this holds for any quotient in-between the Toeplitz-Nica-Pimsner and the covariance algebra.
Here we generalize to controlled maps by amenable weak right LCM-inclusions.
Recall that in the amenably controlled case the reduced C*-algebras are universal.

%%%%%%%%%%%%%%%%%%%%%%%%%%%%%%%%%
\begin{theorem}\label{T:exact}
Let $\vartheta \colon (G,P) \to (\G,\P)$ be a controlled map between weak right LCM-inclusions with $\G$ amenable and let $X$ be a compactly aligned product system over $P$ with coefficients in $A$.
Let $(\pi,t)$ be an equivariant injective Nica-covariant representation of $X$.
Then $A$ is exact if and only if $\ca(\pi,t)$ is exact.
\end{theorem}

\begin{proof}
We are going to introduce new product systems from $X$.
Therefore in order to make a distinction we will write $\ol{E}^X_P$ for the faithful conditional expectation on the Fock C*-algebra $\T_\la(X)$ of $X$.

If $\ca(\pi,t)$ is exact then so is $A$, since exactness passes to C*-subalgebras.
For the converse by Theorem \ref{T:ntx un} we have that $X$ is amenable and thus $\ca(\pi,t)$ is a quotient of $\T_\la(X)$.
Hence it suffices to show that $\T_\la(X)$ is exact.
In view of Theorem \ref{T:nuclear} it suffices to show that $\ol{E}_P^X \otimes_{\max} \id_D$ is faithful on $\T_\la(X) \otimes_{\max} D$ for all C*-algebras $D$.

Towards this end let the product system $Y = \{Y_p\}_{p \in P}$ be defined by
\[
Y_p := \ol{\ol{t}_p(X_p) \odot D}^{\otimes_{\max}} \subseteq \T_\la(X) \otimes_{\max} D.
\]
That $Y$ is a product system follows by that $X$ is so.
Since $X$ is compactly aligned we have that
\[
Y_p Y_p^* Y_q Y_q^* \subseteq \ol{\ol{\psi}_p(\K X_p) \ol{\psi}_q(\K X_q) \odot D}^{\otimes_{\max}} = \ol{\ol{\psi}_w(\K X_w) \odot D}^{\otimes_{\max}} = [Y_w Y_w^*]
\]
for $wP = pP \cap qP$, with the understanding that $Y_p Y_p^* Y_q Y_q^* = (0)$ when $p$ and $q$ have no common right common multiple.
Thus by Proposition \ref{P:com al} we get that $Y$ is a compactly aligned product system over $P$ with coefficients in $A$.

Again by Theorem \ref{T:ntx un} we have that $Y$ is amenable.
Our goal is to show that the identity representation on $Y$ is faithful on $\N\T(Y) \simeq \T_\la(Y)$, and thus we have that
\[
\N\T(Y) \simeq \T_\la(Y) \simeq \T_\la(X) \otimes_{\max} D.
\]
We then derive that the faithful conditional expectation $\ol{E}^Y_P$ on $\T_\la(Y)$ coincides with $\ol{E}^X_P \otimes_{\max} \id_D$ and the proof will be completed.
We will invoke Theorem \ref{T:ntx un}.

First we see that the identity representation is $\G$-equivariant.
Indeed we have that $(\ol{\pi}, \ol{t})$ admits a coaction $\ol{\de}_\G$ of $\G$ and thus we have an equivariant $*$-homomorphism
\[
\ol{\de}_\G \otimes_{\max} \id_D \colon \T_\la(X) \otimes_{\max} D \longrightarrow (\T_\la(X) \otimes \ca(\G)) \otimes_{\max} D.
\]
By amenability of $\G$ and associativity of the maximal tensor product we get that
\begin{align*}
(\T_\la(X) \otimes \ca(\G)) \otimes_{\max} D
& \simeq
\T_\la(X) \otimes_{\max} \ca(\G) \otimes_{\max} D \\
& \simeq
(\T_\la(X) \otimes_{\max} D) \otimes_{\max} \ca(\G)
\simeq
(\T_\la(X) \otimes_{\max} D) \otimes \ca(\G)
\end{align*}
and thus we deduce that $\ol{\de}_\G \otimes_{\max} \id_D$ is a coaction of $\G$ on $\T_\la(X) \otimes_{\max} D$. % (and trivially it has a left inverse by using the character of $\ca(\G)$).
By construction $\ol{\de}_\G \otimes_{\max} \id_D$ satisfies the coaction identity with aligned fibers in the sense that 
\[
[\T_\la(X) \otimes_{\max} D]_g = \ol{[\T_\la(X)]_g \odot D}^{\otimes_{\max}}.
\]

Secondly let $F \subseteq P$ be $\vee$-closed finite set and let $k'_p \in \K Y_p$ such that $\sum_{p \in F} \id(k'_p) = 0$.
For every state $\phi \in \S(D)$ we have the completely contractive map
\[
\id \otimes_{\max} \phi \colon [Y_p Y_p^*] \otimes_{\max} D  \longrightarrow \ol{\psi}_p(\K X_p); \ol{\psi}_p(k_p) \otimes d \mapsto \phi(d) \ol{\psi}_p(k_p).
\]
Therefore we derive 
\[
\sum_{p \in F} (\id \otimes_{\max} \phi)(k'_p) = (\id \otimes_{\max} \phi)(\sum_{p \in F} k'_p) = 0.
\]
Note here that this is a relation in $\T_\la(X)$ with every $(\id \otimes_{\max} \phi)(k'_p) \in \ol{\psi}_p(\K X_p)$.
Thus if $p_0$ is a minimal element in $F$ such that $(\id \otimes_{\max} \phi)(k'_{p_0}) \neq 0$ then we get
\[
P_{p_0} (\id \otimes_{\max} \phi)(k'_{p_0}) P_{p_0} = P_{p_0} (\id \otimes_{\max} \phi)(\sum_{p \in F} k'_p) P_{p_0} = 0,
\]
where $P_{p_0} \colon \F X \to X_{p_0}$ is the canonical projection.
However the compression to $X_{p_0}$ is a faithful $*$-representation on $\ol{\psi}_{p_0}(\K X_{p_0})$, and thus we get the contradiction that $(\id \otimes_{\max} \phi)(k'_{p_0}) = 0$.
Continuing inductively we deduce that $(\id \otimes_{\max} \phi)(k'_p) = 0$ for all $p \in F$ (one by one for fixed $\phi$).
As this holds for all $\phi$ and the family $\{\id \otimes_{\max} \phi\}_{\phi \in \S(D)}$ separates $[Y_p Y_p^*] \otimes_{\max} D$ we get that $k'_p = 0$ for all $p \in F$.
Hence the assumptions of Theorem \ref{T:ntx un} hold for $Y$ and the proof is completed.
\end{proof}

%%%%%%%%%%%%%%%%%%%%%%%%%%%%%%%%%
\begin{theorem}\label{T:nuclear 2}
Let $\vartheta \colon (G,P) \to (\G,\P)$ be a controlled map between weak right LCM-inclusions with $\G$ amenable and let $X$ be a compactly aligned product system over $P$ with coefficients in $A$.
Let $(\pi,t)$ be an equivariant injective Nica-covariant representation of $X$.
Then $A \hookrightarrow \ca(\pi,t)$ is nuclear if and only if $\ca(\pi,t)$ is nuclear.
\end{theorem}

\begin{proof}
It is clear that if $\ca(\pi,t)$ is nuclear then $A \hookrightarrow \ca(\pi,t)$ is nuclear.
Let us prove the converse.
By Theorem \ref{T:ntx un} we have that $\T_\la(X) \simeq \N\T(X)$ and so $(\pi,t)$ promotes to a $*$-representation of $\T_\la(X)$.
Due to amenability $\ca(\G) = \ca_\la(\G)$ is nuclear (and so the minimal and the maximal tensor product coincide).
Let $\de \colon \ca(\pi,t) \to \ca(\pi,t) \otimes \ca_\la(\G)$ be the coaction of $\G$ and let $E = (\id \otimes E_\G) \de$ be the faithful conditional expectation induced on $\ca(\pi,t)$ by the faithful conditional expectation $E_\G$ of $\ca_\la(\G)$.
Let $D$ be any C*-algebra.
Associativity of $\otimes_{\max}$ and nuclearity of $\ca_\la(\G)$ yields
\[
D \otimes_{\max} \ca(\pi,t) \otimes_{\max} \ca_\la(\G) \simeq (D \otimes_{\max} \ca(\pi,t) ) \otimes \ca_\la(\G)
\]
and so $\id_D \otimes_{\max} \id \otimes_{\max} E_\G = (\id_D \otimes_{\max} \id) \otimes E_\G$ is faithful on $D \otimes_{\max} \ca(\pi,t) \otimes_{\max} \ca_\la(\G)$.
Hence 
\[
\id_D \otimes_{\max} E := (\id_D \otimes_{\max} \id \otimes_{\max} E_\G) (\id_D \otimes_{\max} \de)
\]
is a faithful conditional expectation of $D \otimes_{\max} \ca(\pi,t)$ on $D \otimes_{\max} B_\P$.
Therefore we have the following commutative diagram
\[
\xymatrix@R=35pt
{
\ca(\pi, t) \otimes_{\max} D \ar[rr] \ar[d] & & \ca(\pi, t) \otimes D \ar[d] \\
B_{\P} \otimes_{\max} D \ar[rr] & & B_{\P} \otimes D
}
\]
where the vertical arrows are faithful conditional expectations.
Hence it suffices to show that if $\pi \colon A \to B_{\P}$ is nuclear then the fixed point algebra $B_{\P}$ is nuclear.
As the latter is an inductive limit, it suffices to show that nuclearity of $\pi$ in $B_{\P}$ induces nuclearity of the embedding $B_{\vartheta^{-1}(F)} \hookrightarrow B_{\vartheta^{-1}(\P)}$ for every finite $\vee$-closed $\F \subseteq \P$.
We will actually show nuclearity of the embedding 
\[
B_{\vartheta^{-1}(\F)} \hookrightarrow B_{\vartheta^{-1}(\F \cdot \P)} \subseteq B_{\vartheta^{-1}(\P)},
\]
where we write
\[
\vartheta^{-1}(\F \cdot \P) = \{p P \mid \vartheta(p) \in \F \}.
\]

First we remark that $B_\F$ contains a c.a.i.\ for $B_{\F \cdot \P}$.
Indeed let $(e_i)$ be a c.a.i.\ for $B_\F$ so that $\lim_i e_i \psi_p(k_p) = \psi_p(k_p)$ for every $p \in \vartheta^{-1}(\F)$.
Consequently $\lim_i e_i t_p(\xi_p) = t_p(\xi_p)$ for every $p \in \vartheta^{-1}(\F)$ and thus
\[
\lim_i e_i t_p(\xi_p) t_r(\xi_r) t_s(\xi_s)^* = t_p(\xi_p) t_r(\xi_r) t_s(\xi_s)^* \foral r, s \in P.
\]
Thus $\lim_i e_i \psi_{p,q}(k_{p,q}) = \psi_{p,q}(k_{p,q})$ for every $p, q \in \vartheta^{-1}(\F \cdot P)$.

Now fix a finite $\vee$-closed $\F$.
By using maximal elements we can write $\F$ in levels, i.e.,
\[
\F =
\{
h_{11}, \cdots, h_{1 n_1}, 
h_{21}, \cdots, h_{2 n_2}, 
\cdots, 
h_{m1}, \cdots, h_{m n_m}
\},
\]
such that every
\[
\F_i := \{h_{i1}, \dots, h_{i n_i}\} \text{ with } i \in \{1, \dots m\},
\]
consists of the maximal elements of $\F \setminus \cup_{j=1}^{i-1} \F_j$ and $\F_1$ consists of the maximal elements of $\F$.

We now proceed by induction.
For the base case let $h \in \P$ and consider the space
\[
Y_h := \sum_{p \in \vartheta^{-1}(h)} t_{p}(X_{p}).
\]
By using condition (A2) of Definition \ref{D:control} we can equip $Y_h$ with the $A$-valued bilinear map defined by
\[
\sca{y_h, y'_h} := y_h^* y'_h \in \pi(A) \foral y_h, y'_h \in Y_h.
\] 
Then each $Y_h$ becomes a C*-correspondence over $A$, since $\pi$ is faithful.  
The embedding $Y_h \hookrightarrow \big[ Y_h B_{\vartheta^{-1}(\P)} \big]$ and nuclearity of $\pi(A) \hookrightarrow B_{\vartheta^{-1}(\P)}$ imply nuclearity of the embedding
\[
B_{\vartheta^{-1}(h)} = [Y_h Y_h^*] \hookrightarrow \big[ Y_h B_{\vartheta^{-1}(\P)} Y_h^* \big] = B_{\vartheta^{-1}(h \P)}, \foral h \in \P,
\]
by \cite[Proposition B.8]{Kat04}.
Maximality of the $h_{1j}$ in $\F$ yields that the $h_{1j} \P$ are minimal in $\{h \P \mid h \in \F\}$ with respect to inclusions.
As $\F$ is $\vee$-closed we have that $h_{1 j} \P \cap h_{1 j'} \P = \mt$ for $j \neq j'$.
Hence the C*-algebras $B_{\vartheta^{-1}(h_{1j})}$ are orthogonal and thus the embedding
\[
B_{\F_1} = \sumoplus_{j=1}^{n_1} B_{\vartheta^{-1}(h_{1j})} \hookrightarrow \sum_{j=1}^{n_1} B_{\vartheta^{-1}(h_{1 j} \P)} \subseteq B_{\vartheta^{-1}(\F_1 \cdot \P)}
\]
is nuclear.
For the inductive hypothesis suppose that we have shown that the embedding $B_{\vartheta^{-1}(\F')} \hookrightarrow B_{\vartheta^{-1}(\F' \cdot \P)}$
is nuclear for
\[
\F' =
\{
h_{11}, \cdots, h_{1 n_1}, 
\cdots  
h_{i1}, \cdots, h_{ij}
\}
\]
for some $j \in \{1, \dots, n_i\}$.
If $j < n_i$ then set $h := h_{i (j+1)}$; if $j = n_i$ then set $h = h_{(i+1) 1}$.
We will show that the embedding
\[
B_{\vartheta^{-1}(\F'')} \hookrightarrow B_{\vartheta^{-1}(\F'' \cdot \P)}
\qfor
\F'' := \F' \cup \{h\}
\]
is nuclear.
By construction $B_{\vartheta^{-1}(\F')}$ is an ideal in $B_{\vartheta^{-1}(\F'')}$ and $B_{\vartheta^{-1}(F'')} = B_{\vartheta^{-1}(h)} + B_{\vartheta^{-1}(\F')}$; thus
\[
\quo{B_{\vartheta^{-1}(\F'')}}{B_{\vartheta^{-1}(\F')}} \simeq \quo{B_{\vartheta^{-1}(h)}}{B_{\vartheta^{-1}(h)} \cap B_{\vartheta^{-1}(\F')}}.
\]
Likewise $B_{\vartheta^{-1}(\F' \cdot \P)}$ is an ideal of $B_{\vartheta^{-1}(\F'' \cdot \P)}$.
From the base case we have nuclearity of the map
\[
B_{\vartheta^{-1}(h)} \hookrightarrow B_{\vartheta^{-1}(h \P)} \subseteq B_{\vartheta^{-1}(\F'' \cdot \P)}.
\]
By applying Lemma \ref{L:nuc-quot} on the commutative diagram of short exact sequences
\[
\xymatrix{
0 \ar[r] & \quo{B_{\vartheta^{-1}(h)}}{B_{\vartheta^{-1}(\F')}} \ar[r] \ar[d] & B_{\vartheta^{-1}(h)} \ar[r] \ar[d] & \quo{B_{\vartheta^{-1}(h)}}{B_{\vartheta^{-1}(h)} \cap B_{\vartheta^{-1}(\F')}} \ar[r] \ar[d] & 0\\
0 \ar[r] & B_{\vartheta^{-1}(\F' \cdot P)} \ar[r] & B_{\vartheta^{-1}(\F'' \cdot \P)} \ar[r] & \quo{B_{\vartheta^{-1}(\F'' \cdot \P)}}{B_{\vartheta^{-1}(\F' \cdot \P)}} \ar[r] & 0
}
\]
we get that the right vertical arrow is nuclear, i.e., the map
\[
\quo{B_{\vartheta^{-1}(\F'')}}{B_{\vartheta^{-1}(\F')}} \simeq \quo{B_{\vartheta^{-1}(h)}}{B_{\vartheta^{-1}(h)} \cap B_{\vartheta^{-1}(\F')}} \longrightarrow \quo{B_{\vartheta^{-1}(\F'' \cdot \P)}}{B_{\vartheta^{-1}(\F' \cdot \P)}}
\]
is nuclear.
Let $(e_i) \subseteq B_{\vartheta^{-1}(\F')}$ be a c.a.i.\ for $B_{\vartheta^{-1}(\F' \cdot \P)}$, and note that
\[
B_{\vartheta^{-1}(\F'')} \cdot e_i \subseteq B_{\vartheta^{-1}(\F'')} \cdot B_{\vartheta^{-1}(\F')} = B_{\vartheta^{-1}(\F')}.
\]
Using the inductive hypothesis and Lemma \ref{L:nuc-ext} on the commutative diagram of short exact sequences
\[
\xymatrix{
0 \ar[r] & B_{\vartheta^{-1}(\F')} \ar[r] \ar[d] & B_{\vartheta^{-1}(\F'')} \ar[r] \ar[d] & \quo{B_{\vartheta^{-1}(\F'')}}{B_{\vartheta^{-1}(\F')}} \ar[r] \ar[d] & 0\\
0 \ar[r] & B_{\vartheta^{-1}(\F' \cdot \P)} \ar[r] & B_{\vartheta^{-1}(\F'' \cdot \P)} \ar[r] & \quo{B_{\vartheta^{-1}(\F'' \cdot \P)}}{B_{\vartheta^{-1}(\F' \cdot \P)}} \ar[r] & 0
}
\]
we derive that the middle vertical arrow is nuclear, as required.
This concludes the inductive step.
Now by using induction we derive that $B_{\vartheta^{-1}(\F)} \hookrightarrow B_{\vartheta^{-1}(\F \cdot \P)}$ is nuclear, and the proof is complete.
\end{proof}

%%%%%%%%%%%%%%%%%%%%%%%%%%%%%%%%
\section{Saturated controlled maps} \label{S:scm}
%%%%%%%%%%%%%%%%%%%%%%%%%%%%%%%%

%%%%%%%%%%%%%%%%%%%%%%%%%%%%%%%%
\subsection{A product system re-parametrization}
%%%%%%%%%%%%%%%%%%%%%%%%%%%%%%%%

Let $\vartheta \colon (G,P) \to (\G, \P)$ be a controlled map between weak right LCM-inclusions and let $X$ be a compactly aligned product system over $P$ with coefficients in $A$. 
We can then define the C*-correspondence
\[
Y_h := \sumoplus_{p \in \vartheta^{-1}(h)} X_p \foral h \in \P.
\]
One is tempted to consider the family $Y = \{Y_h\}_{h \in \P}$ and associate its C*-algebras with those of $X$.
However it is not clear that $Y$ is in general a product system (let alone compactly aligned).
Nevertheless this happens for controlled maps that satisfy one extra condition.

%%%%%%%%%%%%%%%%%%%%%%%%%%%%%%%%
\begin{definition}
Let $\vartheta \colon (G, P) \rightarrow (\G, \P)$ be a controlled map of weak right LCM-inclusions.
We say that $\vartheta$ is \emph{saturated} if for any $h \in \P$ and $t \in \vartheta^{-1}(h \P)$ there exists an $s \in P$ with $\vartheta(s) \P = h \P$ and $t \in s P$.
\end{definition}

%%%%%%%%%%%%%%%%%%%%%%%%%%%%%%%%
\begin{remark}
In particular, saturated maps satisfy the following property:
\begin{enumerate}
\item[(A3)] If $z \in \P^*$ then there exists an $x \in P^*$ such that $\vartheta(x) = z$.
\end{enumerate}
Indeed, we apply the saturation property for $z \in \P^*$ and $t = e_G \in \vartheta^{-1}(z \P)$ to obtain an $x \in P$ with $e_G \in x P$.
Hence we get that $P = x P$ giving that $x \in P^*$.
\end{remark}

%Let $y \in P$ such that $xy=e$.
%Then $yx = yxy y^{-1} = y e y^{-1} = e$ and thus $x\in P^*$.

The following provides a good supply of saturated controlled maps.
Recall that a pair $(G,P)$ is a \emph{total order} if $G = P^{-1} \cup P$ and $P^{-1} \cap P = \{e_G\}$.
It is clear that total orders, being lattices, form weak right LCM-inclusions. 
%In particular an abelian weak right LCM-inclusion $(G,P)$ is a total order if and only if $-P \cap P = \{0\}$. xxxx

%%%%%%%%%%%%%%%%%%%%%%%%%%%%%%%%
\begin{proposition}\label{P:supply}
Let $(G, P)$ be an abelian total order.
For $n \in \bN \cup \{\infty\}$ consider the free product $(\Asterisk_{i=1}^n G, \Asterisk_{i=1}^n P)$ of $n$ copies of $(G, P)$.
Then the map
\[
\vartheta \colon (\Asterisk_{i=1}^n G, \Asterisk_{i=1}^n P) \longrightarrow (G, P) ; (g_1)_{i_1} (g_2)_{i_2} \dots (g_k)_{i_k} \mapsto g_1 + g_2 + \cdots + g_k
\]
is a saturated controlled map.
\end{proposition}

\begin{proof}
For condition (A1) of Definition \ref{D:control}, if $\ol{p}, \ol{q} \in \Asterisk_{i=1}^n P$ with $\ol{p} (\Asterisk_{i=1}^n P) \bigcap \ol{q} (\Asterisk_{i=1}^n P) \neq \mt$ then the freeness construction implies that either $\ol{p} \leq \ol{q}$ or $\ol{q} \leq \ol{p}$.

For condition (A2) of Definition \ref{D:control} suppose that $\ol{p}, \ol{q}$ have a right LCM and they satisfy $\vartheta(\ol{p}) = \vartheta(\ol{q})$.
Without loss of generality assume that $\ol{r} = \ol{p}^{-1} \ol{q} \in \Asterisk_{i=1}^n P$.
Then $\vartheta(\ol{r}) = 0$.
If $\ol{r} = (r_1)_{i_1} \cdots (r_k)_{i_k}$ then $r_1 + \cdots + r_k = 0$ giving that $r_k \in -P \cap P = \{0\}$.
Inductively we get that $r_1 =\cdots = r_k = 0$ and so $\ol{p} = \ol{q}$.

Next we verify that $\vartheta$ is saturated.
To this end let 
\[
\overline{p} = (p_1)_{i_1} (p_2)_{i_2} \dots (p_k)_{i_k},
\]
and let $h \in \P$ with 
\[
h \leq \vartheta(\overline{p}) =p_1+p_2+\dots +p_k.
\]
Let $\ell \in \{1, \dots, k\}$ be the smallest index so that $h \leq p_1+p_2+\dots +p_{\ell}$.
Set
\[
h' = 
\begin{cases}
h & \text{if } \ell=1, \\
h-(p_1+p_2+\dots +p_{\ell-1}) & \text{otherwise},
\end{cases}
\]
and notice that $h' \in P$ with $h' \leq p_\ell$. 
Let 
\[
\overline{q} = 
\begin{cases}
(h')_{i_1} & \text{if } \ell=1, \\
(p_1)_{i_1} (p_2)_{i_2} \cdots (p_{\ell-1})_{i_{\ell-1}} (h')_{i_\ell} & \text{otherwise.}
\end{cases}
\]
Then $\overline{q} \leq \overline{p}$ and $\vartheta(\overline{q}) = h$, as desired.
\end{proof}

%%%%%%%%%%%%%%%%%%%%%%%%%%%%%%%%
\begin{example}
A second example comes from types of semi-direct products.
Let $(G,P)$ and $(H,S)$ be quasi-lattice ordered groups and let an action $\al \colon H \to \Aut(G)$ such that $\al|_S \colon S \to \Aut(P)$ restricts to automorphisms of $P$.
Then we can form the semi-direct products $G \rtimes_\al H$ and $P \rtimes_\al S$ with respect to the relations $\al_h(g) h = h g$.
The condition on $\al$ makes $P \cdot S$ a subsemigroup of the semi-direct product, and in \cite{Kak19} it is shown that the pair $(G \rtimes_\al H, P \rtimes_\al S)$ is quasi-lattice ordered.
Now suppose that $(G, P)$ admits an abelian controlled map $\vartheta_1$ in $(\G_1, \P_1)$ and $(H, S)$ admits an abelian controlled map $\vartheta_2$ in $(\G_2, P_2)$.
In order for the semi-direct product to inherit the obvious controlled map on $(\G_1 \oplus \G_2, \P_1 \oplus \P_2)$ it is necessary that $\al$ is $\vartheta_1$-invariant in the sense that $\vartheta_1 \al_h = \vartheta_1$ for all $h \in H$.
We can then define the homomorphism
\[
\vartheta \colon (G \rtimes_\al H, P \rtimes_\al S) \longrightarrow (\G_1 \oplus \G_2, \P_1 \oplus \P_2) \textup{ such that } \vartheta(g h) = (\vartheta_1(g), \vartheta_2(h)).
\]
We claim that if $\vartheta_1$ and $\vartheta_2$ are saturated, then so is $\vartheta$.
Suppose that $\vartheta(gh) = (\vartheta_1(g), \vartheta_2(h)) \geq (m, \ell)$.
Then there are $s_1, r_1 \in G$ and $s_2, r_2 \in H$ such that
\[
g = s_1 r_1, \vartheta_1(s_1) = m
\qand
h = s_2 r_2, \vartheta(s_2) = \ell.
\]
It follows that $\vartheta(s_1 s_2) = (m, \ell)$ and $gh = s_1 s_2 \al_{s_2}^{-1}(r_1) r_2$.
\end{example}

The following examples show that surjectivity is not enough to render a controlled map saturated.

%%%%%%%%%%%%%%%%%%%%%%%%%%%%%%%%
\begin{example}
Take the free quasi-lattice on two symbols $a,b$ and take $\vartheta$ be its abelianization map.
Then for $ab$ and $(0,1) \in \bZ^2$ we have that $\vartheta(ab) = (1,1) \geq (0,1)$.
However $\{b\} = \vartheta^{-1}((0,1))$ and $ab \not\geq b$.
(Although, Proposition \ref{P:supply} induces a saturated map on free quasi-lattices.)
\end{example}

%%%%%%%%%%%%%%%%%%%%%%%%%%%%%%%%
\begin{example}
Consider the Baumslag-Solitar group $B(3, 3) = \sca{a, b \mid a^3 b = b a^3}$.
Recall that every element $x \in B(3, 3)$ admits a unique normal form
\[
x = a^{p_1} b^{\eps_1} a^{p_2} \cdots a^{p_{k}} b^{\eps_k} a^{p_{k+1}}
\text{ with }
p_1, \dots, p_k \in \{0, 1, 2\}, p_{k+1} \in \bZ, k \in \bZ_+.
\]
Let $B_+(3, 3)$ be its sub-semigroup generated by $a, b$.
It follows that if $x$ is in its normal form as above then
\[
x = a^{p_1} b^{\eps_1} a^{p_2} \cdots a^{p_{k}} b^{\eps_k} a^{p_{k+1}} \in B_+(3, 3)
\qiff
\eps_1, \dots, \eps_{k} = 1, p_{k+1} \geq 0.
\]
By \cite[Theorem 2.11]{Spi12} we have that the pair $(B(3, 3), B_+(3, 3))$ is a quasi-lattice ordered group. 
In \cite{Kak19} it is shown that the abelianization gives a surjective controlled map
\[
\vartheta \colon (B(3, 3), B_+(3, 3)) \longrightarrow (\bZ^2, \bZ_+^2) ; a^{p_1} b a^{p_2} b \cdots a^{p_k} b a^{p_{k+1}} \mapsto (p_1 + \cdots + p_{k+1}, k).
\]
However this map is not saturated. 
Take $t = a^2 b$ and $h = (1,1)$ so that 
\[
\vartheta(t) = (2,1) \in (1,1) + \bZ_+^2.
\]
We have that $\vartheta^{-1}(1,1) = \{ab, ba\}$ and thus these are the only choices for a possible $s$ with $\vartheta(s) = (1, 1)$ and $s \leq t$.
However we see that
\[
(ab)^{-1} t  = b^{-1} a b \notin B_+(3, 3)
\qand
(ba)^{-1} t = a^{-1} b^{-1} a^{2} b \notin B_+(3, 3).
\]
\end{example}

%%%%%%%%%%%%%%%%%%%%%%%%%%%%%%%%
\begin{theorem}\label{T:con ta}
Let $\vartheta \colon (G, P) \to (\G, \P)$ be a saturated controlled map between weak right LCM-inclusions.
Let $X$ be a (resp.\ injective) compactly aligned product system over $P$ with coefficients in $A$ and let
\begin{equation*} \label{eq:Y}
Y_h: = \sumoplus_{p \in \vartheta^{-1}(h)} X_p \qfor h \in \P.
\end{equation*}
Then the collection $Y = \{Y_h\}_{h \in \P}$ is a (resp.\ injective) compactly aligned product system over $\P$ with coefficients in $A$ such that $\T_\la(X)^+ \simeq \T_\la(Y)^+$ with
\[
\T_\la(X) \simeq \T_\la(Y) 
\qand
A \times_{X, \la} P \simeq  A \times_{Y, \la} \P,
\]
by $*$-homomorphisms that preserve the inclusions $X_p \mapsto Y_{\vartheta(p)}$ for all $p \in P$.
These $*$-isomor\-phisms further lift to $*$-isomorphisms
\[
\N\T(X) \simeq \N\T(Y)
\qand
A \times_{X} P \simeq A \times_{Y} \P,
\]
 that preserve the inclusions $X_p \hookrightarrow Y_{\vartheta(p)}$ for all $p \in P$.
\end{theorem}

\begin{proof}
Let $A$ act on both the left and right of each $Y_h$ with $h \in \P$ via the usual multiplication of operators.
By using condition (A2) of Definition \ref{D:control} we can equip $Y_h$ with the $A$-valued bilinear map defined by
\[
\sca{y_h, y'_h} := y_h^* y'_h \in A \subseteq \T_\la(X) \foral y_h, y'_h \in Y_h.
\] 
Then each $Y_h$ becomes a C*-correspondence over $A$.  
Since $\ker \vphi_{Y_h} = \bigcap_{p \in \vartheta^{-1}(h)} \ker \vphi_{X_p}$ we have that every $Y_h$ is injective when every $X_p$ is so.

We now show that $Y:=\{Y_h\}_{h \in \P}$ is a product system.
Since $[Y_h Y_g] \simeq Y_h \otimes_A Y_g$ we have to show that
\[
[Y_h Y_g] = Y_{h g} \foral h, g \in \P.
\]
As $(\ol{\pi}, \ol{t})$ is a Toeplitz representation we have that $Y_h Y_g \subseteq Y_{h g}$ for all $h,g \in \P$.
For the reverse inclusion, let $p \in P$ with $\vartheta(p)=hg$ and we will show that $\ol{t}_p(X_p) \in [Y_h Y_g]$.
Since $\vartheta$ is saturated there are $q, q' \in P$ such that
\[
p = q q' \qand \vartheta(q) \P = h \P.
\]
We can write $\vartheta(q) = h z$ for some $z \in \P^*$ and let $w \in P^*$ with $\vartheta(w) = z$ by condition (A3) of the saturation property.
Since $\vartheta(q) \vartheta(q') = \vartheta(p) = h g$ it follows that $\vartheta(q') = z^{-1} g$.
We thus conclude that
\[
\vartheta(q w^{-1}) = h \qand \vartheta(wq') = g.
\]
Recall that $X_w$ satisfies $[\ol{t}_{w^{-1}}(X_{w^{-1}}) \ol{t}_{w}(X_{w})] = \ol{\pi}(A)$.
By taking elementary vectors we get the required
\begin{align*}
\ol{t}_p(X_p)
& =
[\ol{t}_q(X_q) \ol{t}_{q'}(X_{q'})]
=
[\ol{t}_q(X_q) \ol{\pi}(A) \ol{t}_{q'}(X_{q'})] \\
& =
[\ol{t}_q(X_q) \ol{t}_{w^{-1}}(X_{w^{-1}}) \ol{t}_w(X_w) \ol{t}_{q'}(X_{q'})] \\
& \subseteq
[\ol{t}_{q w^{-1}}(X_{q w^{-1}}) \ol{t}_{w q'}(X_{w q'})]
\subseteq
[Y_{\vartheta(q) \vartheta(w^{-1})} Y_{\vartheta(w) \vartheta(q')}]
=
[Y_h Y_g].
\end{align*}

Next we show that $Y$ is compactly aligned.
Let $h, h' \in \P$ and take $p \in \vartheta^{-1}(h)$ and $q \in \vartheta^{-1}(h')$.
If $h \vee h' = \infty$ then $p \vee q = \infty$  as well for all $p \in \vartheta^{-1}(h)$ and $q \in \vartheta^{-1}(h')$, and so
\[
Y_h^* Y_{h'} = \sum_{p \in \vartheta^{-1}(h), q \in \vartheta^{-1}(h')} \ol{t}_p(X_p)^* \ol{t}_q(X_q) = (0).
\]
On the other hand, if $h \vee h' < \infty$ and $p \vee q < \infty$ for $p \in \vartheta^{-1}(h)$ and $q \in \vartheta^{-1}(h')$, then
\[
\ol{t}_p(X_p)^* \ol{t}_q(X_q) \subseteq [\ol{t}_{p^{-1}w}(X_{p^{-1}w})\ol{t}_{q^{-1}w}(X_{q^{-1}w})^*].
\]
Since $w = px = q y$ we have that $\vartheta(p^{-1}w) = h^{-1} \vartheta(w) = (h')^{-1} \vartheta(w) = \vartheta(q^{-1} w)$ and also $\vartheta(w) = h \vee h'$.
Hence
\[
Y_h^* Y_{h'} = \sum_{p \in \vartheta^{-1}(h), q \in \vartheta^{-1}(h'), p \vee q < \infty} \ol{t}_p(X_p)^* \ol{t}_q(X_q) \subseteq [Y_{h^{-1} (h \vee h')} Y_{h' (h \vee h')}].
\]
Thus Proposition \ref{P:com al} gives that $Y$ is compactly aligned.

By definition we have that $\F X \simeq \F Y$ (by grouping together summands with the same $\vartheta$-image), and therefore we have that $\T_\la(X) \simeq \T_\la(Y)$ and that $\T_\la(X)^+ \simeq \T_\la(Y)^+$.
Notice that these identifications are $\G$-compatible.
By applying Proposition \ref{P:same G} and Theorem \ref{T:co-univ} we then get 
\begin{align*}
A \times_{X, \la} P & \simeq \cenv(\T_\la(X)^+, G, \ol{\de}_G) \simeq \cenv(\T_\la(X)^+, \G, \ol{\de}_\G) \simeq \cenv(\T_\la(Y)^+, \G, \ol{\de}_\G) \simeq A \times_{Y, \la} \P.
\end{align*}

The second part of the proof is treated likewise.
First note that any representation of $X$ lifts to a representation of $Y$ in a unique way, as every fiber of $Y$ is spanned independently by the corresponding fibers of $X$.
Applying similar arguments as above for a representation $(\pi,t)$ in the place of the Fock representation we see that this correspondence preserves Nica-covariant representations.
Hence we get that $\N\T(X) \simeq \N\T(Y)$.

Finally the $*$-isomorphisms $A \times_{X,\la} P \simeq A \times_{Y, \la} \P$ gives an injective map
\[
[A \times_{X} P]_{p q^{-1}} \simeq [A \times_{X,\la} P]_{p q^{-1}} \hookrightarrow [A \times_{Y, \la} \P]_{\vartheta(p q^{-1})} \simeq [A \times_{Y} \P]_{\vartheta(p q^{-1})}.
\]
Therefore we get a commutative diagram
\[
\xymatrix{
\N\T(X) \ar[rr]^{\Phi} \ar[d]^{q_X} & & \N\T(Y) \ar[d]^{q_Y} \\
A \times_{X} P \ar[rr]^{\Psi} & & A \times_{Y} \P
}
\]
where the upper horizontal arrow is a $*$-isomorphism.
Since the ideals of strong covariance relations are induced, it suffices to show that
\[
\ker \Phi q_Y \bigcap [\N\T(X)]_{e_\G} \subseteq \ker q_X.
\]
Equivalently that $\Psi$ is faithful on the $\G$-fixed point algebra defined on $A \times_{X} P$.
However this follows by Corollary \ref{C:fpa con} as $\Psi$ is by definition faithful on $A$.
\end{proof}

Theorem \ref{T:con ta} gives a very clear picture for the covariance algebras of a product system over a free product order of the form $(\Asterisk_{i=1}^n G, \Asterisk_{i=1}^n P)$ for an abelian total order $(G, P)$. 
It is well-known that the Cuntz C*-algebra $\O_n$, $n \in \bN$, can be viewed as either the Nica-Cuntz-Pimsner C*-algebra of the trivial product system over the free semigroup on $n$ generators or as the Cuntz-Pimsner C*-algebra of the C*-correspondence $(\bC^n, \bC)$. 
Our next result generalizes this fact to arbitrary product systems over the free semigroup.

%%%%%%%%%%%%%%%%%%%%%%%%%%%%%%%%
\begin{corollary} \label{C:freesem}
Let $X$ be a compactly aligned product system over the free semigroup $\bF^+_n = \sca{i_1, \dots, i_n}$.
Then $A \times_{X} \bF^+_n \simeq \O_Y$ for the C*-correspondence $Y = \sumoplus\limits_{j=1, \dots, n} X_{i_j}$. 
\end{corollary}

%%%%%%%%%%%%%%%%%%%%%%%%%%%%%%%%
\subsection{Reversible product systems and total orders}
%%%%%%%%%%%%%%%%%%%%%%%%%%%%%%%%

%%%%%%%%%%%%%%%%%%%%%%%%%%%%%%%%
%\begin{remark}
An application of Burns-Hale Theorem \cite{BH72} asserts that $G$ admits a total order if and only if for every non-trivial finitely-generated subgroup $H$ of $G$ there exists a totally ordered $L$ and a non-trivial homomorphism $H \to L$.
If $L = \bZ$ then the group is called \emph{left indicable}.
There are plenty of abelian total orders.
Examples include $\bR^2$ with the lexicographical order and $\bZ^2$ with the semigroup given by the half-plane defined by any line through the origin with irrational slope.
Conrad's Theorem asserts that if $(G,P)$ is a total order and $G$ is Archimedean then $G$ embeds in $\bR$ so that $P$ embeds in $\bR^+$ \cite{Con59}.
Here we say that $G$ is \emph{Archimedean} if whenever $e_G < x < y$, there exists an $n \in \bN$ such that $y  < x^n$.
We refer the reader to \cite{CR16} for an exposition of these results.
%\end{remark}

There are not many ways for a total order to be controlled by an \emph{abelian} total order.

%%%%%%%%%%%%%%%%%%%%%%%%%%%%%%%%
\begin{proposition}
Let $(G,P)$ be a total order and let $\vartheta_{\textup{ab}} \colon (G,P) \to (G_{\textup{ab}}, P_{\textup{ab}})$ be the abelianization map.
Then the following are equivalent:
\begin{enumerate}
\item There is a controlled map $\vartheta \colon (G,P) \to (\G,\P)$ where $(\G,\P)$ is an abelian total order.
\item $\vartheta_{\textup{ab}}^{-1}(0) \cap P = \{e_G\}$.
\item The abelianization map is a controlled map and $(G_{\textup{ab}}, P_{\textup{ab}})$ is a total order.
\end{enumerate}
If any (and thus all) of the above holds then the abelianization is a saturated controlled map.
\end{proposition}

\begin{proof}
If item (i) holds then $\vartheta$ factors through the abelianization.
Since $\vartheta^{-1}(e_\G) \cap P = \{e_G\}$ being a controlled map, then $\vartheta_{\textup{ab}}^{-1}(0) \cap P = \{e_G\}$ as well.

Assume that item (ii) holds and we will show that $(G_{\textup{ab}}, P_{\textup{ab}})$ is a total order.
First we clearly have that
\[
-P_{\textup{ab}} \cup P_{\textup{ab}} = \vartheta_{\textup{ab}}(P^{-1} \cup P) = G_{\textup{ab}}.
\]
Next suppose that $-P_{\textup{ab}} \cap P_{\textup{ab}} \neq \{0\}$ so that there are $h, g \in P_{\textup{ab}}$ with $h+g = 0$.
As the abelianization map is surjective there are $p, q \in P$ with $pq = e_G$ with $\vartheta(p) = h$ and $\vartheta(q) = g$.
As $(G,P)$ is a total order we derive that $p = q = e_G$ and thus $h = g = 0$.
Next we show that $\vartheta_{\textup{ab}}$ satisfies conditions (A1) and (A2) of Definition \ref{D:control}.
Let $p,q \in P$.
Then either $p \leq q$ or $q \leq p$ and condition (A1) follows.
For condition (A2) suppose without loss of generality that $p \leq q$ with $\vartheta_{\textup{ab}}(p) = \vartheta_{\textup{ab}}(q)$.
Then $q = ps$ for $s \in P \cap \vartheta_{\textup{ab}}^{-1}(0)$.
Then $s = e_G$ and so $p =q$.

If item (iii) holds then clearly item (i) holds, concluding the equivalences between all items.

For the saturation property let a $t \in P$ and an $h \in P_{\textup{ab}}$ such that $\vartheta_{\textup{ab}}(t) = h + h'$.
Take an $s \in \vartheta_{\textup{ab}}^{-1}(h)$ since the abelianization map is surjective.
Then either $s \leq t$ or $s > t$.
But if $s > t$ then $h = \vartheta_{\textup{ab}}(s) > \vartheta_{\textup{ab}}(t)$ which is a contradiction.
Thus we must have that $s \leq t$.
\end{proof}

%%%%%%%%%%%%%%%%%%%%%%%%%%%%%%%%
\begin{remark}
There are exact total orders for which the abelianization map is not controlled.
An example is given by the Klein bottle group
\[
\bK := \langle x,y \mid x^{-1} y x = y^{-1} \rangle = \langle x,y \mid x = y x y \rangle
\]
with the total order induced by the semigroup $\bK^+$ generated by $x, y$ in $\bK$.
It is not hard to see that $\bK^+$ induces a total order on $\bK$, being left indicable (or since $\bK$ is the extension $\bZ \rtimes \bZ$).
Alternatively one can see that every element in $\bK$ is written (uniquely) in the form $x^m y^n$ for $m,n \in \bZ$ and we take cases:
if $m,n \geq 0$ then $x^m y^n \in \bK$;
if $m \geq 1$ and $n \leq 0$ then have that $x^m y^n = x^{m-1} y^{-n} x \in \bK^+$; 
if $m=0$ and $n \leq 0$ then $x^m y^n = y^n \in (\bK^+)^{-1}$.
By symmetry these cover all cases.
We see that $\vartheta_{\textup{ab}}(yxy) = \vartheta_{\textup{ab}}(x)$ and so $e_{\bK} \neq y^2 \in \bK^+ \cap \vartheta^{-1}_{\textup{ab}}(0)$.
In fact we have that $\bK_{\textup{ab}} = \bZ \times \bZ_2$ and $\bK^+_{\textup{ab}} = \bZ^+ \times \bZ_2$ and thus it does not define a total order as $-\bK_{\textup{ab}}^+ \cap \bK^+_{\textup{ab}} = \bZ_2$.
\end{remark}

%%%%%%%%%%%%%%%%%%%%%%%%%%%%%%%%
\begin{definition} 
Let $(G,P)$ be a total order and let $X$ be a product system over $P$ with coefficients in $A$.
We say that $X$ is a \textit{reversible} product system if every $X_p$ is a Hilbert bimodule in $A \times_{X, \la} P$, i.e., if $A \times_{X, \la} P = \ca(\pi,t)$ then $t_p(X_p) t_p(X_p)^* \subseteq A$ for all $p \in P$.
\end{definition} 

It follows that reversible product systems consist of Hilbert bimodules.
The converse holds also for injective product systems, as in this case every strongly covariant representation is Katsura-covariant fiberwise.

%%%%%%%%%%%%%%%%%%%%%%%%%%%%%%%%
\begin{proposition}\label{P:aut cov} 
Let $(G,P)$ be a total order and let $X$ be a product system over $P$ with coefficients in $A$.
Suppose that every $X_p$ is injective.
If $(\pi,t)$ is a strongly covariant representation of $X$ then $(\pi, t_p)$ is a covariant representation of $X_p$, in the sense of Katsura, for every $p \in P$.

Therefore an injective product system $X$ is reversible if and only if every $X_p$ is a Hilbert bimodule.
\end{proposition}

\begin{proof}
Fix $p \in P$ and $a \in A$ such that $\vphi_p(a) = k_p \in \K X_p$.
In view of strong covariance of Proposition \ref{P:sc lcm} and Katsura covariance we have to show that 
\[
\left[\pi_F(a) + \psi_{p,F}(k_p)\right]_{X_F} = 0 \qfor F = \{e, p\},
\]
where
\[
X_F = \oplus_{r \in P} X_r I_{r^{-1} (r \vee F)}.
\]
Let $r \in P$ with $r = ps$ for some $s \in P$.
Then for every $\xi_r = \xi_p \xi_s \in X_r$ and $b \in I_{r^{-1}(r \vee F)}$ we have that
\[
\pi_F(a) \xi_r b = (\vphi_p(a)\xi_p) \xi_s b = (k_p \xi_p) \xi_s b = \psi_{F,p}(k_p) \xi_r b.
\]
Now suppose that $r < p$.
Then by construction $\psi_{F,p}(k_p) \xi_r b = 0$ and we have to show that $\pi_F(a) \xi_r b = 0$ as well.
To this end it suffices to show that 
\[
I_{r^{-1}(r \vee F)} := I_{r^{-1} K_{\{r, e\}}} \bigcap I_{r^{-1} K_{\{r, p\}}} = (0).
\]
Since $r < p$ we have that $r \notin K_{\{r, p\}} \subseteq p P$ while $p \in K_{\{r, p\}}$.
Therefore $r^{-1} p \neq e_G$ and so
\[
I_{r^{-1} K_{\{r, p\}}} = \bigcap_{t \in K_{\{r,p\}}} \ker \vphi_{r^{-1}t} \subseteq \ker \vphi_{r^{-1}p} = (0),
\]
and the proof is complete.
\end{proof}

In the case of $(G,P) =(\bZ, \bZ_+)$, the following result was established in \cite{Kak13}.

%%%%%%%%%%%%%%%%%%%%%%%%%%%%%%%%
\begin{proposition}\label{P:Dirichlet}
Let $(G,P)$ be a total order and let $X$ be a product system over $P$ with coefficients in $A$.
Then $X$ is a reversible product system if and only if the tensor algebra $\T_\la(X)^+$ is Dirichlet in $A \times_{X, \la} P$.
\end{proposition}

\begin{proof} 
Fix $(\pi, t)$ be a faithful representation of $A \times_{X, \la} P$.
Suppose first that $X$ is a reversible product system so that $t_p(X_p) t_p(X_p)^* \subseteq \pi(A)$ for all $p \in P$.
We will show that
\[
A \times_{X, \la} P = \ol{\spn}\{ t_s(X_s) + t_r(X_r)^* \mid s,r \in P\}.
\]
Let $s, r \in P$. 
If $rs^{-1} \in P$ then we have that
\[
t_s(X_s) t_r(X_r)^* \subseteq \big[ t_s(X_s) t_s(X_s)^* t_{rs^{-1}}(X_{rs^{-1}})^* \big] \subseteq \big[ \pi(A) t_{rs^{-1}}(X_{rs^{-1}})^* \big] = t_{rs^{-1}}(X_{rs^{-1}})^*.
\]
If $sr^{-1} \in P$ then we have that
\[
t_s(X_s) t_r(X_r)^* \subseteq \big[ t_{sr^{-1}}(X_{sr^{-1}}) t_r(X_r) t_r(X_r)^* \big] \subseteq \big[ t_{sr^{-1}}(X_{sr^{-1}}) \pi(A) \big] = t_{sr^{-1}}(X_{sr^{-1}}).
\]
Hence
\[
A \times_{X, \la} P = \ol{\spn} \{t_s(X_s) t_r(X_r)^* \mid s, r \in P\} \subseteq \ol{\spn}\{ t_s(X_s) + t_r(X_r)^* \mid s, r \in P\} \subseteq A \times_{X, \la} P,
\]
and so $\T_\la(X)^+$ is Dirichlet in $A \times_{X, \la} P$.

Conversely, assume that $\T_\la(X)^+$ is Dirichlet in $A \times_{X, \la} P$ and let $E$ be the conditional expectation induced by the coaction of $G$ on $A \times_{X, \la} P$.
Then $E(\T_\la(X)^+) = \pi(A)$ and 
\[
E(A \times_{X, \la} P)= E( \ol{\T_\la(X)^+ +  (\T_\la(X)^+)^*})= \pi(A).
\]
Thus for each $p \in P$ we have that $t_p(X_p) t_p(X_p)^* \subseteq E(A \times_{X, \la} P) = \pi(A)$ as desired.
\end{proof}

The next corollary squares with the fact that Popescu's non-commutative disc algebra is not Dirichlet.
Recall that for abelian coactions the C*-envelope of a cosystem coincides with the usual C*-envelope of the ambient operator algebra.

%%%%%%%%%%%%%%%%%%%%%%%%%%%%%%%%
\begin{corollary}
Let $\vartheta \colon (G,P) \to (\G,\P)$ be a saturated controlled map between weak right LCM-inclusions and suppose that $(\G,\P)$ is an abelian total order.
Let $X$ be an injective product system over $P$ with coefficients in $A$. 
Then $\T_\la(X)^+$ is Dirichlet if and only if every strongly covariant representation $(\pi,t)$ of $X$ satisfies $t_p(X_p) t_q(X_q)^* \subseteq A$ whenever $\vartheta(p) = \vartheta(q)$.
\end{corollary}

\begin{proof}
By Theorem \ref{T:cenv red}, and since the controlling pair is abelian, the C*-envelope of $\T_\la(X)^+$ is $A \times_{X} P$.
For the injective $X$, let $Y$ be the injective product system over $\P$ with coefficients in $A$ constructed in Theorem \ref{T:con ta}.
By construction we see that $Y_h$ with $h \in \P$ is a Hilbert bimodule if and only if $t_p(X_p) t_q(X_q)^* \subseteq A$ for all $p,q \in \vartheta^{-1}(h)$.
By applying Remark \ref{R:cenv ab}, Theorem \ref{T:cenv red}, Theorem \ref{T:con ta}, Proposition \ref{P:aut cov}, and Proposition \ref{P:Dirichlet} we have that the Fock tensor algebra $\T_\la(X)^+$ is Dirichlet in $A \times_{X} P$, if and only if $\T_\la(Y)^+$ is Dirichlet in $A \times_{Y} \P$, if and only if every $Y_h$ with $h \in \P$ is a Hilbert bimodule, if and only if $t_p(X_p) t_q(X_q)^* \subseteq \pi(A)$ whenever $\vartheta(p) = \vartheta(q) = h$ for all $h \in \P$.
\end{proof}

The next theorem shows that, for weak right LCM-inclusions that are controlled by total orders in a saturated way, reversible product systems produce all possible covariance algebras.

%%%%%%%%%%%%%%%%%%%%%%%%%%%%%%%%
\begin{theorem} \label{T:dil rev}
Let $\vartheta \colon (G,P) \to (\G,\P)$ be a saturated controlled map between weak right LCM-inclusions and suppose that $(\G,\P)$ is a total order.
Let $X$ be a (resp.\ injective) product system over $P$ with coefficients in $A$.
Then there exists a (resp.\ injective) reversible product system $Z$ over $\P$ with coefficients in a C*-algebra $B$ such that
\begin{equation}
A \subseteq B \qand X_p \subseteq Z_{\vartheta(p)} \foral p \in P,
\end{equation}
that satisfies
\begin{equation}
A \times_{X} P \simeq B \times_{Z} \P
\qand
A \times_{X, \la} P \simeq B \times_{Z, \la} \P,
\end{equation}
by $*$-homomorphisms that preserve the inclusions $X_p \hookrightarrow Z_{\vartheta(p)}$ for all $p \in P$.
\end{theorem}

\begin{proof}
By Theorem \ref{T:con ta} we can assume that $(G,P) = (\G,\P)$.
Fix $(\pi,t)$ be a faithful representation of $A \times_{X, \la} P$ and let 
\[
B := B_P = \ca( \{ t_s(X_s) t_s(X_s)^*\mid s\in P\} \mbox{ and } Z_p := [t_p(X_{p}) B] \foral p \in P \backslash \{e\}.
\]
The trivial C*-correspondence structure on $A \times_{X, \la} P$ descends to a C*-correspondence structure on each $Z_p$ over $B$.
Notice here that since $(G,P)$ is totally ordered we automatically have that the product system $Z = \{Z_p\}_{p \in P}$ is compactly aligned.
Also $\ca(B,Z) = \ca(\pi,t)$ admits a coaction of $G$ from $A \times_{X, \la} P$.
Hence by Theorem \ref{T:extension} we have that
\[
\overline{\alg}\{B, Z_p \mid p \in P\} \simeq \T_\la(Z)^+.
\]
By construction
\[
A \times_{X, \la} P = \ol{\T_\la(Z)^+ + (\T_\la(Z)^+)^*},
\]
thus the cosystem of $\T_\la(Z)^+$ over $G$ is Dirichlet in a C*-cover.
This gives at the same time that this C*-cover $A \times_{X, \la} P$ is the C*-envelope of the cosystem $\T_\la(Z)^+$ over $G$, and that $Z$ is reversible by Proposition \ref{P:Dirichlet}.
Theorem \ref{T:co-univ} then concludes that
\[
B \times_{Z, \la} P \simeq \cenv(\T_\la(Z)^+, G, \ol{\de}_G) \simeq A \times_{X, \la} P.
\]

For the case of the universal C*-algebras we proceed as in Theorem \ref{T:con ta}.
That is first we notice that the $*$-isomorphism between the reduced C*-algebras implies an embedding of the Fell bundles
\[
[A \times_{X} P]_{p q^{-1}} \simeq [A \times_{X, \la} P]_{p q^{-1}} \hookrightarrow [B \times_{Z, \la} \P]_{p q^{-1}} \simeq [B \times_{Z} P]_{p q^{-1}}
\]
which lifts to a $*$-epimorphism $\Psi \colon A \times_{X} P \to B \times_{Z} P$.
Since $X \subseteq Z$ we also have a $*$-epimorphism at the level of the Nica-Toeplitz-Pimsner algebras and thus the following diagram
\[
\xymatrix{
\N\T(X) \ar[rr]^{\Phi} \ar[d]^{q_X} & & \N\T(Z) \ar[d]^{q_Z} \\
A \times_{X} P \ar[rr]^{\Psi} & & B \times_{Z} P
}
\]
is commutative, and fixes $X$.
Since the ideals of strong covariance relations are induced, it suffices to show that
\[
\ker \Phi q_Z \bigcap [\N\T(X)]_{e} \subseteq \ker q_X.
\]
Equivalently that $\Psi$ is faithful on the $G$-fixed point algebra defined on $A \times_{X} P$, which by definition is $B$.
However this follows by the property of $A \times_{X} P$ as $\Psi|_A$ is by construction faithful.

It is left to show that injective of $X$ implies injective of $Z$.
%Now suppose that $X$ is injective and we will show that so is $Z$.
By Theorem \ref{T:con ta} we can still assume that $(G,P) = (\G,\P)$.
To this end let $p \in P$ and $f \in \ker \vphi_{p}^Z$.
We need to show that $f = 0$.

As $B_{p P}$ is an ideal in $B$ we have that $B = B_{\{s < p\}} + B_{p P}$, and let $f_1 \in B_{\{s < p\}}$ and $f_2 \in B_{p P}$ such that $f = f_1 + f_2$.
Let $(e_i)$ be a c.a.i.\ of $\psi_p(\K X_p)$ so that
\[
0 = f e_i = f_1 e_i + f_2 e_i.
\]
However $(e_i)$ is also a c.a.i.\ for $B_{p P}$ and so 
\[
\lim_i f_1 e_i = - \lim_\la f_2 e_i = f_2.
\]
By Nica-covariance $f_1 e_i \in B_p$ for all $i$, and so we have that $f_2 \in B_p$.
Thus we can assume without loss of generality that $f \in B_{\{s \leq p\}}$.
As $B_{\{s \leq p\}}$ is the inductive limit of $B_{F}$ for $F = \{p_1 < p_2 < \cdots < p_n = p\}$ we may assume that
\[
f = \sum_{i=1}^n \psi_{p_i}(k_{p_i}) \text{ with } k_{p_i} \in \K X_{p_i} \text{ and } p_1 < p_2 < \cdots < p_n = p.
\]
Recall the representation $(\pi_F, t_F)$ on $X_F = \oplus_{r \in P} X_r I_{r^{-1}(r \vee F)}^X$ and we will show that
\[
\sum_{i=1}^n \psi_{F, p_i}(k_{p_i})|_{X_F} = 0.
\]
As $(\pi,t)$ is strongly covariant this will give that $f = 0$ by Proposition \ref{P:sc lcm}.
For $r \geq p$ we have that $f \in \ker \vphi_{p}^Z \subseteq \ker \vphi_{r}^Z$, and for every $\eta_r \in X_r I_{r^{-1}(r \vee F)}^X$ we have that $t_r(\eta_r) \in t_r(X_r) \subseteq Z_r$.
Hence
\[
t_r( \sum_{i=1}^n i_{p_i}^{r}(k_{p_i}) (\eta_r) ) = \sum_{i=1}^n \psi_{p_i}(k_{p_i}) t_r(\eta_r) = f t_r(\eta_r) = 0.
\]
As $t$ is isometric we obtain
\begin{equation}\label{eq:aut cov 1}
\sum_{i=1}^n \psi_{F, p_i}(k_{p_i})|_{X_r I_{r^{-1}(r \vee F)}^X} = \sum_{i=1}^n i_{p_i}^{r}(k_{p_i}) = 0, \qfor r \geq p.
\end{equation}
On the other hand for $r < p$ we have that $r^{-1} p \neq e_G$ and so
\[
I_{r^{-1}(r \vee F)}^X \subseteq I_{r^{-1} K_{\{r,p\}}}^X \subseteq \ker\vphi_{{r^{-1}p}}^X = (0).
\]
Hence trivially
\begin{equation}\label{eq:aut cov 2}
\sum_{i=1}^n \psi_{F, p_i}(k_{p_i})|_{X_r I_{r^{-1}(r \vee F)}^X} = 0, \qfor r < p.
\end{equation}
By equations (\ref{eq:aut cov 1}) and (\ref{eq:aut cov 2}) we have that  $\sum_{i=1}^n \psi_{p_i, F}(k_{p_i})|_{X_F} = 0$,
and the proof is complete.
\end{proof}

%%%%%%%%%%%%%%%%%%%%%%%%%%%%%%%%

\end{document}